\documentclass[11pt]{amsart}
\usepackage{}
\usepackage{mathrsfs}
\usepackage{amsfonts}

\usepackage{amssymb,amsthm,amsmath,hyperref,txfonts}

\usepackage{graphicx,float}

\numberwithin{equation}{section}

\newtheorem{thm}{Theorem}[section]
\newtheorem{coro}{Corollary}[section]
\newtheorem{lem}{Lemma}[section]
\newtheorem{rem}{Remark}[section]
\newtheorem{prop}{Proposition}[section]
\newtheorem{defn}{Definition}[section]

\newcommand{\beq}{\begin{eqnarray}}
\newcommand{\eeq}{\end{eqnarray}}
\newcommand{\beqno}{\begin{eqnarray*}}
\newcommand{\eeqno}{\end{eqnarray*}}
\newcommand{\be}{\begin{equation}}
\newcommand{\ee}{\end{equation}}
\newcommand{\beno}{\begin{equation*}}
\newcommand{\eeno}{\end{equation*}}

\newcommand\tl{\tilde}

\newcommand\Dl{\Delta}

\newcommand\N{\mathbb{N}}
\newcommand\nb{\nabla}
\newcommand\nn{\nonumber}
\newcommand{\R}{\mathbb{R}}
\newcommand\Z{\mathbb{Z}}

\newcommand\eps{\epsilon}
\newcommand\ga{\gamma}

\newcommand\fr{\frac}
\newcommand\al{\alpha}
\newcommand{\dv}{\mathrm{div}}

\newcommand\lm{\lambda}
\newcommand\Lm{\Lambda}

\newcommand\Pe{\mathbb{P}}
\newcommand\pr{\partial}

\newcommand{\ddl}{\dot{\Dl}_q}

\allowdisplaybreaks

\begin{document}

\title[Compressible Navier-Stokes equations with a class of large data]{Global solutions to the isentropic compressible Navier-Stokes equations with a class of large initial data}

\author{Daoyuan Fang, Ting Zhang,  Ruizhao Zi*}

\address{Department of Mathematics, Zhejiang University, Hangzhou 310027, China}
\email{dyf@zju.edu.cn}

\address{Department of Mathematics, Zhejiang University, Hangzhou 310027, China}
\email{zhangting79@zju.edu.cn}

\address{*Corresponding author. School of Mathematics and Statistics \& Hubei Key Laboratory of Mathematical Sciences, Central China Normal University, Wuhan 430079, China}
\email{rzz@mail.ccnu.edu.cn}



\begin{abstract}
In this paper, we consider the global well-posedness problem of the isentropic compressible Navier-Stokes equations  in the whole space $\R^N$ with $N\ge2$. In order to better reflect the characteristics of the dispersion equation, we make full use of the role of the frequency on the integrability and regularity of the solution,  and prove that the isentropic compressible Navier-Stokes equations admit global solutions when the initial data are close to a stable equilibrium in the sense of suitable hybrid Besov norm.   As a consequence, the initial velocity with arbitrary $\dot{B}^{\fr{N}{2}-1}_{2,1}$ norm of potential part $\Pe^\bot u_0$  and large highly oscillating  are allowed in our results. The proof relies heavily on the dispersive estimates for  the system of acoustics, and a careful
study of the nonlinear terms.
\end{abstract}
\maketitle

\section{Introduction}
The isentropic compressible Navier-Stokes equations are governed by conservation of mass and conservation of momentum:
\beq\label{CNS}
\begin{cases}
\pr_t\rho+\dv(\rho u)=0,\ \ \ t>0,\ x\in\mathbb{R}^N,\ N\geq2,\\
\rho(\pr_tu+u\cdot\nb u)+\nb P(\rho)-\mu\Dl u-(\lm+\mu)\nb\dv u=0,\\
(\rho,u)\rightarrow (\bar{\rho},0),\ \textrm{ as }|x|\rightarrow\infty,\\
(\rho, u)|_{t=0}=(\rho_0, u_0).
\end{cases}
\eeq
where the unknowns $\rho$ and $u$ are the density and velocity of the fluid, respectively. $P=P(\rho)$ is the pressure, which is a smooth function of $\rho$. The viscous  coefficients $\mu$ and $\lambda$ are  assumed to be constants, satisfying the following physical restrictions:
\be\label{vis-coefficients}
\mu>0, \quad 2\mu+N\lambda\ge0,
\ee
with $N\ge2$ the spacial  dimension. Clearly, \eqref{vis-coefficients} implies $\nu:=\lm+2\mu>0$, which, together with \eqref{vis-coefficients} ensures the ellipticity for the Lam\'e operator $\mu\Dl+(\lm+\mu)\nb\dv$. Moreover, without loss of generality,  we assume that $\bar{\rho}=1$ and
\be\label{a-P}
P'(1)=1.
\ee

There are huge literatures on the well-posedness results of the compressible Navier-Stokes equations.
To the best of our knowledge, the local existence and uniqueness of classical solutions are first established in \cite{Nash62, Serrin59} with $\rho_0$ bounded away from zero.
For the case that the initial density  may vanish in open sets, see \cite{CCK04,SS93}. The global classical solutions were first
obtained by Matsumura and Nishida \cite{MN80} for initial data $(\rho_0, u_0)$ close to a  equilibrium $(\bar{\rho}, 0)$ in $H^3\times H^3$, $\bar{\rho}>0$.
Later,  by exploiting some smoothing effects of the so-called {\em effective viscous flux} $F := (2\mu +\lambda ��)\dv u - P(\rho) + P(\bar{\rho})$, Hoff \cite{Hoff951, Hoff952} constructed the global weak solutions with discontinuous initial data. For
arbitrary initial data and $\bar{\rho}=0$, the breakthrough was made by Lions \cite{Lions98}, where he proved the global existence of weak solutions
provided the specific heat ratio $\gamma$ is appropriately large, for
example, $\gamma\geq3N/(N+2), N=2, 3$. Later, Feireisl,  Novotn\'{y}
and Petzeltov\'{y} \cite{FNP01} improved  Lions's results to the
case $\gamma>\frac{N}{2}$. If the initial data possess
some symmetric properties, Jiang and  Zhang \cite{JZ01, JZ03}
obtained the global weak solutions for any $\gamma>1$. Even in the two dimensional case, the uniqueness of weak
solutions is still an open problem up to now. For the case of {\em small energy}, Huang, Li and Xin \cite{HLX12} recently established the global existence and uniqueness of classical solutions,  which can be regarded as a uniqueness and regularity theory of Lions-Feireisl's weak solutions.

The common point among all these papers above is that they did not use scaling considerations, which can help us to find solution spaces as large as possible.  This approach goes back to the pioneering work by Fujita and Kato \cite{FK64} for the classical incompressible Navier-Stokes equations:
\beq\label{INS}
\begin{cases}
\pr_tv+v\cdot\nb v-\mu\Dl v+\nb\Pi=0, \ \ \ t>0,\ x\in\mathbb{R}^N,\ N\geq2,\\
\dv v=0,\\
v|_{t=0}=v_0.
\end{cases}
\eeq
The classical incompressible Navier-Stokes equations, the system
\eqref{INS}, possesses a structure of scaling invariance. Indeed, if $v$ is a
solution of \eqref{INS} on a time interval $[0,T]$ with initial
data $v_0$, then the vector field $v_\lambda$ defined by
    $$
    v_\lambda(t,x)=\lambda v(\lambda^2t,\lambda x)
    $$
is also a solution of \eqref{INS} on the time interval
$[0,\lambda^{-2}T]$ with the initial data $\lambda v_0(\lambda
x)$. There are many works considering the global well-posedness for the classical incompressible Navier-Stokes equations
\eqref{INS} in the scaling invariant spaces, like \cite{Cannone93,Ca97,FK64,Kato,Koch01} etc.  The importance of these results can be illustrated by the following example \cite{Ca97} in three dimensional case: if $\phi$
is a function in the Schwartz space $\mathcal{S}
(\mathbb{R}^3)$, let us introduce the family of divergence free vector fields
 \be
\label{1.4-in} {\phi_\varepsilon}:=\varepsilon^{\al-1}\sin\left(\fr{x_3}{\varepsilon}\right)(-\pr_2 {\phi}, \pr_1 {\phi}, 0).
\ee
Then, for small  $\varepsilon$, the size of $\|\phi_\varepsilon\|_{BMO^{-1}}$ is  $\varepsilon^{\al}$. The result in \cite{Koch01} implies that the classical incompressible Navier-Stokes system \eqref{INS} is global well-posed with the initial data $v_0=\phi_\varepsilon$ for sufficient small  $\varepsilon$. If Supp$\widehat{\phi}\subset B(0,R)=\{\xi\in \mathbb{R}^3, |\xi|\leq R\}$, then Supp$\widehat{\phi_\varepsilon}\subset B((0,0,\frac{1}{\varepsilon}),R)=\{\xi\in \mathbb{R}^3, |\xi-(0,0,\frac{1}{\varepsilon})|\leq R\}$. Thus, such class of the initial data $v_0=\phi_\varepsilon$ has a interesting property that in the frequency space, it almost concentrates on the high frequency part. We would like to remark that due to the parabolic property of the system \eqref{INS}, the high frequency part of the solution can decay very fast.   A natural question which arises is: what will happen when the initial data almost concentrate on the low frequency part?

 Inspired by this  question, let us come back to the  isentropic compressible Navier-Stokes equations \eqref{CNS}. In this case, the first work following the scaling invariant  approach was given by Danchin, see \cite{Danchin00}, who proved  the global well-posedness   of  strong solutions to \eqref{CNS} with  initial data $(\rho_0, u_0)$ close to a stable equilibrium in
\be\label{Ddata}
\left(\dot{B}^{\fr{N}{2}-1}_{2,1}\cap\dot{B}^{\fr{N}{2}}_{2,1}\right)\times\dot{B}^{\fr{N}{2}-1}_{2,1}.
\ee
 In fact, \eqref{CNS} is not really invariant under the transformation
\begin{gather}\label{scaling}
\begin{cases}
(\rho_0, u_0)\rightarrow (\rho_0(l x), l u_0(l x)),\\
(\rho(t,x), u(t,x))\rightarrow (\rho(l^2t,l x), l u(l^2t,l x)), \quad l>0,
\end{cases}
\end{gather}
unless  we neglect the pressure term $P=P(\rho)$. That's why Danchin introduced the hybrid Besov spaces in \cite{Danchin00}. Roughly speaking, by careful analysis of behaviors of the following hyperbolic-parabolic system
\beq\label{linear2}
\begin{cases}
\pr_tb+\Lm d=f,\\
\pr_t d-\Dl d-\Lm b=g, \quad\mathrm{with}\quad \Lm=\sqrt{-\Dl},
\end{cases}
\eeq
 both in  low frequency and high frequency parts, Danchin obtained the $L^2$-decay in time for $\rho-\bar{\rho}$ in a $L^2$ type Besov space, which is the key point to construct global solutions to \eqref{CNS}. There is an interesting question how to obtain the global well-posedness result with the large initial data in the space (\ref{Ddata}).
   Inspired by works about the classical incompressible Navier-Stokes system \cite{Cannone93,Ca97},   with the aid of Green matrix of \eqref{linear2},
 Charve and Danchin \cite{CD10},   Chen, Miao and Zhang \cite{CMZ10} obtained the global well-posedness result in the critical $L^p$
framework respectively, i.e, the high frequency part of the initial data are small in the following Besov space,
    $$
    b_{0H}\in \dot{B}^\frac{N}{p}_{p,1}, \ u_{0H}\in \dot{B}^{\frac{N}{p}-1}_{p,1},\ b_0=\rho_0-1.
    $$
In this paper,
\beno
f_L:=\sum_{q< 1} f_q,\quad\mathrm{and}\quad f_H:=\sum_{q\ge1} f_q,
\eeno
with $f\in \mathcal{S}'$ and $f_q:=\ddl f$.
 Later, Haspot \cite{Ha11} gave a new proof via the so called {\em effective velocity}.
 Similar to the incompressible Navier-Stokes system, the results in \cite{CD10,CMZ10,Ha11} imply that the  isentropic compressible Navier-Stokes system \eqref{CNS}  is global well-posed with the  highly
oscillating  initial velocity $u_0=\phi_\varepsilon$ in (\ref{1.4-in}) for $N=3$, small  $\varepsilon$ and some $\al$. However, in \cite{CD10,CMZ10,Ha11}, the low frequency part of the initial data are small in the following Besov space,
    \begin{equation}\label{L}
    b_{0L},u_{0L}\in \dot{B}^{\frac{N}{2}-1}_{2,1}.
    \end{equation}
    A natural question which arises is: what will happen when the low frequency part of the initial data are large in (\ref{L})?
   Recently, for the large volume viscosity $\lambda$,  Danchin and Mucha \cite{DM16} established the global solutions to the two dimensional compressible Navier-Stokes equations \eqref{CNS} with large initial velocity and almost constant density.

The aim of this paper is to construct global solutions to the isentropic  compressible Navier-Stokes equations \eqref{CNS} when the low frequency part of the initial velocity field is large. For example,  if $N=3$,
for any fixed $\phi\in\mathcal{S}$ with $\hat{\phi}$ supported in a compact set, say, $\mathrm{Supp}\,\hat{\phi}\subset B(0,1)$, the initial data can be chosen as
\be
(\rho_0, u_0):=(1, l^{-\beta}\nb \phi_l+\tl{\phi_\varepsilon}),
\ee
in our result, where
\beno
\phi_l(x):=\phi(lx),
\eeno
and
\be
\tl{\phi_\varepsilon}:=\varepsilon^{\fr{3}{p}-1}\sin\left(\fr{x_3}{\varepsilon}\right)(-\pr_2\tl{\phi}, \pr_1\tl{\phi}, 0), \quad\mathrm{for\ \ some}\quad \tl{\phi}\in\mathcal{S},
\ee
with some $0<l\ll1,  \beta\ge0$, $\varepsilon>0$, and $p>3$. Please refer to Remark
\ref{rem1.2} for more details.

Since \eqref{CNS} is not really invariant under the transformation (\ref{scaling}), one may guess that the Besov space $\dot{B}^{\frac{N}{2}-1}_{2,1}$ is not a good functional space for the low frequency part of the initial data. By virtue of the low frequency embedding
\be
\|\phi_L\|_{\dot{B}^{s_1}_{2,1}}\le C\| \phi_L\|_{\dot{B}^{s_2}_{2,1}}, \quad\mathrm{for\ \ all}\quad \phi\in\dot{B}^{s_2}_{2,1}, \textrm{ and } s_1>s_2,
\ee
we should consider a class of the initial data that the low frequency part of the initial data $(b_{0L},\Pe^\bot u_{0L})$ are small in the Besov space $\dot{B}^{\frac{N}{2}-1+\al}_{2,1}$ but large in $\dot{B}^{\frac{N}{2}-1}_{2,1}$.
More precisely, we will prove the global existence and uniqueness of solutions to the isentropic  compressible Navier-Stokes system \eqref{CNS} with initial data $(\rho_0, u_0)$ close to a stable equilibrium (1,0), satisfying $(\rho_0-1,u_0)\in\mathcal{E}_0$ defined by
\be\label{E0}
\mathcal{E}_0:=\left\{(\phi,\varphi)\in\mathcal{S}'_h\times\mathcal{S}'_h: (\phi_L, \Pe^\bot \varphi_L)\in \dot{B}^{\fr{N}{2}-1+\al}_{2,1}, \ \phi_H\in \dot{B}^{\fr{N}{2}}_{2,1},\ \Pe^\bot \varphi_H\in\dot{B}^{\fr{N}{2}-1}_{2,1},\  \Pe \varphi\in\dot{B}^{\fr{N}{p}-1}_{p,1} \right\},
 \ee
with some $\al>0$ and $p>2$.
To simplify the presentation, in the following we denote
\be\label{bu1}
\|(b_0, u_0)\|_{\mathcal{E}_0}:=\|(b_{0L}, \Pe^\bot u_{0L})\|_{\dot{B}^{\fr{N}{2}-1+\al}_{2,1}}+\|b_{0H}\|_{\dot{B}^{\fr{N}{2}}_{2,1}}+\|\Pe^\bot u_{0H}\|_{\dot{B}^{\fr{N}{2}-1}_{2,1}}+\|\Pe u_0\|_{\dot{B}^{\fr{N}{p}-1}_{p,1}}.
\ee

We shall construct solutions $(\rho,u)$ to system \eqref{CNS} with $(\rho-1, u)$ lying in the  spaces below.
\begin{defn}\label{space}
Let $T>0, $ and $N\ge2$.
\begin{itemize}
\item For $p>2$, $\al>0$, denote by $\mathcal{E}^{\fr{N}{2},\al}_{p}(T)$ the space of functions $(b, u)$ such that
\begin{eqnarray*}
&&(b_L,\Pe^\bot u_L)\in \widetilde{C}_T(\dot{B}^{\fr{N}{2}-1+\al}_{2,1})\cap L^1_T(\dot{B}^{\fr{N}{2}+1+\al}_{2,1})\cap \widetilde{L}^{\fr{1}{\al}}_T(\dot{B}^{\fr{N}{p}+2\al-1}_{p,1});\\
&&b_H\in\widetilde{C}_T(\dot{B}^{\fr{N}{2}}_{2,1})\cap L^1_T(\dot{B}^{\fr{N}{2}}_{2,1}),\ \ \Pe^\bot u_H\in\widetilde{C}_T(\dot{B}^{\fr{N}{2}-1}_{2,1})\cap L^1_T(\dot{B}^{\fr{N}{2}+1}_{2,1});\\
&&\Pe u\in \widetilde{C}_T(\dot{B}^{\fr{N}{p}-1}_{p,1})\cap L^1_T(\dot{B}^{\fr{N}{p}+1}_{p,1}).
\end{eqnarray*}
We shall endow the space with the norm:
\beqno
\|(b,u)\|_{\mathcal{E}^{\fr{N}{2}, \al}_{p}(T)}&:=&\|(b_L,\Pe^\bot u_L)\|_{\widetilde{L}^\infty_T(\dot{B}^{\fr{N}{2}-1+\al}_{2,1})\cap{L}^1_T(\dot{B}^{\fr{N}{2}+1+\al}_{2,1})\cap \widetilde{L}^{\fr{1}{\al}}_T(\dot{B}^{\fr{N}{p}+2\al-1}_{p,1})}\\
&&+\|b_H\|_{\widetilde{L}^\infty_T(\dot{B}^{\fr{N}{2}}_{2,1})\cap{L}^1_T(\dot{B}^{\fr{N}{2}}_{2,1})}+\|\Pe^\bot u_H\|_{\widetilde{L}^\infty_T(\dot{B}^{\fr{N}{2}-1}_{2,1})\cap{L}^1_T(\dot{B}^{\fr{N}{2}+1}_{2,1})}
 +\|\Pe u\|_{\widetilde{L}^\infty_T(\dot{B}^{\fr{N}{p}-1}_{p,1})\cap{L}^1_T(\dot{B}^{\fr{N}{p}+1}_{p,1})}.
\eeqno
\item For $p=2$, denote by $\mathcal{E}^{\fr{N}{2}}(T)$ the space of functions $(b, u)$ such that
\beno
(b_L, u)\in \widetilde{C}_T(\dot{B}^{\fr{N}{2}-1}_{2,1})\cap L^1_T(\dot{B}^{\fr{N}{2}+1}_{2,1}),\ \ b_H\in \widetilde{C}_T(\dot{B}^{\fr{N}{2}}_{2,1})\cap L^1_T(\dot{B}^{\fr{N}{2}}_{2,1}).
\eeno
with
\beno
\|(b,u)\|_{\mathcal{E}^{\fr{N}{2}}(T)}:=\|(b_L, u)\|_{\widetilde{L}^\infty_T(\dot{B}^{\fr{N}{2}-1}_{2,1})\cap{L}^1_T(\dot{B}^{\fr{N}{2}+1}_{2,1})}+\|b_H\|_{\widetilde{L}^\infty_T(\dot{B}^{\fr{N}{2}}_{2,1})\cap{L}^1_T(\dot{B}^{\fr{N}{2}}_{2,1})}.
\eeno
Indeed, $\mathcal{E}^{\fr{N}{2}}(T)$ is nothing but the space introduced by Danchin in \cite{Danchin00}.
\end{itemize}
We use the notation $\mathcal{E}^{\fr{N}{2},\al}_p$ ($\mathcal{E}^\fr{N}{2}$) if $T=\infty$, changing $[0, T]$ into $[0,\infty)$ in the definition above.
\end{defn}

The main results are stated as follows.
\begin{thm}\label{thm-p>2}
Let
\beq\label{p1}
\begin{cases}
2< p<4,\quad\quad{if}\quad N=2,\\
2< p\le4,\quad\quad{if}\quad N=3,\\
2< p\le\fr{2N}{N-2},\quad{if}\quad N\ge4,
\end{cases}
\eeq
and
\be\label{al1}
0<\al\le\fr{N-1}{2}\left(\fr12-\fr{1}{p}\right).
\ee
Assume that $(\rho_0, u_0)$ satisfies $(\rho_0-1, u_0)\in \mathcal{E}_0$. Then there exist two constants $c_0$ and $C_0>0$ depending on $N,  \mu$ and $\lambda$,  such that  if
\be\label{initial1}
\|(\rho_0-1, u_0)\|_{\mathcal{E}_0}\le c_0,
\ee
then system \eqref{CNS} admits a global solution $(\rho, u)$ with $(\rho-1, u)\in \mathcal{E}^{\fr{N}{2},\al}_p$, satisfying
\be\label{uniform1}
\|(\rho-1,u)\|_{\mathcal{E}^{\frac{N}{2},\al}_p}\le C_0\|(\rho_0-1, u_0)\|_{\mathcal{E}_0}.
\ee
Furthermore, if
    \begin{equation*}
      N\geq 3,\ \
   \textrm{ or }\ \ \Pe u_0\in \dot{B}^0_{2,1} \textrm{ when }  N=2,
    \end{equation*}
    then the solution is unique.
\end{thm}

\begin{rem}\label{rem1.1}
For initial data $(\rho_0, u_0)$ with $(\rho_0-1, u_0)\in \left(\dot{B}^{\fr{N}{2}-1}_{2,1}\cap\dot{B}^{\fr{N}{2}}_{2,1}\right)\times\dot{B}^{\fr{N}{2}-1}_{2,1}$ with
$$\|(\rho_0-1, u_0)\|_{\left(\dot{B}^{\fr{N}{2}-1}_{2,1}\cap\dot{B}^{\fr{N}{2}}_{2,1}\right)\times\dot{B}^{\fr{N}{2}-1}_{2,1}}:=R,$$
 one easily deduces that
    \begin{eqnarray}\label{1.37}
     \nn\|(b_0, u_0)\|_{\mathcal{E}_0}
        &\leq& C 2^{-Q\al}\left( \|P_{<-Q}b_0\|_{\dot{B}^{\frac{N}{2}-1 }_{2,1}}+\|P_{<-Q}\Pe^\bot u_0\|_{\dot{B}^{\frac{N}{2}-1 }_{2,1}}
        \right)+C\|P_{\geq-Q}b_0\|_{\dot{B}^{\frac{N}{2}-1}_{2,1}\cap{\dot{B}^{\frac{N}{2}}_{2,1}}}\\
       \nn &&+C\|P_{\geq-Q}\Pe^\bot u_0\|_{\dot{B}^{\frac{N}{2}-1}_{2,1}}+\|\Pe u_0\|_{\dot{B}^{\frac{N}{p}-1}_{p,1}}\\
        &\leq&C 2^{-Q\al}R+C\|P_{\geq-Q}b_0\|_{\dot{B}^{\frac{N}{2}-1}_{2,1}\cap{\dot{B}^{\frac{N}{2}}_{2,1}}}
        +C\|P_{\geq-Q}\Pe^\bot u_0\|_{\dot{B}^{\frac{N}{2}-1}_{2,1}}+C\|\Pe u_0\|_{\dot{B}^{\frac{N}{2}-1}_{2,1}},
    \end{eqnarray}
 where $P_{<-Q}$ and $P_{\geq-Q}$ are defined in (\ref{PQ}).
Therefore, if
    \be\label{1.38}
    C\|P_{\geq-Q}b_0\|_{\dot{B}^{\frac{N}{2}-1}_{2,1}\cap{\dot{B}^{\frac{N}{2}}_{2,1}}}+C\|P_{\geq-Q}\Pe^\bot u_0\|_{\dot{B}^{\frac{N}{2}-1}_{2,1}}+C\|\Pe u_0\|_{\dot{B}^{\frac{N}{2}-1}_{2,1}}\leq \frac{c_0}{2},
    \ee
with
\be\label{1.39}
Q:=\left[\fr{1}{\al}\log_2\fr{2CR}{c_0}\right]+1,
\ee
then the initial data $(\rho_0, u_0)$ satisfy the condition \eqref{initial1}.
\end{rem}
From Theorem \ref{thm-p>2} and Remark \ref{rem1.1}, we easily obtain the following theorem.
\begin{thm}\label{thm-p=2}
Let
\be\label{al2}
\begin{cases}
0<\al <\frac{1}{8},\quad\quad{if}\quad N=2,\\
0< \al \leq\frac14,\quad\quad{if}\quad N=3,\\
0< \al\le\fr{ N-1}{2N},\quad{if}\quad N\ge4.
\end{cases}
\ee
There exists a constant $c_1$ depending on $N, \mu$ and $\lambda$,  such that for all $(\rho_0, u_0)$ with $(\rho_0-1, u_0)\in\left(\dot{B}^{\fr{N}{2}-1}_{2,1}\cap\dot{B}^{\fr{N}{2}}_{2,1}\right)\times\dot{B}^{\fr{N}{2}-1}_{2,1}$, and $Q\in \N$,  if
\beq\label{initial3}
\nn&&2^{-Q\al}\left( \|P_{<-Q}(\rho_0-1)\|_{\dot{B}^{\frac{N}{2}-1 }_{2,1}}
+\|P_{<-Q}\Pe^\bot u_0\|_{\dot{B}^{\frac{N}{2}-1 }_{2,1}}
        \right)\\
        &&+\|P_{\ge-Q}(\rho_0-1)\|_{\dot{B}^{\fr{N}{2}-1}_{2,1}\cap\dot{B}^{\fr{N}{2}}_{2,1}}+\|P_{\ge-Q}\Pe^\bot u_0\|_{\dot{B}^{\fr{N}{2}-1}_{2,1}}+\|\Pe u_0\|_{\dot{B}^{\fr{N}{2}-1}_{2,1}}\le c_1,
\eeq
then system \eqref{CNS} admits a unique solution $(\rho, u)$ with $(\rho-1, u)\in \mathcal{E}^{\fr{N}{2}}$.
\end{thm}
\begin{rem}\label{rem1.2}
We give some examples of {\em large} initial data $(\rho_0, u_0)$ with $(\rho_0-1, u_0)$ satisfying \eqref{initial1} and \eqref{initial3}. For the sake of simplicity, we take $\rho_0=1$.  In doing so, we just need to focus  on the initial velocity $u_0$. More precisely, for any fixed $\phi\in\mathcal{S}$ with $\|\nb\phi\|_{\dot{B}^{\fr{N}{2}-1}_{2,1}}=R$ and $\hat{\phi}$ supported in a compact set, say, $\mathrm{Supp}\,\hat{\phi}\subset B(0,1)$, let us denote
\beno
\phi_l(x):=\phi(lx).
\eeno
Then
\begin{itemize}
\item $\|\nb \phi_l\|_{\dot{B}^{\fr{N}{2}-1}_{2,1}}=\|\nb\phi\|_{\dot{B}^{\fr{N}{2}-1}_{2,1}}=R$.
\item  $\widehat{\nb \phi_l(\xi)}=l^{1-N}\widehat{\nb\phi}\left(\fr{\xi}{l}\right)$, and $\mathrm{Supp}\, \widehat{\nb \phi_l}\subset B(0,l)$.
\end{itemize}
Consequently,  for all $\beta\in[0,\al)$, taking $l>0$ and $Q\in\N$ satisfying
\be\label{l}
\begin{cases}
l<\fr{3}{4}2^{-Q},\\[3mm]
2^{-\al Q}l^{-\beta}R\le\fr{c_0}{2},
\end{cases}
\ee
we find that
\beno
\dot{\Dl}_q\nb \phi_l=0, \quad q\ge-Q,
\eeno
and
\be\label{1.32}
\|l^{-\beta}P_{<-Q}(\nb \phi_l)\|_{\dot{B}^{\fr{N}{2}-1+\al}_{2,1}}\le\fr{c_0}{2}.
\ee
{\em Example 1.} Let
\be\label{ex1}
(\rho_0, u_0):=(1, l^{-\beta}\nb \phi_l),
\ee
with $l$ satisfying \eqref{l}. Then from Remark \ref{rem1.1} and \eqref{1.32}, it is not difficult to verify that the data given in \eqref{ex1} apply to Theorems \ref{thm-p>2} and \ref{thm-p=2}. This indicates that our results allow for initial data with large potential part of the initial velocity.

In addition, noticing the smallness restriction in \eqref{1.38}, initial velocity with highly oscillating is also permitted in Theorem \ref{thm-p>2} as a by-product. For instance,  if $N=3$, the incompressible part of the initial velocity can be given as in \cite{Ca97},
\be\label{iu01}
\tl{\phi_\varepsilon}:=\varepsilon^{\fr{3}{p}-1}\sin\left(\fr{x_3}{\varepsilon}\right)(-\pr_2\tl{\phi}, \pr_1\tl{\phi}, 0), \quad{for\ \ some}\quad \tl{\phi}\in\mathcal{S},
\ee
with $\varepsilon>0$, and $p>3$. Combining \eqref{ex1} with \eqref{iu01}, we can give another example. \par
\noindent{\em Example 2.} Let
\be\label{ex2}
(\rho_0, u_0):=(1, l^{-\beta}\nb \phi_l+\tl{\phi_\varepsilon}).
\ee
Then for $N=3$, the data in \eqref{ex2} are applicable to Theorem \ref{thm-p>2}.
\end{rem}

\begin{rem}
Our results can be extended  to the case with the high frequency  part $(b_{0H}, u_{0H})$ of the initial data lying in some $L^p$- type  Besov spaces. We omit the details in this paper to avoid tedious computations.
\end{rem}

\begin{rem}
Taking the anisotropy into consideration as in \cite{CZ07} and \cite{Zhang09},  it is possible to relax the smallness  restriction on the divergence free part $\Pe u_0$ of the initial velocity $u_0$. Please refer to \cite{CG10, CGP11} for a recent panorama.
\end{rem}

\begin{rem}
Our methods can be used to other related models. Similar results for the incompressible viscoelastic fluids will be given in a forthcoming paper.
\end{rem}

It is worth pointing out that we  impose neither any symmetrical structure on the initial data,  nor largeness assumptions on the viscosity coefficients $\mu$ or $\lambda$.  What's more, our results hold for all dimensional $N\ge2$. Different from \cite{Danchin00}, our proof relies not only  on the energy estimates for the hyperbolic-parabolic system \eqref{linear2}, but also on the dispersive properties for the following acoustics system:
\be\label{acoustics}
\begin{cases}
\pr_tb+\fr{\Lm d}{\eps}={\bf f},\\
\pr_td-\fr{\Lm b}{\eps}={\bf g}.
\end{cases}
\ee
This method was used before to study the zero Mach number limit problem of the compressible Navier-Stokes equations \cite{DG99, Danchin02, DH14}. It seems that the combination of the energy estimates and Strichartz estimates has never been used to study the global well-posedness problem of the viscous compressible fluids.   In this paper, we try to apply this idea to the isentropic compressible Navier-Stokes equations.

Let us now explain how to construct our solution  spaces and show the ingredients of the proof. First of all, just as in \cite{Danchin00},  writing $\rho=1+b$, and decomposing $u=\Pe^\bot u+\Pe u$,
 where
 \beno
 \Pe^\bot:=-\nb(-\Dl)^{-1}\dv, \quad\mathrm{and}\quad \Pe:=\mathbb{I}-\Pe^\bot,
 \eeno
we reformulate \eqref{CNS} as follows:
\beq\label{bcu}
\begin{cases}
\pr_t b+{\dv \Pe^\bot u}=-\dv(bu),\\
\pr_t \Pe^\bot u-\nu \Dl\Pe^\bot u+{\nb b}=-\Pe^\bot\left(u\cdot\nb u+K( b){\nb b}+I( b)\mathcal{A}u\right),\\
(b,\Pe^\bot u)|_{t=0}=(b_0,\Pe^\bot u_0),
\end{cases}
\eeq
and
\beq\label{biu}
\begin{cases}
\pr_t \Pe u-\mu \Dl\Pe u=-\Pe\left(u\cdot\nb u+I( b)\mathcal{A}u\right),\\
\Pe u|_{t=0}=\Pe u_0,
\end{cases}
\eeq
where $b_0:=\rho_0-1$, $I(a):=\fr{a}{1+a}$, and $K(a):=\fr{P'(1+ a)}{1+ a}-1$. For the sake of simplicity, $\nu$ is assumed to be 1 throughout this paper. Moreover,
the condition (\ref{a-P})
  ensures that $K(0)=0$. Let us denote $d:=\Lm^{-1}\dv u$. From \eqref{bcu}, one easily deduces that $(b,d)$ solves
\beq\label{bd}
\begin{cases}
\pr_t b+{\Lm d}=-\dv(bu),\\
\pr_t d- \Dl d-{\Lm b}=-\Lm^{-1}\dv\left(u\cdot\nb u+K( b){\nb b}+I( b)\mathcal{A}u\right),\\
(b,d)|_{t=0}=(b_0,d_0).
\end{cases}
\eeq
Since $\mathrm{curl}\Pe^\bot=0$, it is easy to verity that \eqref{bd} is equivalent to \eqref{bcu}. In the following,  we shall use \eqref{bd} to replace \eqref{bcu} and do not
 make a distinction between $\Pe^\bot u$ and $d$ in the absence of confusion.

Next, in order to show our ideas more clearly, we divide Danchin's arguments in \cite{Danchin00} into the following three parts:
 \begin{description}
  \item[(i)]\ \ Global estimates for the  linearized system of \eqref{bd}.
  \item[(ii)]\ Commutator estimates for the convection terms.
  \item[(iii)]Product estimates for  other nonlinear terms.
\end{description}
Combining {\bf(i)} with {\bf(ii)}, Danchin  established the global estimates for the paralinearized system of \eqref{bd} with $(b_0, d_0)\in\left(\dot{B}^{\fr{N}{2}-1}_{2,1}\cap\dot{B}^{\fr{N}{2}}_{2,1}\right)\times\dot{B}^{\fr{N}{2}-1}_{2,1}$, see Proposition 10.23 in \cite{Bahouri-Chemin-Danchin11}, for example. Then substituting the results in {\bf(iii)} into the estimates obtained in {\bf(i)} and {\bf(ii)} yields the global estimates of  $(b, d)$. Our proof follows this line, but aside from part {\bf(i)}, we develop different approaches to deal with parts {\bf(ii)} and {\bf(iii)}. In particular, the dispersive properties of the system of acoustics \eqref{acoustics} is taken into consideration. Indeed, for $(b_0,d_0)$ satisfying
\be\label{extra}
(b_0,d_0)\in\left(\dot{B}^{\fr{N}{2}-1}_{2,1}\cap\dot{B}^{\fr{N}{2}+\al}_{2,1}\right)\times\left(\dot{B}^{\fr{N}{2}-1}_{2,1}\cap\dot{B}^{\fr{N}{2}-1+\al}_{2,1}\right),
\ee
it has been shown in \cite{Danchin02} that some Strichartz norms of $(b,d)$ decay algebraically with respect to the Mach number $\eps$. In our case, $\eps=1$, we can not expect any   decay with respect to the Mach number. Nevertheless, in the low frequency part, we still gain some decay by means of the low frequency embedding:
\beq\label{decay}
\nn\|P_{<-Q}(b, d)\|_{\widetilde{L}^{r}_T(\dot{B}^{\fr{N}{p}-1+\al+\fr{1}{r}}_{p,1})}&\lesssim& \left(\|P_{<-Q}(b_0, d_0)\|_{\dot{B}^{\fr{N}{2}-1+\al}_{2,1}}+\|P_{<-Q}({\bf f, g})\|_{L^1_T(\dot{B}^{\fr{N}{2}-1+\al}_{2,1})}\right)\\
&\lesssim&2^{-\al Q}\left(\|(b_0, d_0)\|_{\dot{B}^{\fr{N}{2}-1}_{2,1}}+\|P_{<-Q}({\bf f, g})\|_{L^1_T(\dot{B}^{\fr{N}{2}-1}_{2,1})}\right),
\eeq
with
\be\label{pr}
\al>0,\quad p\ge2, \quad  \fr{2}{r}\le \min\left\{1, (N-1)\left(\fr12-\fr{1}{p}\right)\right\}, \quad (r, p, N)\neq (2, \infty, 3).
\ee
This is the basic idea underneath our approach, which leads us to believe that it is possible to  construct global solutions to \eqref{CNS} with large potential part $\Pe^\bot u_0$   of initial velocity  in    $\dot{B}^{\fr{N}{2}-1}_{2,1}$.

Motivated by \eqref{decay}, we just impose the extra regularity on the low frequency part of $(b_0, d_0)$. More precisely,
\be\label{bd0}
(b_{0L}, d_{0L})\in \dot{B}^{\fr{N}{2}-1+\al}_{2,1}\times\dot{B}^{\fr{N}{2}-1+\al}_{2,1}, \quad (b_{0H}, d_{0H})\in \dot{B}^{\fr{N}{2}}_{2,1}\times\dot{B}^{\fr{N}{2}-1}_{2,1}.
\ee
 In order to handle parts {\bf(ii)} and {\bf(iii)} under the condition \eqref{bd0}, we need to  compensate the loss of    critical norms of $(b, d)$ in the low frequency part. To this end,  we set
 $$
 r=\fr{1}{\al}
 $$
in \eqref{decay}. In this way, $p=2$ is not permitted in \eqref{pr} any more, otherwise $\al=0$. This explains the condition \eqref{al1} in Theorem \ref{thm-p>2}.

On the other hand, the divergence free part $\Pe u$ of the velocity $u$  satisfies the parabolic system \eqref{biu}, and hence possesses no dispersive property at all. Accordingly,  it seems that it is reasonable to assume
\be\label{hiu0}
\Pe u_0\in\dot{B}^{\fr{N}{2}-1}_{2,1}.
\ee
As a result, by the property of heat equation, the space for $\Pe u$ should be
$$
\widetilde{L}^\infty_T(\dot{B}^{\fr{N}{2}-1}_{2,1})\cap{L}^1_T(\dot{B}^{\fr{N}{2}+1}_{2,1}),
$$
and we have to bound the right hand side of \eqref{biu} in $L^1_T(\dot{B}^{\fr{N}{2}-1}_{2,1})$. Unfortunately,  we do not know how to  control $\|\Pe(\dot{T}_{\Pe u}\nb\Pe^\bot u)\|_{L^1_T(\dot{B}^{\fr{N}{2}-1}_{2,1})}$ since from \eqref{decay} and the property of heat equation, we just have
\be\label{s-cu}
\Pe^\bot u\in\widetilde{L}^\infty_T(\dot{B}^{\fr{N}{2}-1+\al}_{2,1})\cap{L}^1_T(\dot{B}^{\fr{N}{2}+1+\al}_{2,1})\cap{\widetilde{L}^{\fr{1}{\al}}_T(\dot{B}^{\fr{N}{p}-1+2\al}_{p,1})}, \quad\mathrm{with}\quad p>2.
\ee
To overcome this problem, owing to the fact that $\Pe\Pe^\bot=0$, we find that
\be\label{com-Tiucu}
\Pe(\dot{T}_{\Pe u}\cdot\nb\Pe^\bot u)=[\Pe,\dot{T}_{(\Pe u)^k}]\pr_k\Pe^\bot u.
\ee
Then the commutator estimate (Lemma 2.99 in \cite{Bahouri-Chemin-Danchin11}) enables us to bound
\beno
\|[\Pe,\dot{T}_{(\Pe u)^k}]\pr_k\Pe^\bot u\|_{L^1_T(\dot{B}^{\fr{N}{2}-1}_{2,1})}\\
\lesssim\|\nb \Pe u\|_{\widetilde{L}^{\fr{1}{1-\al}}_T(\dot{B}^{\fr{N}{p^*}-2\al}_{p^*,1})}\|\nb\Pe^\bot u\|_{{\widetilde{L}^{\fr{1}{\al}}_T(\dot{B}^{\fr{N}{p}-2+2\al}_{p,1})}},
\eeno
with $p*:=\fr{2p}{p-2}$,
provided
\be
\fr{N}{p^*}-2\al+1\le1, \quad\mathrm{i. e.}\quad \al\ge\fr{N}{2}\left(\fr12-\fr{1}{p}\right),
\ee
which contradicts to \eqref{pr}. The above analysis has proved a blind alley if the assumption on $\Pe u_0$ is given by \eqref{hiu0}. However, if
\be\label{iu0}
\Pe u_0\in\dot{B}^{\fr{N}{p}-1}_{p,1}
\ee
with $p$ the same as in \eqref{decay}, the above method to deal with $\Pe(\dot{T}_{\Pe u}\cdot\nb\Pe^\bot u)$ works since
\be\label{key}
\|[\Pe,\dot{T}_{(\Pe u)^k}]\pr_k\Pe^\bot u\|_{L^1_T(\dot{B}^{\fr{N}{p}-1}_{p,1})}\\
\lesssim\|\nb \Pe u\|_{\widetilde{L}^{\fr{1}{1-\al}}_T(\dot{B}^{-2\al}_{\infty,1})}\|\nb\Pe^\bot u\|_{{\widetilde{L}^{\fr{1}{\al}}_T(\dot{B}^{\fr{N}{p}-2+2\al}_{p,1})}},
\ee
holds for all  $\al>0$. Combining \eqref{bd0} with \eqref{iu0}, the condition on  $(b_0, u_0)$ becomes
\be\label{bu0}
\begin{cases}
(b_{0L}, \Pe^\bot u_{0L})\in\dot{B}^{\fr{N}{2}-1+\al}_{2,1},\\[3mm]
b_{0H}\in\dot{B}^{\fr{N}{2}}_{2,1}, \ \ \Pe^\bot u_{0H}\in\dot{B}^{\fr{N}{2}-1}_{2,1},\\[3mm]
\Pe u_0\in\dot{B}^{\fr{N}{p}-1}_{p,1}, \quad\mathrm{with}\quad p>2.
\end{cases}
\ee
This explains the construction of $\mathcal{E}_0$ in \eqref{E0}.

The rest part of this paper is organized as follows. In Section 2, we introduce the tools ( the Littlewood-Paley decomposition
and paradifferential calculus) and give some product estimates in Besov spaces. In Section 3, we recall some properties of the system of acoustics, transport and heat equations.   Section 4 is devoted to the global a priori estimates of system \eqref{bcu}--\eqref{biu}. The proof of Theorem \ref{thm-p>2} is given in Section 5. In Section 6, we prove Theorem \ref{thm-p=2}. Some nonlinear estimates needed in the proof of Theorems \ref{thm-p>2} and \ref{thm-p=2} are put in the Appendix in Section 7.

\bigbreak\noindent{\bf Notation.}\par
\begin{enumerate}
\item For $f\in \mathcal{S}'$, $Q\in \N$,  denote $f_q:=\ddl f$, and
	\be\label{PQ}
	P_{<-Q}f:=\sum_{q<-Q} f_q,\quad\ P_{\geq-Q}f:=f-P_{<-Q}f=\sum_{q\geq-Q}f_q.
\ee
In particular,
\beno
f_L:=\sum_{q< 1} f_q,\quad\mathrm{and}\quad f_H:=\sum_{q\ge1} f_q.
\eeno

\item Denote $p*:=\fr{2p}{p-2}$, i. e. $\fr{1}{p*}=\fr{1}{2}-\fr{1}{p}$, for $p\ge2$.
\item Throughout the paper, $C$ denotes various ``harmless'' positive constants, and
we sometimes use the notation $A \lesssim B$ as an equivalent to $A \le CB$. The
notation $A \approx B$ means that $A \lesssim B$ and $B \lesssim A$.
\end{enumerate}
\section{The Functional Tool Box}
\noindent The results of the present paper rely on the use of  a
dyadic partition of unity with respect to the Fourier variables, the so-called the
\textit{Littlewood-Paley  decomposition}. Let us briefly explain how
it may be built on $\R^N$, and the readers may see more details
in \cite{Bahouri-Chemin-Danchin11,Ch1}. Let $(\chi, \varphi)$ be a couple of $C^\infty$ functions satisfying
$$\hbox{Supp}\,\chi\subset\left\{|\xi|\leq\frac{4}{3}\right\},
\ \ \ \
\hbox{Supp}\,\varphi\subset\left\{\frac{3}{4}\leq|\xi|\leq\frac{8}{3}\right\},
$$
and
$$\chi(\xi)+\sum_{q\geq0}\varphi(2^{-q}\xi)=1,$$

$$\sum_{q\in \mathbb{Z}}\varphi(2^{-q}\xi)=1, \quad \textrm{for} \quad \xi\neq0.$$
Set $\varphi_q(\xi)=\varphi(2^{-q}\xi),$
$h_q=\mathcal{F}^{-1}(\varphi_q),$ and
$\tilde{h}=\mathcal{F}^{-1}(\chi)$. The dyadic blocks and the low-frequency cutoff operators are defined for all $q\in\mathbb{Z}$ by
$$\dot{\Delta}_{q}u=\varphi(2^{-q}\mathrm{D})u=\int_{\R^N}h_q(y)u(x-y)dy,$$
$$\dot{S}_qu=\chi(2^{-q}\mathrm{D})u=\int_{\R^N}\tl{h}_q(y)u(x-y)dy.$$
Then
\begin{equation}\label{e2.1}
u=\sum_{q\in \mathbb{Z}}\ddl u,
\end{equation}
holds for tempered distributions {\em modulo polynomials}. As working modulo polynomials is not appropriate for nonlinear problems, we
shall restrict our attention to the set $\mathcal {S}'_h$ of tempered distributions $u$ such that
$$
\lim_{q\rightarrow-\infty}\|\dot{S}_qu\|_{L^\infty}=0.
$$
Note that \eqref{e2.1} holds true whenever $u$ is in $\mathcal{S}'_h$ and that one may write
$$
\dot{S}_qu=\sum_{p\leq q-1}\dot{\Dl}_{p}u.
$$
Besides, we would like to mention that the Littlewood-Paley decomposition
has a nice property of quasi-orthogonality:
\begin{equation}\label{e2.2}
\dot{\Delta}_p\dot{\Delta}_qu\equiv 0\ \ \hbox{if}\ \ \ |p-q|\geq 2\ \
\hbox{and}\ \ \dot{\Delta}_p(\dot{S}_{q-1}u\dot{\Delta}_qu)\equiv 0\ \ \hbox{if}\ \ \
|p-q|\geq 5.
\end{equation}
One can now  give the definition of
homogeneous Besov spaces.
\begin{defn}\label{D2.1}
For $s\in\R$, $(p,r)\in[1,\infty]^2$, and
$u\in\mathcal{S}'(\R^N),$ we set
$$\|u\|_{\dot{B}_{p,r}^s}=\left\|2^{ sq}\|\dot{\Delta}_qu\|_{L^p} \right\|_{\ell^r}.$$
We then define the spaces
$\dot{B}_{p,r}^s:=\{u\in\mathcal{S}'_h(\R^N),\
\|u\|_{\dot{B}_{p,r}^s}<\infty\}$.
\end{defn}

The following lemma describes the way derivatives act on spectrally localized functions.
\begin{lem}[Bernstein's inequalities]\label{Bernstein}
Let $k\in\N$ and $0<r<R$. There exists a constant $C$ depending on $r, R$ and $d$ such that for all $(a,b)\in[1,\infty]^2$, we have for all $\lm>0$ and multi-index $\al$
\begin{itemize}
\item If $\mathrm{Supp} \hat{f}\subset B(0,\lm R)$, then $\sup_{\al=k}\|\pr^\al f\|_{L^b}\le C^{k+1}\lm^{k+d(\fr1a-\fr1b)}\|f\|_{L^a}$.
\item If $\mathrm{Supp} \hat{f}\subset \mathcal{C}(0,\lm r, \lm R)$, then $C^{-k-1}\lm^k\|f\|_{L^a}\le\sup_{|\al|=k}\|\pr^\al f\|_{L^a}\le C^{k+1}\lm^k\|f\|_{L^a}$
\end{itemize}
\end{lem}

Let us now state some classical
properties for the Besov spaces.
\begin{prop}\label{prop-classical}
For all $s, s_1, s_2\in\R$, $1\le p, p_1, p_2, r, r_1, r_2\le\infty$, the following properties hold true:

\begin{itemize}
\item If $p_1\leq p_2 $  and $r_1\leq r_2,$ then
$\dot{B}_{p_1,r_1}^{s}\hookrightarrow
\dot{B}_{p_2,r_2}^{s-\frac{N}{p_1}+\frac{N}{p_2}}$.

\item If $s_1\neq s_2$ and $\theta\in(0,1)$,
$\left[\dot{B}_{p,r_1}^{s_1},\dot{B}_{p,r_2}^{s_2}\right]_{(\theta,r)}=\dot{B}_{p,r}^{\theta
s_1+(1-\theta)s_2}$.

\item For any smooth homogeneous of degree $m\in\Z$ function $F$ on $\R^N\backslash\{0\}$, the operator $F(D)$ maps $\dot{B}^{s}_{p,r}$ in $\dot{B}^{s-m}_{p,r}$.
\end{itemize}
\end{prop}

Next we  recall a few nonlinear estimates in Besov spaces which may be
obtained by means of paradifferential calculus. Firstly introduced
 by J. M. Bony in \cite{Bony81}, the paraproduct between $f$
and $g$ is defined by
$$\dot{T}_fg=\sum_{q\in\mathbb{Z}}\dot{S}_{q-1}f\dot{\Delta}_qg,$$
and the remainder is given by
$$\dot{R}(f,g)=\sum_{q\in\Z}\tilde{\dot{\Delta}}_qf{\dot{\Delta}}_qg$$
with
$$\tilde{\dot{\Delta}}_qf=(\dot{\Delta}_{q-1}+\dot{\Delta}_{q}+\dot{\Delta}_{q+1})f.$$
We have the following so-called Bony's decomposition:
 \be\label{Bony-decom}
fg=\dot{T}_fg+\dot{T}_gf+\dot{R}(f,g)=\dot{T}_fg+\dot{T}'_gf,
 \ee
where $\dot{T}'_gf:=\dot{T}_gf+\dot{R}(f,g)$. The paraproduct $\dot{T}$ and the remainder $\dot{R}$ operators satisfy the following
continuous properties.

\begin{prop}[\cite{Bahouri-Chemin-Danchin11}]\label{p-TR}
For all $s\in\mathbb{R}$, $\sigma\ge0$, and $1\leq p, p_1, p_2\leq\infty,$ the
paraproduct $\dot T$ is a bilinear, continuous operator from $\dot{B}_{p_1,1}^{-\sigma}\times \dot{B}_{p_2,1}^s$ to
$\dot{B}_{p,1}^{s-\sigma}$ with $\frac{1}{p}=\frac{1}{p_1}+\frac{1}{p_2}$. The remainder $\dot R$ is bilinear continuous from
$\dot{B}_{p_1, 1}^{s_1}\times \dot{B}_{p_2,1}^{s_2}$ to $
\dot{B}_{p,1}^{s_1+s_2}$ with
$s_1+s_2>0$, and $\frac{1}{p}=\frac{1}{p_1}+\frac{1}{p_2}$.
\end{prop}
In view of \eqref{Bony-decom}, Proposition \ref{p-TR} and Bernstein's inequalities,  one easily deduces the following  product estimates. Please  find the proof in Appendix.
\begin{coro}\label{coro-product}
Let $\rho,p_1,p_2,q_1,q_2\in[1,\infty]$, $\frac{1}{\rho}\leq\frac{1}{p_1}+\frac{1}{p_2}$, $\frac{1}{\rho}\leq\frac{1}{q_1}+\frac{1}{q_2}$, $s_1-\frac{N}{p_1}\leq \min\{0,N(\frac{1}{p_2}-\frac{1}{\rho})\}$, $\sigma_1-\frac{N}{q_1}\leq \min\{0,N(\frac{1}{q_2}-\frac{1}{\rho})\}$,  $s_1+s_2>N\max \{0,\fr{1}{p_1}+\frac{1}{p_2}-1\}$,  $s=s_1+s_2+N(\frac{1}{\rho}-\frac{1}{p_1}-\frac{1}{p_2})=\sigma_1+\sigma_2+N(\frac{1}{\rho}-\frac{1}{q_1}-\frac{1}{q_2})$, then there holds
\be\label{product1}
\|uv\|_{\dot{B}^{s}_{\rho,1}}\leq C\|u\|_{\dot{B}^{s_1}_{p_1,1}}\|v\|_{\dot{B}^{s_2}_{p_2,1}}+C\|v\|_{\dot{B}^{\sigma_1}_{q_1,1}}\|u\|_{\dot{B}^{\sigma_2}_{q_2,1}}.
\ee
In particular,
\be\label{product1-s}
\|uv\|_{\dot{B}^{r_1+r_2-\frac{N}{p}}_{p,1}}\leq C\|u\|_{\dot{B}^{r_1}_{p,1}}\|v\|_{\dot{B}^{r_2}_{p,1}},
\ee
where  $p\in[1,\infty]$, $r_1, r_2\le \fr{N}{p}$ and $r_1+r_2>N\max \{0,\fr{2}{p}-1\}$.
\end{coro}
The following Proposition will be used to prove the uniqueness of solutions obtained in Theorem \ref{thm-p>2} for $N=2$.
\begin{prop}[\cite{Danchin05}]\label{prop-pro}
Let $p\ge2$, $s_1\le\fr{N}{p}, s_2<\fr{N}{p}$, and $s_1+s_2\ge0$, then
\beno
\|uv\|_{\dot{B}^{s_1+s_2-\fr{N}{p}}_{p,\infty}}\le C\|u\|_{\dot{B}^{s_1}_{p,1}}\|v\|_{\dot{B}^{s_2}_{p,\infty}}.
\eeno
\end{prop}
The study of non-stationary PDEs requires spaces of the type
$L^\rho_T(X)=L^\rho(0,T;X)$ for appropriate Banach spaces $X$. In
our case, we expect $X$ to be a  Besov space, so that it
is natural to localize the equations through Littlewood-Paley
decomposition. We then get estimates for each dyadic block and
perform integration in time. But, in doing so, we obtain the bounds
in spaces which are not of the type $L^\rho(0,T;\dot{B}^s_{p,r})$. That
 naturally leads to the following definition introduced by Chemin and Lerner in \cite{CL}.
\begin{defn}\label{defn-chemin-lerne}
For $\rho\in[1,+\infty]$, $s\in\R$, and $T\in(0,+\infty)$, we set
$$\|u\|_{\widetilde{L}^\rho_T(\dot{B}^s_{p,r})}=\left\|2^{qs}
\|\dot{\Delta}_qu(t)\|_{L^\rho_T(L^p)}\right\|_{\ell^r}
$$
and denote by
$\widetilde{L}^\rho_T(\dot{B}^s_{p,r})$ the subset of distributions
$u\in
  \mathcal{D}'((0,T), \mathcal{S}'_{h}(\mathbb{R}^{3}))$ with finite
$\|u\|_{\widetilde{L}^\rho_T(\dot{B}^s_{p,r})}$ norm. When $T=+\infty$, the index $T$ is
omitted. We
further denote $\widetilde{C}_T(\dot{B}^s_{p,r})=C([0,T];\dot{B}^s_{p,r})\cap
\widetilde{L}^\infty_{T}(\dot{B}^s_{p,r}) $.
 \end{defn}
\begin{rem}\label{rem-CM-holder}
All the properties of continuity for the paraproduct, remainder, and product remain true for the Chemin-Lerner spaces. The exponent $\rho$ just has to behave according to
H\"{o}lder's inequality for the time variable.
\end{rem}

\begin{rem}\label{rem-CM-minkowski}
The spaces $\widetilde{L}^\rho_T(\dot{B}^s_{p,r})$ can be linked with the classical space $L^\rho_T(\dot{B}^s_{p,r})$ via the Minkowski inequality:
\beno
\|u\|_{\widetilde{L}^\rho_T(\dot{B}^s_{p,r})}\le\|u\|_{L^\rho_T(\dot{B}^s_{p,r})}\quad \mathrm{if}\quad r\ge\rho,\qquad \|u\|_{\widetilde{L}^\rho_T(\dot{B}^s_{p,r})}\ge\|u\|_{L^\rho_T(\dot{B}^s_{p,r})}\quad \mathrm{if}\quad r\le\rho.
\eeno
\end{rem}

\section{Preliminaries}
In this section, we first recall the estimates for the acoustics system \eqref{acoustics}, which are very useful in the proof of Theorem \ref{thm-p>2}.
\begin{prop}[\cite{Danchin02}]\label{prop-wave}
Let $(b, v)$ be a solution of the following system of acoustics:
 \beq\label{eq_W}
\begin{cases}
\pr_tb+\eps^{-1}\Lm v=\mathbf{f},\\
\pr_tv-\eps^{-1}\Lm b=\mathbf{g},\\
(b,v)|_{t=0}=(b_0, v_0).
\end{cases}
 \eeq
Then, for any $s\in\R$ and $T\in(0,\infty]$, the following estimate holds:
 \be
\|(b,v)\|_{\widetilde{L}^r_T(\dot{B}^{s+N(\frac{1}{p}-\frac12)+\frac{1}{r}}_{p,1})}\leq C\eps^{\frac{1}{r}}\|(b_0, v_0)\|_{\dot{B}^s_{2,1}}+C\eps^{1+\frac{1}{r}-\frac{1}{\bar{r}'}}
\|(\mathbf{f}, \mathbf{g})\|_{\widetilde{L}^{\bar{r}'}_T(\dot{B}^{s+N(\frac{1}{\bar{p}'}-\frac12)+\frac{1}{\bar{r}'}-1}_{\bar{p}',1})},
 \ee
with
 \beqno
&p\geq2, \frac{2}{r}\leq\min(1,\ga(p)), (r, p, N)\neq(2,\infty, 3),&\\
&\bar{p}\geq2, \frac{2}{\bar{r}}\leq\min(1,\ga(\bar{p})), (\bar{r}, \bar{p}, N)\neq(2,\infty,3),
 \eeqno
where $\ga(q):=(N-1)(\frac{1}{2}-\frac{1}{q}), \frac{1}{\bar{p}}+\frac{1}{\bar{p}'}=1$, and $\frac{1}{\bar{r}}+\frac{1}{\bar{r}'}=1$.
\end{prop}
Next, we recall the classical estimates in Besov space for the transport and heat equations (Theorem 3.37, \cite{Bahouri-Chemin-Danchin11}).

\begin{prop}\label{prop3.3}
  Let $\sigma\in (-N\min\{\frac{1}{p},\frac{1}{p'}\},1+\frac{N}{p})$ and  $1\leq p,r\leq +\infty$, or $\sigma=1+\frac{N}{p}$ if $r=1$. Let $v$ be a smooth vector field such that $\nabla v\in L^1_T(\dot{B}^{\frac{N}{p}}_{p,r}\cap L^\infty)$, $f_0\in \dot{B}^{\sigma}_{p,r}$ and $g\in L^1_T(\dot{B}^{\sigma}_{p,r})$. There exists a constant $C$, such that for all solution $f\in L^\infty([0,T];\dot{B}^{\sigma}_{p,r}) $ of the equation
    $$
    \partial_t f+v\cdot\nabla f=g,\ f|_{t=0}=f_0,
    $$
  we have the following a priori estimate
  \begin{equation}
    \|f\|_{\widetilde{L}^\infty_T(\dot{B}^{\sigma}_{p,r})}\leq e^{CV(T)}\left(
    \|f_0\|_{\dot{B}^{\sigma}_{p,r}}+\int^T_0 e^{-CV(t)}\|g(t)\|_{\dot{B}^{\sigma}_{p,r}}dt
    \right),
  \end{equation}
where $V(t)=\int^t_0\|\nabla v(\tau)\|_{\dot{B}^{\frac{N}{p}}_{p,r}\cap L^\infty}d\tau$.
\end{prop}

\begin{prop}\label{prop3.4}
  Let $\sigma\in \mathbb{R}$ and  $1\leq \rho,p,r\leq +\infty$. Assume that $f_0\in \dot{B}^{\sigma}_{p,r}$ and $g\in \widetilde{L}^\rho_T(\dot{B}^{\sigma-2+\frac{2}{\rho}}_{p,r})$. There exists a constant $C$, such that for all solution $f\in L^\infty([0,T];\dot{B}^{\sigma}_{p,r})\cap L^1([0,T];\dot{B}^{\sigma+2}_{p,r}) $ of the equation
    $$
    \partial_t f-\nu\Delta f=g,\ f|_{t=0}=f_0,
    $$
  we have the following a priori estimate, for all $\rho\leq\rho_1\leq +\infty$,
  \begin{equation}
   \nu^{\frac{1}{\rho_1}} \|f\|_{\widetilde{L}^{\rho_1}_T(\dot{B}^{\sigma+\frac{2}{\rho_1}}_{p,r})}\leq C\left(
    \|f_0\|_{\dot{B}^{\sigma}_{p,r}}+\nu^{\frac{1}{\rho}-1}\|g \|_{\widetilde{L}^\rho_T(\dot{B}^{\sigma-2+\frac{2}{\rho}}_{p,r})}
    \right).
  \end{equation}
\end{prop}

\section{A priori estimates}\label{S4}
\noindent Before proceeding  any further, let us denote
\begin{gather*}
X_{L}(T):=\|b_L\|_{\widetilde{L}^\infty_T(\dot{B}^{\frac{N}{2}-1+\al}_{2,1})\cap L^1_T(\dot{B}^{\frac{N}{2}+1+\al}_{2,1})}
+\|\Pe^\bot u_L\|_{\widetilde{L}^\infty_T(\dot{B}^{\frac{N}{2}-1+\al}_{2,1})\cap L^1_T(\dot{B}^{\frac{N}{2}+1+\al}_{2,1})}, \\
 X_{H}(T):=
 \|b_H\|_{\widetilde{L}^\infty_T(\dot{B}^{\frac{N}{2}}_{2,1})\cap L^1_T(\dot{B}^{\frac{N}{2}}_{2,1})} +\|\Pe^\bot u_H\|_{\widetilde{L}^\infty_T(\dot{B}^{\frac{N}{2}-1}_{2,1})\cap L^1_T(\dot{B}^{\frac{N}{2}+1}_{2,1})},\\
Y_\al(T):=\|(b_L,\Pe^\bot u_L)\|_{\widetilde{L}^{\frac{1}{\al}}_T(\dot{B}^{\fr{N}{p}+2\al-1}_{p,1})},\\
W(T):=\|\Pe u\|_{\widetilde{L}^\infty_T(\dot{B}^{\fr{N}{p}-1}_{p,1})}+\|\Pe u\|_{L^1_T(\dot{B}^{\fr{N}{p}+1}_{p,1})},\\
X(T)=X_{L}(T)+X_{H}(T)+Y_\al(T)+W(T),\\
X_{L}^0:=\| b_{0L}\|_{\dot{B}^{\frac{N}{2}-1+\al}_{2,1}}+\|\Pe^\bot u_{0L}\|_{\dot{B}^{\frac{N}{2}-1+\al}_{2,1}}, \quad X_{H}^0:=\| b_{0H}\|_{\dot{B}^{\frac{N}{2}}_{2,1}}+\| \Pe^\bot u_{0H}\|_{\dot{B}^{\frac{N}{2}-1}_{2,1}},
\end{gather*}
and
\beno
W^0:=\|\Pe u_0\|_{\dot{B}^{\fr{N}{p}-1}_{p,1}},\quad X^0:= X^0_L+X^0_H+W^0.
\eeno

\subsection{Nonlinear estimates}
Now we estimate the nonlinear terms one by one as follows.

By virtue of the low frequency embedding
\be\label{lf-embeding1}
\|P_{<1}\phi\|_{\dot{B}^{s_1}_{2,1}}\le C\|P_{<1}\phi\|_{\dot{B}^{s_2}_{2,1}}, \quad\mathrm{for\ \ all}\quad \phi\in\dot{B}^{s_2}_{2,1}, \textrm{ and } s_1>s_2,
\ee
  the   high frequency embedding
\be\label{hf-embedding1}
\|P_{\geq1}\phi\|_{\dot{B}^{s_1}_{2,1}}\le C\|P_{\geq 1}\phi\|_{\dot{B}^{s_2}_{2,1}}, \quad\mathrm{for\ \ all}\quad \phi\in\dot{B}^{s_2}_{2,1}, \textrm{ and } s_1<s_2,
\ee
and Corollary \ref{coro-product}, we can obtain the following lemma, whose proof will be given in Appendix.
\begin{lem}\label{lem4.1}
  Assume $(b,u)\in \mathcal{E}^{\frac{N}{2},\al}_p(T)$ with $(p,\al)$ satisfying \eqref{p1}-\eqref{al1}, then we have
    \begin{equation}\label{5.1}
    \|P_{<1} (b \dv u)\|_{L^1_T(\dot{B}^{\frac{N}{2}-1+\al}_{2,1})}
    \leq C X^2(T),
    \end{equation}
        \begin{equation}\label{5.1-0}
    \|P_{<1}(\dot{T}'_{\nb b}u)\|_{L^1_T(\dot{B}^{\frac{N}{2}-1+\al}_{2,1})}
    \leq C X^2(T).
    \end{equation}
and
    \begin{equation}\label{5.1-1}
    \|P_{<1}(\dot{T}_{u}\nb b)\|_{L^1_T(\dot{B}^{\frac{N}{2}-1+\al}_{2,1})}
    \leq C X^2(T).
    \end{equation}
\end{lem}

Since $\dv(bu)=b\dv u+\dot{T}'_{\nb b}u+\dot{T}_{u}\nb b$, from Lemma \ref{lem4.1}, we easily get the following Corollary, which will be used to  bound $Y_\al(T)$.
\begin{coro}\label{coro1}
Under the conditions in Lemma \ref{lem4.1}, we have
\be\label{5.2}
\|P_{<1} \dv(b  u)\|_{L^1_T(\dot{B}^{\frac{N}{2}-1+\al}_{2,1})}
    \leq C X^2(T).
\ee
\end{coro}

From (\ref{lf-embeding1})-(\ref{hf-embedding1}), Lemma \ref{Bernstein} and Proposition \ref{p-TR},  we can obtain the following lemma, whose proof will be given in Appendix.

\begin{lem}\label{lem4.2}
Under the assumptions in Lemma \ref{lem4.1}, we have
\begin{equation}\label{5.1'}
    \|P_{\ge1} (b \dv u)\|_{L^1_T(\dot{B}^{\frac{N}{2}}_{2,1})}
    \leq C X^2(T),
    \end{equation}
        \begin{equation}\label{5.1-0'}
    \|P_{\ge1}(\dot{T}'_{\nb b}u)\|_{L^1_T(\dot{B}^{\frac{N}{2}}_{2,1})}
    \leq C X^2(T).
    \end{equation}
\end{lem}

From the low frequency embedding \eqref{lf-embeding1}, Lemma \ref{Bernstein},  Proposition \ref{p-TR}, Corollary \ref{coro-product},   Theorem 2.61 in \cite{Bahouri-Chemin-Danchin11}, and the special structure of $ \dv(I(b)\mathcal{A}\Pe u)$, we could get the following lemma, whose proof will be given in Appendix.

\begin{lem}\label{lem4.3}
Under the assumptions in Lemma \ref{lem4.1} and
\beno
\|b\|_{L^\infty_T(L^\infty)}\le\fr12,
\eeno
  we have
\be
\|P_{<1}\left(I(b)\mathcal{A}\Pe^\bot u\right)\|_{L^1_T(\dot{B}^{\fr{N}{2}-1+\al}_{2,1})}+\|P_{\ge1}\left(I(b)\mathcal{A}\Pe^\bot u\right)\|_{L^1_T(\dot{B}^{\fr{N}{2}-1}_{2,1})}\le CX^2(T),
\ee
and
\be\label{4.26}
\|P_{<1}\left(\Lm^{-1}\dv(I(b)\mathcal{A}\Pe u)\right)\|_{L^1_T(\dot{B}^{\fr{N}{2}-1+\al}_{2,1})}+\|P_{\ge1}\left(\Lm^{-1}\dv(I(b)\mathcal{A}\Pe u)\right)\|_{L^1_T(\dot{B}^{\fr{N}{2}-1}_{2,1})}\le CX^2(T).
\ee
\end{lem}

Similar, using Lemma \ref{lem-Kb} in the Appendix, we could get the following lemma, whose proof will be given in Appendix.

\begin{lem}\label{lem4.4}
Under the assumptions in Lemma \ref{lem4.3},
  we have
\be
\|P_{<1}\left(K(b)\nb b\right)\|_{L^1_T(\dot{B}^{\fr{N}{2}-1+\al}_{2,1})}+\|P_{\ge1}\left(K(b)\nb b\right)\|_{L^1_T(\dot{B}^{\fr{N}{2}-1}_{2,1})}\le CX^2(T).
\ee
\end{lem}

In the next two lemmas, we shall estimate the convection term $\Lm^{-1}\dv(u\cdot\nb u)$.
Here, we distinguish the terms with the potential part $\Pe^\bot u$  from the terms with the divergence free part $\Pe  u$.

\begin{lem}\label{lem4.5}
  Under the assumptions in Lemma \ref{lem4.1}, we have
    \begin{equation}
    \|P_{<1} ( (\Pe^\bot u\cdot\nabla)\Pe^\bot u)\|_{L^1_T(\dot{B}^{\frac{N}{2}-1+\al}_{2,1})} +\|P_{\ge1} ( (\Pe^\bot u\cdot\nabla)\Pe^\bot u)\|_{L^1_T(\dot{B}^{\frac{N}{2}-1}_{2,1})}
    \leq C X^2(T),
    \end{equation}
    and
    \begin{equation}\label{4.40}
   \|P_{<1} ( \dot{T}_{\Pe^\bot u}\cdot\nabla d)\|_{L^1_T(\dot{B}^{\frac{N}{2}-1+\al}_{2,1})}+ \|P_{\ge1} ( \dot{T}_{\Pe^\bot u}\cdot\nabla d)\|_{L^1_T(\dot{B}^{\frac{N}{2}-1}_{2,1})}
    \leq C X^2(T).
    \end{equation}

\end{lem}
\begin{proof}
 From Bony's decomposition, the low frequency embedding \eqref{lf-embeding1} and the high frequency embedding \eqref{hf-embedding1}, Proposition \ref{p-TR}, Corollary \ref{coro-product} and Lemma \ref{Bernstein}, we have
    \begin{eqnarray}\label{uu}
      && \|P_{<1} ( (\Pe^\bot u\cdot\nabla)\Pe^\bot u)\|_{L^1_T(\dot{B}^{\frac{N}{2}-1+\al}_{2,1})}+\|P_{\ge1} ( (\Pe^\bot u\cdot\nabla)\Pe^\bot u)\|_{L^1_T(\dot{B}^{\frac{N}{2}-1}_{2,1})}\nonumber\\
        &\leq&C \left( \| \dot{T}_{\Pe^\bot u_L}\cdot\nb \Pe^\bot u_L\|_{L^1_T(\dot{B}^{\frac{N}{2} -1+\al}_{2,1})}
      +\| \dot{T}'_{\nb\Pe^\bot u_L}\Pe^\bot u_L\|_{L^1_T(\dot{B}^{\frac{N}{2}-1 +\al}_{2,1})}
      +\| \dot{T}_{\Pe^\bot u_L}\cdot  \nb \Pe^\bot u_H  \|_{L^1_T(\dot{B}^{\frac{N}{2}-1}_{2,1})}
          \right.\nonumber\\
          &&\nn +\| \dot{T}'_{\nb \Pe^\bot u_H}  \Pe^\bot u_L \|_{L^1_T(\dot{B}^{\frac{N}{2}-1+\al}_{2,1})}+\| \dot{T}'_{\Pe^\bot u_H}\cdot  \nb \Pe^\bot u_L  \|_{L^1_T(\dot{B}^{\frac{N}{2}-1+\al}_{2,1})}\\
          &&\left.+\| \dot{T}_{\nb \Pe^\bot u_L}  \Pe^\bot u_H\|_{L^1_T(\dot{B}^{\frac{N}{2}-1}_{2,1})}+\| \Pe^\bot u_H\cdot\nb\Pe^\bot u_H\|_{L^1_T(\dot{B}^{\frac{N}{2}-1 }_{2,1})}
      \right)\nonumber\\
      &\leq&C \left( \|\Pe^\bot u_L\|_{L^{\frac{1}{\al}}_T(\dot{B}^{2\al-1}_{\infty,1})}
      \|\Pe^\bot u_L\|_{L^{\frac{1}{1-\al}}_T(\dot{B}^{\frac{N}{2}-\al+1}_{2,1})}
      + \|\Pe^\bot u_L\|_{L^{\frac{1}{\al}}_T(\dot{B}^{2\al-1}_{\infty,1})}
      \|\nb\Pe^\bot u_H\|_{L^{\frac{1}{1-\al}}_T(\dot{B}^{\frac{N}{2}-2\al}_{2,1})}
     \right.\nonumber\\
     &&\nn + \|\nb\Pe^\bot u_H\|_{L^{1}_T(\dot{B}^{\frac{N}{2}}_{2,1})}
      \|\Pe^\bot u_L\|_{L^{\infty}_T(\dot{B}^{\frac{N}{2}-1+\al}_{2,1})}+  \|\Pe^\bot u_H\|_{L^{2}_T(\dot{B}^{\frac{N}{2}}_{2,1})}
      \|\nb\Pe^\bot u_L\|_{L^{2}_T(\dot{B}^{\frac{N}{2}-1+\al}_{2,1})} \\
      &&\nn\left.  + \|\nb\Pe^\bot u_L\|_{L^{\frac{1}{\al}}_T(\dot{B}^{2\al-2}_{\infty,1})}
      \|\Pe^\bot u_H\|_{L^{\frac{1}{1-\al}}_T(\dot{B}^{\frac{N}{2}+1-2\al}_{2,1})}   + \|\Pe^\bot u_H\|_{L^{2}_T(\dot{B}^{\frac{N}{2}}_{2,1})}^2
    \right)\nonumber\\
      &\leq&C X^2(T),
    \end{eqnarray}
where we have used the facts \eqref{5.3} and
 \begin{equation}\label{5.5}
    \widetilde{L}^\infty_T(\dot{B}^{\frac{N}{2}-1 }_{2,1})\cap L^1_T(\dot{B}^{\frac{N}{2}+1 }_{2,1})
    \subset \widetilde{L}^{\frac{1}{1-\al}}_T(\dot{B}^{\frac{N}{2}-2\al+1}_{2,1}).
     \end{equation}
     Moreover, $\dot{T}_{\Pe^\bot u}\cdot\nabla d$ can be bounded in a similar way. This completes the proof of Lemma \ref{lem4.5}.
\end{proof}

\begin{lem}\label{lem4.6}
Under the assumptions in Lemma \ref{lem4.1},  we get
\begin{gather}
\|P_{<1}\left(\Lm^{-1}\dv(\Pe u\cdot\nb\Pe u)\right)\|_{L^1_T(\dot{B}^{\fr{N}{2}-1+\al}_{2,1})}+\|P_{\ge1}\left(\Lm^{-1}\dv(\Pe u\cdot\nb\Pe u)\right)\|_{L^1_T(\dot{B}^{\fr{N}{2}-1}_{2,1})}\le CX^2(T),\\
\|P_{<1}\left(\Lm^{-1}\dv(\Pe^\bot u\cdot\nb\Pe u)\right)\|_{L^1_T(\dot{B}^{\fr{N}{2}-1+\al}_{2,1})}+\|P_{\ge1}\left(\Lm^{-1}\dv(\Pe^\bot u\cdot\nb\Pe u)\right)\|_{L^1_T(\dot{B}^{\fr{N}{2}-1}_{2,1})}\le CX^2(T),\\
\|P_{<1}\left(\Lm^{-1}\dv(\Pe u\cdot\nb\Pe^\bot u)\right)\|_{L^1_T(\dot{B}^{\fr{N}{2}-1+\al}_{2,1})}+\|P_{\ge1}\left(\Lm^{-1}\dv(\Pe u\cdot\nb\Pe^\bot u)\right)\|_{L^1_T(\dot{B}^{\fr{N}{2}-1}_{2,1})}\le CX^2(T),
\end{gather}
and
\be
\|P_{<1}\left(\Lm^{-1}\dv(\dot{T}_{\Pe u}\cdot\nb d)\right)\|_{L^1_T(\dot{B}^{\fr{N}{2}-1+\al}_{2,1})}+\|P_{\ge1}\left(\Lm^{-1}\dv(\dot{T}_{\Pe u}\cdot\nb d)\right)\|_{L^1_T(\dot{B}^{\fr{N}{2}-1}_{2,1})}\le CX^2(T).
\ee
\end{lem}
\begin{proof}
From Lemma \ref{Bernstein}, Bony's decomposition, the low frequency embedding \eqref{lf-embeding1}, and Proposition \ref{p-TR}, we have
\beq\label{PuPu2}
\nn&&\|P_{<1}\left(\Lm^{-1}\dv(\Pe u\cdot\nb\Pe u)\right)\|_{L^1_T(\dot{B}^{\fr{N}{2}-1+\al}_{2,1})}+\|P_{\ge1}\left(\Lm^{-1}\dv(\Pe u\cdot\nb\Pe u)\right)\|_{L^1_T(\dot{B}^{\fr{N}{2}-1}_{2,1})}\\
\nn&\le&C\|\dot{T}_{\Pe u}\cdot\nb\Pe u\|_{L^1_T(\dot{B}^{\fr{N}{2}-1}_{2,1})}+C\|\dot{T}_{\nb\Pe u}\Pe u\|_{L^1_T(\dot{B}^{\fr{N}{2}-1}_{2,1})}+C\|\pr_k\dot{R}((\Pe u)^k,\Pe u)\|_{L^1_T(\dot{B}^{\fr{N}{2}-1}_{2,1})}\\
\nn&\le&C\|\Pe u\|_{L^\infty_T(\dot{B}^{\fr{N}{p^*}-1}_{p^*,1})}\|\Pe u\|_{L^1_T(\dot{B}^{\fr{N}{p}+1}_{p,1})}+C\|\nb\Pe u\|_{L^2_T(\dot{B}^{\fr{N}{p^*}-1}_{p^*,1})}\|\Pe u\|_{L^2_T(\dot{B}^{\fr{N}{p}}_{p,1})}\\
\nn&\le&C\|\Pe u\|_{L^\infty_T(\dot{B}^{\fr{N}{p}-1}_{p,1})}\|\Pe u\|_{L^1_T(\dot{B}^{\fr{N}{p}+1}_{p,1})}+C\|\Pe u\|_{L^2_T(\dot{B}^{\fr{N}{p}}_{p,1})}^2\\
&\le&CX^2(T),
\eeq
where we have used the fact $p*\ge p$ in the third inequality of \eqref{PuPu2}. This explains why we need to assume $p\le4$ in \eqref{p1}. Next, using $\dv \Pe u=0$ and the fact $u=u_L+u_H$, we can decompose $\Lm^{-1}\dv(\Pe^\bot u\cdot\nb\Pe u)$ as follows:
\beno
\Lm^{-1}\dv(\Pe^\bot u\cdot\nb\Pe u)=\Lm^{-1}\dot{T}_{\pr_i(\Pe^\bot u_L)^k}\pr_k(\Pe u)^i+\Lm^{-1}\dv(\dot{T}'_{\nb\Pe u}\Pe^\bot u_L)+\Lm^{-1}\dv(\Pe^\bot u_H\cdot\nb\Pe u).
\eeno
Then it is easy to see that
\beq
\nn&&\|P_{<1}\left(\Lm^{-1}\dot{T}_{\pr_i(\Pe^\bot u_L)^k}\pr_k(\Pe u)^i\right)\|_{L^1_T(\dot{B}^{\fr{N}{2}-1+\al}_{2,1})}\\
\nn&&+\|P_{\ge1}\left(\Lm^{-1}\dot{T}_{\pr_i(\Pe^\bot u_L)^k}\pr_k(\Pe u)^i\right)\|_{L^1_T(\dot{B}^{\fr{N}{2}-1}_{2,1})}\\
\nn&\le&C\|\dot{T}_{\pr_i(\Pe^\bot u_L)^k}\pr_k(\Pe u)^i\|_{L^1_T(\dot{B}^{\fr{N}{2}-2+\al}_{2,1})}\\
\nn&\le&C\|{\pr_i(\Pe^\bot u_L)^k}\|_{L^\infty_T(\dot{B}^{\fr{N}{p^*}-2+\al}_{p^*,1})}\|\pr_k(\Pe u)^i\|_{L^1_T(\dot{B}^{\fr{N}{p}}_{p,1})}\\
\nn&\le&C\|{\Pe^\bot u_L}\|_{L^\infty_T(\dot{B}^{\fr{N}{2}-1+\al}_{2,1})}\|\Pe u\|_{L^1_T(\dot{B}^{\fr{N}{p}+1}_{p,1})}\\
&\le&CX^2(T),
\eeq
and
\beq
\nn&&\|P_{<1}\left(\Lm^{-1}\dv(\dot{T}'_{\nb\Pe u}\Pe^\bot u_L)\right)\|_{L^1_T(\dot{B}^{\fr{N}{2}-1+\al}_{2,1})}\\
\nn&&+\|P_{\ge1}\left(\Lm^{-1}\dv(\dot{T}'_{\nb\Pe u}\Pe^\bot u_L)\right)\|_{L^1_T(\dot{B}^{\fr{N}{2}-1}_{2,1})}\\
\nn&\le&C\|\dot{T}'_{\nb\Pe u}\Pe^\bot u_L\|_{L^1_T(\dot{B}^{\fr{N}{2}-1+\al}_{2,1})}\\
\nn&\le&C\|{\nb\Pe u}\|_{L^\infty_T(\dot{B}^{-2}_{\infty,1})}\|\Pe^\bot u_L\|_{L^1_T(\dot{B}^{\fr{N}{2}+1+\al}_{2,1})}\\
\nn&\le&C\|{\Pe u}\|_{L^\infty_T(\dot{B}^{\fr{N}{p}-1}_{p,1})}\|\Pe^\bot u_L\|_{L^1_T(\dot{B}^{\fr{N}{2}+1+\al}_{2,1})}\\
&\le&CX^2(T).
\eeq
Similar to \eqref{PuPu2} and \eqref{4.30}, we have
\beq
\nn&&\|P_{<1}\left(\Lm^{-1}\dv(\Pe^\bot u_H\cdot\nb\Pe u)\right)\|_{L^1_T(\dot{B}^{\fr{N}{2}-1+\al}_{2,1})}+\|P_{\ge1}\left(\Lm^{-1}\dv(\Pe^\bot u_H\cdot\nb\Pe u)\right)\|_{L^1_T(\dot{B}^{\fr{N}{2}-1}_{2,1})}\\
\nn&\le&C\|\Pe^\bot u_H\|_{L^\infty_T(\dot{B}^{\fr{N}{2}-1}_{2,1})}\|\Pe u\|_{L^1_T(\dot{B}^{\fr{N}{p}+1}_{p,1})}+C\|\Pe u\|_{L^2_T(\dot{B}^{\fr{N}{p}}_{p,1})}\|\Pe^\bot u_H\|_{L^2_T(\dot{B}^{\fr{N}{2}}_{2,1})}\\
&\le&CX^2(T).
\eeq
To bound $\Pe u\cdot\nb\Pe^\bot u$, we need to decompose it as follows:
\beno
\Pe u\cdot\nb\Pe^\bot u=\Pe u\cdot\nb\Pe^\bot u_L+\Pe u\cdot\nb\Pe^\bot u_H.
\eeno
Then using the high frequency embedding \eqref{hf-embedding1}, one easily deduces that
\beq
\nn&&\|P_{<1}\left(\Lm^{-1}\dv(\Pe u\cdot\nb\Pe^\bot u_L)\right)\|_{L^1_T(\dot{B}^{\fr{N}{2}-1+\al}_{2,1})}+\|P_{\ge1}\left(\Lm^{-1}\dv(\Pe u\cdot\nb\Pe^\bot u_L)\right)\|_{L^1_T(\dot{B}^{\fr{N}{2}-1}_{2,1})}\\
\nn&\le&C\|\dot{T}_{\Pe u}\cdot\nb\Pe^\bot u_L\|_{L^1_T(\dot{B}^{\fr{N}{2}-1+\al}_{2,1})}+C\|\dot{T}'_{\nb\Pe^\bot u_L}\Pe u\|_{L^1_T(\dot{B}^{\fr{N}{2}-1+\al}_{2,1})}\\
\nn&\le&C\|{\Pe u}\|_{L^\infty_T(\dot{B}^{-1}_{\infty,1})}\|\nb\Pe^\bot u_L\|_{L^1_T(\dot{B}^{\fr{N}{2}+\al}_{2,1})}+C\|{\nb\Pe^\bot u_L}\|_{L^\infty_T(\dot{B}^{\fr{N}{p^*}-2+\al}_{p^*,1})}\|\Pe u\|_{L^1_T(\dot{B}^{\fr{N}{p}+1}_{p,1})}\\
\nn&\le&C\|{\Pe u}\|_{L^\infty_T(\dot{B}^{\fr{N}{p}-1}_{p,1})}\|\Pe^\bot u_L\|_{L^1_T(\dot{B}^{\fr{N}{2}+1+\al}_{2,1})}+C\|{\Pe^\bot u_L}\|_{L^\infty_T(\dot{B}^{\fr{N}{2}-1+\al}_{2,1})}\|\Pe u\|_{L^1_T(\dot{B}^{\fr{N}{p}+1}_{p,1})}\\
&\le&CX^2(T).
\eeq
In the same manner as \eqref{PuPu2}, we are led to
\beq
\nn&&\|P_{<1}\left(\Lm^{-1}\dv(\Pe u\cdot\nb\Pe^\bot u_H)\right)\|_{L^1_T(\dot{B}^{\fr{N}{2}-1+\al}_{2,1})}+\|P_{\ge1}\left(\Lm^{-1}\dv(\Pe u\cdot\nb\Pe^\bot u_H)\right)\|_{L^1_T(\dot{B}^{\fr{N}{2}-1}_{2,1})}\\
\nn&\le&C\|\Pe u\|_{L^\infty_T(\dot{B}^{\fr{N}{p}-1}_{p,1})}\|\Pe^\bot u_H\|_{L^1_T(\dot{B}^{\fr{N}{2}+1}_{2,1})}+C\|\Pe u\|_{L^2_T(\dot{B}^{\fr{N}{p}}_{p,1})}\|\Pe^\bot u_H\|_{L^2_T(\dot{B}^{\fr{N}{2}}_{2,1})}\\
&\le&CX^2(T).
\eeq
Finally, $\dot{T}_{\Pe u}\cdot\nb d$ can be bounded in the same way as  $\Pe u\cdot\nb\Pe^\bot u$. The proof of Lemma \ref{lem4.6} is completed.
\end{proof}
The next two lemmas will be used to bound the nonlinear terms in the equation of the divergence free part $\Pe u$ of the velocity (\ref{biu}).
\begin{lem}\label{lem4.7}
Under the assumptions in Lemma \ref{lem4.1},  we obtain
\be
\|\Pe(u\cdot \nb u)\|_{L^1_T(\dot{B}^{\fr{N}{p}-1}_{p,1})}\le CX^2(T).
\ee
\end{lem}
\begin{proof}
It is not difficult to verify that
\be
\Pe(u\cdot \nb u)=\Pe(\Pe u\cdot \nb\Pe u)+\Pe(\Pe u\cdot \nb\Pe^\bot u)+\Pe(\Pe^\bot u\cdot \nb\Pe u).\label{4.63}
\ee
Then using Lemma \ref{Bernstein} and (\ref{product1-s}) yields
\beq\label{PuPup}
\nn&&\|\Pe u\cdot\nb\Pe u\|_{L^1_T(\dot{B}^{\fr{N}{p}-1}_{p,1})}=\|\mathrm{div}(\Pe u\otimes\Pe u)\|_{L^1_T(\dot{B}^{\fr{N}{p}-1}_{p,1})}\\
\nn&\le&C\|\Pe u\otimes\Pe u\|_{L^1_T(\dot{B}^{\fr{N}{p}}_{p,1})}
\le C\|\Pe u\|_{\widetilde{L}^2_T(\dot{B}^{\fr{N}{p}}_{p,1})}^2\le CX^2(T).
\eeq
Next, in view of \eqref{com-Tiucu}, using Proposition \ref{p-TR}, Lemma 2.99 in \cite{Bahouri-Chemin-Danchin11} and $\dv \Pe u=0$, we find that
\beq\label{PuPbotu}
\nn&&\|\Pe(\Pe u\cdot\nb\Pe^\bot u)\|_{L^1_T(\dot{B}^{\fr{N}{p}-1}_{p,1})}\\
\nn&\le&C\|[\Pe,\dot{T}_{(\Pe u)^k}]\pr_k\Pe^\bot u\|_{L^1_T(\dot{B}^{\fr{N}{p}-1}_{p,1})}+C\|\dot{T}_{\nb\Pe^\bot u}\Pe u\|_{L^1_T(\dot{B}^{\fr{N}{p}-1}_{p,1})}+C\|\dot{R}({\Pe u},\nb\Pe^\bot u)\|_{L^1_T(\dot{B}^{\fr{N}{p}-1}_{p,1})}\\
\nn&\le&C\|\nb\Pe u\|_{L^{\fr{1}{1-\al}}_T(\dot{B}^{-2\al}_{\infty,1})}\|\nb\Pe^\bot u\|_{L^{\fr{1}{\al}}_T(\dot{B}^{\fr{N}{p}-2+2\al}_{p,1})}+C\|\nb\Pe^\bot u\|_{L^{\fr{1}{\al}}_T(\dot{B}^{2\al-2}_{\infty,1})}\|\Pe u\|_{L^{\fr{1}{1-\al}}_T(\dot{B}^{\fr{N}{p}+1-2\al}_{p,1})}\\
\nn&&+C\|\pr_k\dot{R}((\Pe u)^k,\Pe^\bot u)\|_{L^1_T(\dot{B}^{\fr{N}{p}-1}_{{p},1})}\\
\nn&\le&C\|\Pe^\bot u\|_{L^{\fr{1}{\al}}_T(\dot{B}^{\fr{N}{p}+2\al-1}_{p,1})}\|\Pe u\|_{L^{\fr{1}{1-\al}}_T(\dot{B}^{\fr{N}{p}+1-2\al}_{p,1})}\\
&\le&CX^2(T).
\eeq
Finally, the condition \eqref{p1} on $p$ ensures that $p<2N$, then it is easy to see that
\beq\label{PbotuPu}
\nn&&\|\Pe^\bot u\cdot\nb\Pe u\|_{L^1_T(\dot{B}^{\fr{N}{p}-1}_{p,1})}\\
\nn&\le&C\|\dot{T}_{\Pe^\bot u}\nb\Pe u\|_{L^1_T(\dot{B}^{\fr{N}{p}-1}_{p,1})}+C\|\dot{T}_{\nb\Pe u}\Pe^\bot u\|_{L^1_T(\dot{B}^{\fr{N}{p}-1}_{p,1})}+C\|\dot{R}({\Pe^\bot u},\nb\Pe u)\|_{L^1_T(\dot{B}^{\fr{N}{p}-1}_{p,1})}\\
\nn&\le&C\|\Pe^\bot u\|_{L^{\fr{1}{\al}}_T(\dot{B}^{2\al-1}_{\infty,1})}\|\nb\Pe u\|_{L^{\fr{1}{1-\al}}_T(\dot{B}^{\fr{N}{p}-2\al}_{p,1})}+C\|\nb\Pe u\|_{L^{\fr{1}{1-\al}}_T(\dot{B}^{-2\al}_{\infty,1})}\|\Pe^\bot u\|_{L^{\fr{1}{\al}}_T(\dot{B}^{\fr{N}{p}+2\al-1}_{p,1})}\\
\nn&&+C\|\dot{R}(\Pe^\bot u,\nb\Pe u)\|_{L^1_T(\dot{B}^{\fr{2N}{p}-1}_{\fr{p}{2},1})}\\
\nn&\le&C\|\Pe^\bot u\|_{L^{\fr{1}{\al}}_T(\dot{B}^{\fr{N}{p}+2\al-1}_{p,1})}\|\Pe u\|_{L^{\fr{1}{1-\al}}_T(\dot{B}^{\fr{N}{p}+1-2\al}_{p,1})}\\
&\le&CX^2(T).
\eeq
This explains why we need to assume $p<4$ if $N=2$. We complete the proof of Lemma \ref{lem4.7}.
\end{proof}

\begin{lem}\label{lem4.8}
Under the assumptions in Lemma \ref{lem4.4},
 we have
\be
\|I(b)\mathcal{A}u\|_{L^1_T(\dot{B}^{\fr{N}{p}-1}_{p,1})}\le CX^2(T).
\ee
\end{lem}
\begin{proof}
Let us first decompose $u$ as
\beno
u=\Pe^\bot u_L+(\Pe^\bot u_H+\Pe u).
\eeno
Then using Corollary \ref{coro-product} with $u=I(b), v=\mathcal{A}\Pe^\bot u_L, p_1=q_2=2, \rho=p_2=q_1=p, s_1=\sigma_2=\fr{N}{2}, s_2=\sigma_1=\fr{N}{p}-1$,  and Theorem 2.61 in \cite{Bahouri-Chemin-Danchin11}, we obtain
\beq
\nn&&\|I(b)\mathcal{A}\Pe^\bot u_L\|_{L^1_T(\dot{B}^{\fr{N}{p}-1}_{p,1})}\\
\nn&\le& C\|I(b)\|_{L^{\fr{2}{1-\al}}_T(\dot{B}^{\fr{N}{2}}_{2,1})}\|\mathcal{A}\Pe^\bot u_L\|_{L^{\fr{2}{1+\al}}_T(\dot{B}^{\fr{N}{p}-1}_{p,1})}\\
\nn&\le& C\|b\|_{L^{\fr{2}{1-\al}}_T(\dot{B}^{\fr{N}{2}}_{2,1})}\|\Pe^\bot u_L\|_{L^{\fr{2}{1+\al}}_T(\dot{B}^{\fr{N}{2}+2\al}_{2,1})}\\
&\le&CX^2(T),
\eeq
where we have used \eqref{b1}, \eqref{5.6-1} and \eqref{lfe1}. Similarly, using \eqref{b2}, we arrive at
\beq
\nn&&\|I(b)\mathcal{A}(\Pe^\bot u_H+\Pe u)\|_{L^1_T(\dot{B}^{\fr{N}{p}-1}_{p,1})}\\
\nn&\le&C\|I(b)\|_{L^{\infty}_T(\dot{B}^{\fr{N}{2}}_{2,1})}\|\mathcal{A}(\Pe^\bot u_H+\Pe u)\|_{L^{1}_T(\dot{B}^{\fr{N}{p}-1}_{p,1})}\\
\nn&\le& C\|b\|_{L^{\infty}_T(\dot{B}^{\fr{N}{2}}_{2,1})}\left(\|\Pe^\bot u_H\|_{L^{1}_T(\dot{B}^{\fr{N}{2}+1}_{2,1})}+\|\Pe u\|_{L^{1}_T(\dot{B}^{\fr{N}{p}+1}_{p,1})}\right)\\
&\le&CX^2(T).
\eeq
The proof of Lemma \ref{lem4.8} is completed.
\end{proof}

\subsection{Estimates of $X(T)$}
Using the above lemmas, we could obtain the  Dispersive estimates and Energy estimates as follows.

\noindent{\bf Step (I): Dispersive estimates.}
\begin{lem}\label{lem4.9}
Let $p$ and $\al$ satisfy \eqref{p1} and  \eqref{al1}, respectively.  Assume that $(b,u)$ is a solution to system \eqref{bcu}--\eqref{biu} in  $\mathcal{E}^{\frac{N}{2},\al}_p(T)$ with
\beno
\|b\|_{L^\infty_T(L^\infty)}\le\fr12.
\eeno
Then we have
\be\label{dis-uE}
Y_\al(T)\leq CX_L^0+CX_L(T)+CX^2(T).
\ee
\end{lem}
\begin{proof}
First of all, let us cut off the system \eqref{bd} by using the operator $P_{<1}$. Then applying Proposition \ref{prop-wave} to the resulting system  with $\eps=1$, $s=\frac{N}{2}-1+\al, \bar{p}=2, \bar{r}=\infty$, and $r=\fr{1}{\al}$,  we arrive at
\beqno
\nn Y_\al(T)
\nn &\le& C\left(\|(b_{0L}, d_{0L})\|_{\dot{B}^{\frac{N}{2}-1+\al}_{2,1}}+\|P_{<1}\dv(bu)\|_{L^1_T(\dot{B}^{\frac{N}{2}-1+\al}_{2,1})}+\|P_{<1}\Dl d\|_{L^1_T(\dot{B}^{\frac{N}{2}-1+\al}_{2,1})}
\right.\\
\nn&&\left.+\|P_{<1}\Lm^{-1}\dv(u\cdot\nb u+K(b)\nb b+I(b)\mathcal{A}u)\|_{L^1_T(\dot{B}^{\frac{N}{2}-1+\al}_{2,1})}\right)\\
&\leq&C \left(X^0_L+X_L(T)
+\|P_{<1}\dv(bu)\|_{L^1_T(\dot{B}^{\frac{N}{2}-1+\al}_{2,1})}
\right.\\
\nn&&\left.+\|P_{<1}\Lm^{-1}\dv(u\cdot\nb u+K(b)\nb b+I(b)\mathcal{A}u)\|_{L^1_T(\dot{B}^{\frac{N}{2}-1+\al}_{2,1})}\right).
\eeqno
 Combining Corollary \ref{coro1} with Lemmas \ref{lem4.3}--\ref{lem4.6}, we find that the estimate
 \eqref{dis-uE} holds. This completes the proof of Lemma \ref{lem4.9}.
\end{proof}

\noindent{\bf Step (II): Energy estimates.}\par
To begin with, let us  localize  the system \eqref{bd} as follows:
\beq\label{viscoelastic-local}
\begin{cases}
\pr_tb_q+\dot{S}_{q-1}u\cdot\nb b_q+\Lm d_q=\tl{f}_q,\\
\pr_td_q+\dot{S}_{q-1}u \cdot \nb d_q-\Dl d_q-\Lm b_q=\tl{g}_q,
\end{cases}
\eeq
with
\beqno
\tl{f}_q&:=&f_q+\left(\dot{S}_{q-1}u\cdot\nb b_q-\ddl\dot{T}_u\cdot\nb b\right),\\
\tl{g}_q&:=&g_q+\left(\dot{S}_{q-1}u \cdot \nb d_q-\ddl\dot{T}_u\cdot\nb d\right),
\eeqno
and
\beqno
{f}&:=&-b\dv u-\dot{T}'_{\pr_kb}u^k,\\
{g}&:=&\dot{T}_u\cdot\nb d-\Lm^{-1}\dv\left(u\cdot\nb u+K(b)\nb b+I(b)\mathcal{A}u\right).
\eeqno
Now, we estimate the low frequency part $X_L(T)$ and high frequency part $X_H(T)$ separately.

\noindent (i) \underline{ Estimates of $X_{L}(T)$}.\par

\begin{lem}\label{lem4.10}
Under the conditions in Lemma \ref{lem4.9}, we have
  \be\label{XQL}
	X_{L}(T)\leq CX_{L}^0+CX^2(T).
	\ee
\end{lem}
\begin{proof}
Similar to the energy estimates for the isentropic Navier-Stokes equations obtained by Danchin \cite{Danchin00},   we easily get the following three equalities
\be\label{e2}
\fr12\fr{d}{dt}\|b_q\|^2_{L^2}+(\Lm d_q|b_q)=\fr12\int \dv\dot{S}_{q-1}u |b_q|^2+(\tl{f}_q|b_q),
\ee

\be\label{e1}
\fr12\fr{d}{dt}\|d_q\|^2_{L^2}+\|\nb d_q\|_{L^2}^2-(\Lm b_q|d_q)=\fr12\int \dv\dot{S}_{q-1}u |d_q|^2+(\tl{g}_q|d_q),
\ee
and
\beq\label{e3}
\nn&&\fr{d}{dt}(d_q|\Lm b_q)-\|\Lm b_q\|_{L^2}^2+\|\Lm d_q\|_{L^2}^2-(\Dl d_q|\Lm b_q)\\
&=&\int\dv\dot{S}_{q-1}u\Lm b_qd_q-([\Lm, \dot{S}_{q-1}u\cdot\nb]b_q|d_q)+(\tl{g}_q|\Lm b_q)+(\Lm \tl{f}_q|d_q).
\eeq
A linear combination of \eqref{e2}--\eqref{e3} yields
\beq\label{eL}
\nn&&\fr12\fr{d}{dt}\left(\| b_q\|^2_{L^2}+\|d_q\|^2_{L^2}-\fr14(d_q|\Lm b_q)\right)+\fr78\|\Lm d_q\|^2_{L^2}+\fr18\|\Lm b_q\|^2_{L^2}+\fr18(\Dl d_q|\Lm b_q)\\
\nn&=&\fr12\int \dv\dot{S}_{q-1}u \left(|b_q|^2+|d_q|^2\right)-\fr18\int\dv \dot{S}_{q-1}u\Lm b_qd_q\\
&&+(\tl{g}_q|d_q-\fr18\Lm b_q)+( \tl{f}_q| b_q-\fr18\Lm d_q)+\fr18([\Lm, \dot{S}_{q-1}u\cdot \nb]b_q|d_q).
\eeq
Noting that $u=u_L+ u_H$, it is easy to see that
\beq\label{Su}
\nn\|\nb \dot{S}_{q-1}u\|_{L^\infty}&\le&C\left(2^{q(2-2\al)}\left(2^{q(2\al-2)}\|\nb \dot{S}_{q-1}u_L\|_{L^\infty}\right)+\|\nb \dot{S}_{q-1} u_H\|_{L^\infty}\right)\\
&\le&C\left(m(u)+\|\nb u_H\|_{L^\infty}\right),
\eeq
where
\be\label{m1}
m(u):=\min\left\{2^{q(2-2\al)}\|\nb u_L\|_{\dot{B}^{2\al-2}_{\infty,\infty}}, \|\nb u_L\|_{L^\infty}\right\}.
\ee
Then following the proof of Lemma 2.99 in \cite{Bahouri-Chemin-Danchin11}, we have
\beq\label{com1}
\nn&&\|[\Lm, \dot{S}_{q-1}u\cdot\nb]b_q\|_{L^2}\\
\nn&\le& C\|\nb \dot{S}_{q-1}u\|_{L^\infty}\|\Lm b_q\|_{L^2}\\
&\le&C\left(m(u)+\|\nb u_H\|_{L^\infty}\right)\left(2^q\| b_q\|_{L^2}\right).
\eeq
According to Lemma 7.5 in \cite{Danchin02}, we arrive at
\beq\label{com2}
\nn&&\|\dot{S}_{q-1}u\cdot\nb d_q-\dot{\Dl}_q\dot{T}_u\cdot\nb d\|_{L^2}\\
\nn&\le&\|\dot{S}_{q-1}u_L\cdot\nb d_q-\dot{\Dl}_q\dot{T}_{u_L}\cdot\nb d\|_{L^2}+\|\dot{S}_{q-1} u_H\cdot\nb d_q-\dot{\Dl}_q\dot{T}_{ u_H}\cdot\nb d\|_{L^2}\\
&\le&C\left(m(u)+\|\nb  u_H\|_{L^\infty}\right)\sum_{|q'-q|\le4}\|d_{q'}\|_{L^2},
\eeq
and
\be\label{com3}
\|\dot{S}_{q-1}u\cdot\nb b_q-\dot{\Dl}_q\dot{T}_u\cdot\nb b\|_{L^2}\\
\le C\left(m(u)+\|\nb  u_H\|_{L^\infty}\right)\sum_{|q'-q|\le4}\|b_{q'}\|_{L^2}.
\ee
Then thanks to Bernstein's inequality, we infer from \eqref{eL}--\eqref{com3} that, for $q\le0$, there holds
\beq\label{qle0}
\nn&&\| b_q(t)\|_{L^2}+\|d_q(t)\|_{L^2}+2^{2q}\| b_q\|_{L^1_t(L^2)}+2^{2q}\| d_q\|_{L^1_t(L^2)}\\
\nn&\le&C\left(\| b_q(0)\|_{L^2}+\|d_q(0)\|_{L^2}+\| f_q\|_{L^1_t(L^2)}+\| g_q\|_{L^1_t(L^2)}\right)\\
&&+C\int_0^t\left(m(u)+\|\nb  u_H\|_{L^\infty}\right)\sum_{|q'-q|\le4}(\| b_{q'}\|_{L^2}+\|d_{q'}\|_{L^2})dt'.
\eeq
Recalling that
	$$
X_{L}(T)=\|(b_L, \Pe^\bot u_L)\|_{\widetilde{L}^\infty_T(\dot{B}^{\frac{N}{2}-1+\al}_{2,1})\cap L^1_T(\dot{B}^{\frac{N}{2}+1+\al}_{2,1})},
	$$
multiplying \eqref{qle0} by $2^{q(\frac{N}{2}-1+\al)}$, and taking sum with respect to $q$ over $\{\cdots, -2, -1, 0 \}$,  we obtain
\begin{eqnarray}\label{e1-d4L}
X_{L}(T)&\le& C\left(X^0_{L}+\|P_{<1}f\|_{L^1_T(\dot{B}^{\frac{N}{2}-1+\al}_{2,1})}
+\|P_{<1}g\|_{L^1_T(\dot{B}^{\frac{N}{2}-1+\al}_{2,1})}\right)\nn\\
&&+C\int^T_0\sum_{q\leq0}2^{q(\frac{N}{2}-1+\al)} ( m(u) +\|\nb  u_H\|_{L^\infty})\sum_{|q'-q|\leq4} \|(\dot{\Delta}_{q'} b, \dot{\Delta}_{q'} d)\|_{L^2}ds.
\end{eqnarray}
Now we go to bound the right hand side of  \eqref{e1-d4L}. First of all, from Lemmas \ref{lem4.1}, \ref{lem4.3}--\ref{lem4.6}, we have
    \begin{equation}
      \|P_{<1}f\|_{L^1_T(\dot{B}^{\frac{N}{2}-1+\al}_{2,1})}
+\|P_{<1}g\|_{L^1_T(\dot{B}^{\frac{N}{2}-1+\al}_{2,1})}\leq CX^2(T).
    \end{equation}
The remaining terms of the right hand side of \eqref{e1-d4L} can be bounded as follows. In fact, by virtue of Young's inequality, H\"{o}lder's inequality and the high frequency embedding \eqref{hf-embedding1}, we are led to
\begin{eqnarray*}
&&\int^T_0\sum_{q\leq 0}2^{q(\frac{N}{2}-1+\al)}\|\nb  u_H\|_{L^\infty}\sum_{|q'-q|\leq4} \|(\dot{\Delta}_{q'} b,\dot{\Delta}_{q'} d)\|_{L^2}ds\\
&\leq&\int^T_0\sum_{q\leq 0}2^{q(\frac{N}{2}-1+\al)}\|\nb  u_H\|_{L^\infty}\sum_{|q'-q|\leq4} \|(\dot{\Delta}_{q'} b_L,\dot{\Delta}_{q'} d_L)\|_{L^2}ds\\
&&+\int^T_0\sum_{q\leq 0}2^{q(\frac{N}{2}-1)}\|\nb  u_H\|_{L^\infty}\sum_{|q'-q|\leq4} \|(\dot{\Delta}_{q'} b_H,\dot{\Delta}_{q'} d_H)\|_{L^2}ds\\
&\leq& C\|\nb  u_H\|_{L^1_T(L^\infty)}
\left(\|(b_L,d_L)\|_{\widetilde{L}^\infty_T(\dot{B}^{\frac{N}{2}-1+\al}_{2,1})}
+\|b_H\|_{\widetilde{L}^\infty_T(\dot{B}^{\frac{N}{2}}_{2,1})}
+\|d_H\|_{\widetilde{L}^\infty_T(\dot{B}^{\frac{N}{2}-1}_{2,1})}
\right)\\
&\leq&C\left(\|\nb  \Pe^\bot u_H\|_{L^1_T(L^\infty)}+\|\nb  \Pe u_H\|_{L^1_T(L^\infty)}\right)X(T)\\
&\leq&CX^2(T).
\end{eqnarray*}
Similarly, using \eqref{5.3}, and the interpolation
 \begin{equation}\label{5.4'}
    \widetilde{L}^\infty_T(\dot{B}^{\frac{N}{p}-1 }_{p,1})\cap L^1_T(\dot{B}^{\frac{N}{p}+1 }_{p,1})
   \subset \widetilde{L}^{\frac{1}{\al}}_T(\dot{B}^{\frac{N}{p}+2\al-1}_{p,1}) \subset \widetilde{L}^{\frac{1}{\al}}_T(\dot{B}^{2\al-1}_{\infty,1}),
     \end{equation}
we have
\begin{eqnarray*}
&&\int^T_0\sum_{q\leq0}2^{q(\frac{N}{2}+1-\al)} \|\nabla u_L\|_{\dot{B}^{2\al-2}_{\infty,\infty}}
\sum_{|q'-q|\leq4} \|(\dot{\Delta}_{q'} b_L,\dot{\Delta}_{q'} d_L)\|_{L^2}ds\\
&\leq& C\|u_L\|_{\widetilde{L}^{\frac{1}{\al}}_T(\dot{B}^{2\al-1}_{\infty,1})}\|(b_L,d_L)\|_{\widetilde{L}^{\frac{1}{1-\al}}_T(
\dot{B}^{\frac{N}{2}+1-\al}_{2,1})}\\
&\leq& C\left(\|\Pe^\bot u_L\|_{\widetilde{L}^{\frac{1}{\al}}_T(\dot{B}^{\fr{N}{p}+2\al-1}_{p,1})}+\|\Pe u_L\|_{\widetilde{L}^{\frac{1}{\al}}_T(\dot{B}^{\fr{N}{p}+2\al-1}_{p,1})}\right)\|(b_L,d_L)\|_{\widetilde{L}^{\frac{1}{1-\al}}_T(
\dot{B}^{\frac{N}{2}+1-\al}_{2,1})}\\
&\leq&CX^2(T).
\end{eqnarray*}
Finally, according to the following interpolations
   \begin{gather}
    \widetilde{L}^\infty_T(\dot{B}^{\frac{N}{2}-1}_{2,1})\cap L^1_T(\dot{B}^{\frac{N}{2}+1}_{2,1})
    \subset \widetilde{L}^{\frac{2}{\al}}_T(\dot{B}^{\frac{N}{2}-1+\al}_{2,1}),\\
   \label{4.87} \widetilde{L}^\infty_T(\dot{B}^{\frac{N}{2}-1+\al}_{2,1})\cap L^1_T(\dot{B}^{\frac{N}{2}+1+\al}_{2,1})
    \subset \widetilde{L}^{\frac{2}{2-\al}}_T(\dot{B}^{\frac{N}{2}+1}_{2,1}),\\
    \widetilde{L}^\infty_T(\dot{B}^{\frac{N}{p}-1}_{p,1})\cap L^1_T(\dot{B}^{\frac{N}{p}+1}_{p,1})
    \subset \widetilde{L}^{\frac{2}{2-\al}}_T(\dot{B}^{\frac{N}{p}+1-\al}_{p,1}),
    \end{gather}
the low frequency embedding
\beno
\|\Pe u_L\|_{\widetilde{L}^{\frac{2}{2-\al}}_T(\dot{B}^{\fr{N}{p}+1}_{p,1})}\le C\|\Pe u_L\|_{\widetilde{L}^{\frac{2}{2-\al}}_T(\dot{B}^{\fr{N}{p}+1-\al}_{p,1})},
\eeno
and the high frequency embedding
\beno
\|b_H\|_{\widetilde{L}^{\frac{2}{\al}}_T(\dot{B}^{\frac{N}{2}-1+\al}_{2,1})}\le C\|b_H\|_{\widetilde{L}^{\frac{2}{\al}}_T(\dot{B}^{\frac{N}{2}}_{2,1})},
\eeno
we find that
\begin{eqnarray}\label{4.78}
\nn&&\int^T_0\sum_{q\leq0}  \|\nabla u_L\|_{L^\infty}2^{q(\frac{N}{2}-1+\al)} \sum_{|q'-q|\leq4} \|(\dot{\Delta}_{q'} b_H,\dot{\Delta}_{q'} d_H)\|_{L^2}ds\\
\nn&\leq& C\|\nb u_L\|_{{L}^{\frac{2}{2-\al}}_T(L^\infty)}\left(
\|b_H\|_{\widetilde{L}^{\frac{2}{\al}}_T(\dot{B}^{\frac{N}{2}-1+\al}_{2,1})}
+
\|d_H\|_{\widetilde{L}^{\frac{2}{\al}}_T(\dot{B}^{\frac{N}{2}-1+\al}_{2,1})}\right)\\
\nn&\leq& C\|\nb u_L\|_{{L}^{\frac{2}{2-\al}}_T(L^\infty)}\left(
\|b_H\|_{\widetilde{L}^{\frac{2}{\al}}_T(\dot{B}^{\frac{N}{2}}_{2,1})}
+
\|d_H\|_{\widetilde{L}^{\frac{2}{\al}}_T(\dot{B}^{\frac{N}{2}-1+\al}_{2,1})}\right)\\
\nn&\leq&C\left(\|\Pe u_L\|_{\widetilde{L}^{\frac{2}{2-\al}}_T(\dot{B}^{\fr{N}{p}+1-\al}_{p,1})}+\| \Pe^\bot u_L\|_{\widetilde{L}^{\frac{2}{2-\al}}_T(\dot{B}^{\frac{N}{2}+1}_{2,1})}\right)X(T)\\
&\le&CX^2(T).
\end{eqnarray}
Combining these estimates with (\ref{e1-d4L}), we obtain \eqref{XQL}. The proof of Lemma \ref{lem4.10} is completed.
\end{proof}

\noindent (ii) \underline{ Estimates of $X_{H}(T)$}.\par
\begin{lem}\label{lem4.11}
  Under the conditions in Lemma \ref{lem4.9}, we have
  \be\label{XQH}
	X_{H}(T)\leq CX_{H}^0+CX^2(T).
	\ee
\end{lem}
\begin{proof}
To begin with, let us give the $L^2$ energy estimate for $\Lm b_q$,
\be\label{e4}
\fr12\fr{d}{dt}\|\Lm b_q\|^2_{L^2}+(\Lm^2 d_q|\Lm b_q)=\fr12\int\dv\dot{S}_{q-1}u|\Lm b_q|^2-([\Lm,\dot{S}_{q-1}u\cdot\nb]b_q|\Lm b_q)+(\Lm \tl{f}_q|\Lm b_q).
\ee
It follows from \eqref{e1}, \eqref{e3} and \eqref{e4} that
\beq\label{Hq}
\nn&&\fr12\fr{d}{dt}\left(\|\Lm b_q\|^2_{L^2}+2\|d_q\|^2_{L^2}-2(d_q|\Lm b_q)\right)+\|\Lm d_q\|^2_{L^2}+\|\Lm b_q\|^2_{L^2}-2(d_q|\Lm b_q)\\
\nn&=&\int\dv\dot{S}_{q-1}u\left(\fr12|\Lm b_q|^2+|d_q|^2\right)-\int\dv\dot{S}_{q-1}u\Lm b_qd_q\\
&&+(\tl{g}_q|2d_q-\Lm b_q)+(\Lm \tl{f}_q|\Lm b_q-d_q)+([\Lm, \dot{S}_{q-1}u\cdot \nb]b_q|d_q-\Lm b_q).
\eeq
To exhibit the smoothing effect of $u$ in high frequency case, we need the
following $L^2$ energy estimate for $d_q$,
\be\label{Huq}
\fr12\fr{d}{dt}\|d_q\|^2_{L^2}+\|\Lm d_q\|_{L^2}^2=\fr12\int\dv\dot{S}_{q-1}u|d_q|^2+(\Lm b_q|d_q)+(\tl{g}_q|d_q).
\ee
Using \eqref{com1} and Lemma 7.5 in \cite{Danchin02} again yields
\beq\label{com4}
\nn&&\|\Lm\left(\dot{S}_{q-1}u\cdot\nb b_q-\dot{\Dl}_q\dot{T}_u\cdot\nb b\right)\|_{L^2}\\
\nn&\le& C\|[\Lm,\dot{S}_{q-1}u\cdot\nb] b_q\|_{L^2}+C\|\dot{S}_{q-1}u\cdot\nb\Lm b_q-\Lm\dot{\Dl}_q\dot{T}_u\cdot\nb b\|_{L^2}\\
&\le& C\left(m(u)+\|\nb u_H\|_{L^\infty}\right)\sum_{|q'-q|\le4}2^{q'}\|b_{q'}\|_{L^2}.
\eeq
Taking the advantage of Bernstein's inequality, we infer from \eqref{Su}--\eqref{com2}, \eqref{Hq} and \eqref{com4} that,  for $q\ge1$, there holds
\beqno
&&2^q\| b_q(t)\|_{L^2}+\|d_q(t)\|_{L^2}+2^q\| b_q\|_{L^1_t(L^2)}+\| d_q\|_{L^1_t(L^2)}\\
&\le&C\left(2^q\| b_q(0)\|_{L^2}+\|d_q(0)\|_{L^2}+2^q\| f_q\|_{L^1_t(L^2)}+\| g_q\|_{L^1_t(L^2)}\right)\\
&&+C\int_0^t\left(m(u)+\|\nb u_H\|_{L^\infty}\right)\sum_{|q'-q|\le4}(2^{q'}\| b_{q'}\|_{L^2}+\|d_{q'}\|_{L^2})dt'.
\eeqno
Similarly, for $q\geq1$,  \eqref{Su}, \eqref{com2} and \eqref{Huq} imply that
\beqno
\nn\|d_q(t)\|_{L^2}+2^{2q}\| d_q\|_{L^1_t(L^2)}&\le& C\left(\|d_q(0)\|_{L^2}+2^q\| b_q\|_{L^1_t(L^2)}+\|g_q\|_{L^1_t(L^2)}\right)\\
&&+C\int_0^t\left(m(u)+\|\nb  u_H\|_{L^\infty}\right)\sum_{|q'-q|\le4}(2^{q'}\| b_{q'}\|_{L^2}+\|d_{q'}\|_{L^2})dt'.
\eeqno
Combining these two inequalities, we find that, if $q\ge1$, there holds
\beq\label{qge1}
\nn&&2^q\| b_q(t)\|_{L^2}+\|d_q(t)\|_{L^2}+2^q\| b_q\|_{L^1_t(L^2)}+2^{2q}\| d_q\|_{L^1_t(L^2)}\\
\nn&\le&C\left(2^q\| b_q(0)\|_{L^2}+\|d_q(0)\|_{L^2}+2^q\| f_q\|_{L^1_t(L^2)}+\| g_q\|_{L^1_t(L^2)}\right)\\
&&+C\int_0^t\left(m(u)+\|\nb  u_H\|_{L^\infty}\right)\sum_{|q'-q|\le4}(2^{q'}\| b_{q'}\|_{L^2}+\|d_{q'}\|_{L^2})dt'.
\eeq
Multiplying   \eqref{qge1} by $2^{q(\frac{N}{2}-1)}$, and taking sum with respect to $q$ over  $\{1, 2, 3,  \cdots\}$,   we arrive at
\begin{eqnarray}\label{e1-d4H}
X_{H}(T)
\nn&\le& C\left(X^0_{H}+\|P_{\geq1}f\|_{L^1_T(\dot{B}^{\frac{N}{2}}_{2,1})}
+\|P_{\geq1}g\|_{L^1_T(\dot{B}^{\frac{N}{2}-1}_{2,1})}\right)\\
&&+C\int^T_0\sum_{q\geq1}2^{q(\frac{N}{2}-1)} (m(u)+ \|\nb   u_H\|_{L^\infty})\sum_{|q'-q|\leq4}( 2^{q'} \|\dot{\Delta}_{q'} b\|_{L^2}+\|\dot{\Delta}_{q'} d\|_{L^2})ds.
\end{eqnarray}
Now let us  bound the right hand side of \eqref{e1-d4H}. In fact, we infer from Lemmas \ref{lem4.2}, and \ref{lem4.3}--\ref{lem4.6} that
    \begin{equation}
     \|P_{\geq1}f\|_{L^1_T(\dot{B}^{\frac{N}{2}}_{2,1})}
+\|P_{\geq1}g\|_{L^1_T(\dot{B}^{\frac{N}{2}-1}_{2,1})}\leq C X^2(T).
    \end{equation}
The estimates of the last term in (\ref{e1-d4H}) are a little bit trickier. First of all, using
Young's inequality, H\"{o}lder's inequality, and \eqref{b2} yields
\begin{eqnarray*}
 &&\int^T_0\sum_{q\geq1}2^{q(\frac{N}{2}-1)} \|\nb  u_H\|_{L^\infty}
\sum_{|q'-q|\leq4}(2^{q'} \|\dot{\Delta}_{q'} b\|_{L^2} + \|\dot{\Delta}_{q'} d\|_{L^2})ds\nn\\
&\le&\int^T_0\sum_{q\geq1}2^{q(\frac{N}{2}-1)} \|\nb  u_H\|_{L^\infty}
\sum_{|q'-q|\leq4}(2^{q'} \|\dot{\Delta}_{q'} b\|_{L^2} + \|\dot{\Delta}_{q'} d_H\|_{L^2})ds\nn\\
&&+\int^T_0\sum_{q\geq1}2^{q(\frac{N}{2}-1+\al)} \|\nb  u_H\|_{L^\infty}
\sum_{|q'-q|\leq4} \|\dot{\Delta}_{q'} d_L\|_{L^2}ds\nn\\
&\leq&C\| \nb u_H\|_{{L}^1_T(L^\infty)}
\left( \| b\|_{\widetilde{L}^\infty_T(\dot{B}^{\frac{N}{2}}_{2,1})}+\| d_L\|_{\widetilde{L}^\infty_T(\dot{B}^{\frac{N}{2}-1+\al}_{2,1})}+\| d_H\|_{\widetilde{L}^\infty_T(\dot{B}^{\frac{N}{2}-1}_{2,1})}\right)\\
&\le&CX^2(T).
\end{eqnarray*}
Moreover, using \eqref{5.3}, the following low frequency embedding
\beno
\| b_L\|_{\widetilde{L}^{\frac{1}{1-\al}}_T(\dot{B}^{\frac{N}{2}+2-2\al}_{2,1})}\le C\| b_L\|_{\widetilde{L}^{\frac{1}{1-\al}}_T(\dot{B}^{\frac{N}{2}+1-\al}_{2,1})},
\eeno
and the interpolation \eqref{5.4'}, we find that
\beqno
&&\int^T_0\sum_{q\geq1}2^{q(\frac{N}{2}+1-2\al)} \|\nb  u_L\|_{\dot{B}^{2\al-2}_{\infty,\infty}} \nn\sum_{|q'-q|\leq4}\left(2^{q'}\|\dot{\Delta}_{q'} b_L\|_{L^2}+\|\dot{\Delta}_{q'} d_L\|_{L^2}\right)ds\\
&\le&C\|u_L\|_{\widetilde{L}^\frac{1}{\al}_T(\dot{B}^{2\al-1}_{\infty,1})}\| b_L\|_{\widetilde{L}^{\frac{1}{1-\al}}_T(\dot{B}^{\frac{N}{2}+2-2\al}_{2,1})}+C\int^T_0\|\nb  u_L\|_{\dot{B}^{2\al-2}_{\infty,\infty}}\sum_{q\geq1}2^{q(\frac{N}{2}+1-\al)}  \nn\sum_{|q'-q|\leq4}\|\dot{\Delta}_{q'} d_L\|_{L^2}ds\\
&\le&C \|u_L\|_{\widetilde{L}^\frac{1}{\al}_T(\dot{B}^{2\al-1}_{\infty,1})}\left(\| b_L\|_{\widetilde{L}^{\frac{1}{1-\al}}_T(\dot{B}^{\frac{N}{2}+1-\al}_{2,1})}+\| d_L\|_{\widetilde{L}^{\frac{1}{1-\al}}_T(\dot{B}^{\frac{N}{2}+1-\al}_{2,1})}\right)\\
&\le&CX^2(T).
\eeqno
Similar to \eqref{4.78}, we have
\begin{eqnarray*}
 &&\int^T_0\sum_{q\geq1}2^{q(\frac{N}{2}-1)} \|\nb   u_L\|_{L^\infty}
\sum_{|q'-q|\leq4}2^{q'}\|\dot{\Delta}_{q'} b_H\|_{L^2}ds\nn\\
&\leq&C\|\nb u_L\|_{\widetilde{L}^\frac{2}{2-\al}_T(L^\infty)}
\|b_H\|_{\widetilde{L}^{\frac{2}{\al}}_T(\dot{B}^{\frac{N}{2}}_{2,1})}\\
&\leq&C\left(\|\Pe u_L\|_{\widetilde{L}^{\frac{2}{2-\al}}_T(\dot{B}^{\fr{N}{p}+1-\al}_{p,1})}+\| \Pe^\bot u_L\|_{\widetilde{L}^{\frac{2}{2-\al}}_T(\dot{B}^{\fr{N}{2}+1}_{2,1})}\right)X(T)\\
&\le&CX^2(T).
\end{eqnarray*}
Finally, using \eqref{5.5} and  \eqref{5.4'} again, we arrive at
\beqno
&&\int^T_0\sum_{q\geq1}2^{q(\frac{N}{2}+1-2\al)} \|\nb  u_L\|_{\dot{B}^{2\al-2}_{\infty,\infty}} \nn\sum_{|q'-q|\leq4}\|\dot{\Delta}_{q'} d_H\|_{L^2}ds\\
&\le&C\|u_L\|_{\widetilde{L}^\frac{1}{\al}_T(\dot{B}^{2\al-1}_{\infty,1})}\| d_H\|_{\widetilde{L}^{\frac{1}{1-\al}}_T(\dot{B}^{\frac{N}{2}+1-2\al}_{2,1})}\\
&\le&CX^2(T).
\eeqno
Combining these estimates with (\ref{e1-d4H}), we obtain
\eqref{XQH}. This completes the proof of Lemma \ref{lem4.11}.
\end{proof}

\noindent (iii) \underline{ Estimates of $W(T)$}.\par
In fact, applying  Proposition \ref{prop3.4} to \eqref{biu}, and using  Lemmas \ref{lem4.7} and \ref{lem4.8}, we easily get the following estimate for $W(T)$.
\begin{lem}\label{lem4.12}
  Under the conditions in Lemma \ref{lem4.9}, we have
  \be\label{W}
	W(T)\leq CW^0+CX^2(T).
	\ee
\end{lem}
Collecting Lemmas \ref{lem4.9}--\ref{lem4.12}, we conclude that
\begin{prop}\label{prop-global1}
Let $p$ and $\al$ satisfy \eqref{p1} and  \eqref{al1}, respectively.  Assume that $(b,u)$ is a solution to system \eqref{bcu}--\eqref{biu} in  $\mathcal{E}^{\frac{N}{2},\al}_p(T)$ with
\beno
\|b\|_{L^\infty_T(L^\infty)}\le\fr12.
\eeno
Then we have
\be\label{e-global}
X(T)\le CX^0+CX^2(T).
\ee
\end{prop}

\section{Proof of Theorem \ref{thm-p>2}}
\noindent The aim of this Section is to give the proof of Theorem \ref{thm-p>2}.

\subsection{The global existence.}
First of all, we construct the approximate solutions to that system \eqref{bcu}--\eqref{biu} with smoothing initial data. For the sake of
simplicity, we just outline it here (for the details, see
e.g. \cite{Bahouri-Chemin-Danchin11}  and \cite{Danchin07}). To begin with, let us recall the following local existence theorem.
\begin{thm}[\cite{Danchin07}]\label{local}
Let $N\ge2$. Assume that $\rho_0-1\in\dot{B}^{\fr{N}{2}}_{2,1}$ and $u_0\in\dot{B}^{\fr{N}{2}-1}_{2,1}$ with $\rho_0$ bounded away from 0. There exists a positive time $T$
such that system \eqref{CNS} has a unique solution $(\rho, u)$ with $\rho$ bounded  away from 0,
\beno
\rho-1\in\widetilde{C}_T(\dot{B}^{\fr{N}{2}}_{2,1}), \ \ {and} \ \ u\in\widetilde{C}_T(\dot{B}^{\fr{N}{2}-1}_{2,1})\cap L^1_T(\dot{B}^{\fr{N}{2}+1}_{2,1}).
\eeno
Moreover, the solution $(\rho, u)$ can be continued beyond $T$ if the following three conditions hold:\\
{\em (i)}  The function $\rho-1$ belongs to $L^\infty_T(\dot{B}^{\fr{N}{2}}_{2,1})$,\\
{\em (ii)} the function $\rho$ is bounded away from 0,\\
{\em (iii)} $\int_0^T\|\nb u(\tau)\|_{L^\infty}d\tau<\infty$.
\end{thm}
\begin{rem}\label{rem5.1}
In addition, we claim that if $\rho_0-1\in\dot{B}^{\fr{N}{2}-1+\al}_{2,1}$, then $\rho-1\in\widetilde{C}_T(\dot{B}^{\fr{N}{2}-1+\al}_{2,1})$. In fact, using Proposition \ref{prop3.3}, and Corollary \ref{coro-product}, we have
\beqno
\|b\|_{\widetilde{L}^\infty_T(\dot{B}^{\fr{N}{2}-1+\al}_{2,1})}&\le& \exp\left\{C\|\nb u\|_{L^1_T(\dot{B}^{\fr{N}{2}}_{2,1})}\right\}\left(\|b_0\|_{\dot{B}^{\fr{N}{2}-1+\al}_{2,1}}+\int_0^T\|(1+b)\dv u\|_{\dot{B}^{\fr{N}{2}-1+\al}_{2,1}}\right)\\
&\le&C\left(\|b_0\|_{\dot{B}^{\fr{N}{2}-1+\al}_{2,1}}+\left(1+\|b\|_{L^\infty_T(\dot{B}^{\fr{N}{2}}_{2,1})}\right)\|\dv u\|_{L^1_T(\dot{B}^{\fr{N}{2}-1+\al}_{2,1})}\right)\\
&\le&C\left(\|b_0\|_{\dot{B}^{\fr{N}{2}-1+\al}_{2,1}}+\|u\|^{\fr{1-\al}{2}}_{L^\infty_T(\dot{B}^{\fr{N}{2}-1}_{2,1})}\|u\|^{\fr{1+\al}{2}}_{L^1_T(\dot{B}^{\fr{N}{2}+1}_{2,1})}T^{\fr{1-\al}{2}}\right)\\
&\le&C.
\eeqno
\end{rem}

For initial data $(\rho_0, u_0)$ with $(\rho_0-1, u_0)=:(b_0, u_0)\in\mathcal{E}_0$, by embedding, it is easy to see that
\be\label{eta1}
\|b_0\|_{L^\infty}\le C\|b_0\|_{\dot{B}^{\fr{N}{2}}_{2,1}}\le C\|(b_0, u_0)\|_{\mathcal{E}_0}.
\ee
Before proceeding any further, let us denote by $\tl{C}$ the maximum of constants 1 and $C$ appearing in Proposition \ref{prop-global1} and \eqref{eta1}, and choose $(\rho_0, u_0)$ with $(b_0, u_0)$ so small that
\be\label{small1}
\|(b_0, u_0)\|_{\mathcal{E}_0}\le\fr{1}{8\tl{C}^2}.
\ee
It follows from \eqref{eta1} and \eqref{small1} that
\be
\|b_0\|_{L^\infty}\le \fr18.
\ee
Thanks to Proposition 2.27 in \cite{Bahouri-Chemin-Danchin11}, we can find a sequence of functions $\{(b_{0}^n, u_{0}^n)\}\subset\mathcal{S}\times\mathcal{S}$ satisfying
 \be\label{appro1}
\|(b^n_0-b_0, u^n_0-u_0)\|_{\mathcal{E}_0}\rightarrow 0, \quad\mathrm{as}\quad n\rightarrow\infty,
\ee
and
\be\label{bn}
\|b_0^n\|_{L^\infty}\le\fr14, \quad\mathrm{for\ \ all}\quad n\in\N.
\ee
Then using Theorem \ref{local} and Remark \ref{rem5.1} above,   one could obtain a unique local solution $(b^n,u^n)$ to the system \eqref{bcu}--\eqref{biu} with smoothing initial data $(b_0^n,u_0^n)$ on the maximal lifespan $[0,T^n_*)$,   satisfying
    $$
    (b^n,u^n)\in\left( \widetilde{C}_T(\dot{B}^{\frac{N}{2}-1+\al}_{2,1}\cap\dot{B}^{\frac{N}{2}}_{2,1})\times \left( \widetilde{C}_T(\dot{B}^{\frac{N}{2}-1}_{2,1})\cap
L^1_T(\dot{B}^{\frac{N}{2}+1}_{2,1})\right)\right)\subset\mathcal{E}^{\fr{N}{2},\al}_p(T),\ \forall\ T\in(0,T^n_*).
$$
Now define $T_1^n$ be the supremum of all time $T'\in[0,T^n_*)$ such that
\be\label{III1}
X^n(T')\le{4\tl{C}}\|(b_0, u_0)\|_{\mathcal{E}_0},
\ee
where
    \begin{eqnarray*}
      X^n(T)&=&\|(b^n_L,\Pe^\bot u^n_L)\|_{\widetilde{L}^\infty_T(\dot{B}^{\frac{N}{2}-1+\al}_{2,1})\cap L^1_T(\dot{B}^{\frac{N}{2}+1+\al}_{2,1})}+\|(b^n_L,\Pe^\bot u^n_L)\|_{\widetilde{L}^{\frac{1}{\al}}_T(\dot{B}^{\fr{N}{p}+2\al-1}_{p,1})} \\
      &&+\|b^n_H\|_{\widetilde{L}^\infty_T(\dot{B}^{\frac{N}{2}}_{2,1})\cap L^1_T(\dot{B}^{\frac{N}{2}}_{2,1})}
 +\|\Pe^\bot u^n_H\|_{\widetilde{L}^\infty_T(\dot{B}^{\frac{N}{2}-1}_{2,1})\cap L^1_T(\dot{B}^{\frac{N}{2}+1}_{2,1})}
+\|\Pe u^n\|_{\widetilde{L}^\infty_T(\dot{B}^{\fr{N}{p}-1}_{p,1})\cap L^1_T(\dot{B}^{\fr{N}{p}+1}_{p,1})}.
    \end{eqnarray*}
Combining \eqref{III1} with \eqref{eta1}--\eqref{small1}, one easily deduces that
\be\label{bninfty}
\|b^n\|_{L^\infty_{T_1^n}(L^\infty)}\le\fr12.
\ee
Then from  Proposition \ref{prop-global1} and \eqref{appro1}, we find that
    \beq
      \nn X^n(T^n_1)&\leq& \tl{C}\|(b_0^n, u_0^n)\|_{\mathcal{E}_0}+16\tl{C}^3\|(b_0, u_0)\|_{\mathcal{E}_0}^2\\
                           \nn&\le&2\tl{C}\|(b_0, u_0)\|_{\mathcal{E}_0}\left(1+8\tl{C}^2\|(b_0, u_0)\|_{\mathcal{E}_0}\right)\\
                                 &\le&3\tl{C}\|(b_0, u_0)\|_{\mathcal{E}_0},
    \eeq
provided the initial data $(\rho_0, u_0)$ satisfy
\be\label{III2}
\|(b_0, u_0)\|_{\mathcal{E}_0} \le\fr{1}{16\tl{C}^2}.
\ee
Thus $T_1^n=T^n_*$, and  \eqref{III1} holds true on the interval $[0,T^n_*)$ provided $\|(b_0, u_0)\|_{\mathcal{E}_0} \le c_0$ with $c_0:=\fr{1}{16\tl{C}^2}$.
Consequently, \eqref{bninfty} holds with $T_1^n$ replaced by $T^n_*$, and
    $$
    \|b^n\|_{\widetilde{L}^\infty_{T^n_*}(\dot{B}^{\frac{N}{2}}_{2,1})}+\int^{T_*^n}_0\|\nabla \Pe u^n\|_{L^\infty}+\|\nabla \Pe^\bot u^n_H\|_{L^\infty}dt
    +\left(\int^{T_*^n}_0 \|\nabla \Pe^\bot u^n_L\|_{L^\infty}^{\frac{2}{2-\alpha}}dt\right)^{\frac{2-\alpha}{2}}\leq  C.
    $$
Therefore, using Theorem \ref{local} again, we conclude that $T^n_*=+\infty$ for all $n\in\N$. Moreover,  for all $n\in\mathbb{N}$, there holds
 $$
    X^n(T)\le {C}_0\|(b_0, u_0)\|_{\mathcal{E}_0}, \ \forall\ T>0,
    $$
with $C_0:=4\tl{C}$. Then, using the  compactness arguments similar as that in Chapter 10 of \cite{Bahouri-Chemin-Danchin11}, we obtain that $\{(b^n,u^n)\}_{n\in\mathbb{N}}$
weakly converges (up to a subsequence) to some global solution $(b,u)$ to the system \eqref{bcu}--\eqref{biu} with the initial data $(b_0,u_0)$ satisfying \eqref{initial1}. Thus, we prove the global existence part of Theorem \ref{thm-p>2}.

\subsection{The uniqueness  when $N\geq3$.}

Next, we will prove the uniqueness part of Theorem \ref{thm-p>2} when $N\geq3$.
Assume there exist two solutions $(b^1,u^1)$ and $(b^2,u^2)$ for the system \eqref{bcu}--\eqref{biu} with the same initial data $(b_0,u_0)$, satisfying the regularity conditions in Theorem \ref{thm-p>2}. In order to show that these two solutions coincide, we shall give some estimates for $(\delta b,\delta u)=(b^2-b^1,u^2-u^1)$. It is easy to verify that $(\delta b,\delta u)$ satisfies the following system
\beq\label{bcu-uniqu}
\begin{cases}
\pr_t \delta b+u^2\cdot\nabla \delta b=-{\dv \Pe^\bot \delta u} -\delta u\cdot\nabla b^1-\delta b\mathrm{div}u^2-b^1\mathrm{div}\delta u,\\
\pr_t \Pe^\bot \delta u- \Dl\Pe^\bot \delta u+{\nb \delta b}\\
\ \ \ \ \ \ \ \ =-\Pe^\bot\left(u^2\cdot\nb u^2+K( b^2){\nb b^2}+I( b^2)\mathcal{A}u^2
-u^1\cdot\nb u^1-K( b^1){\nb b^1}-I( b^1)\mathcal{A}u^1\right),\\
\pr_t \Pe \delta  u-\mu \Dl\Pe \delta  u=-\Pe\left(u^2\cdot\nb u^2+I( b^2)\mathcal{A}u^2-u^1\cdot\nb u^1-I( b^1)\mathcal{A}u^1\right),\\
(\delta b, \delta  u )|_{t=0}=(0, 0).
\end{cases}
\eeq
Following the proof of Proposition \ref{prop3.3},  using Corollary \ref{coro-product} and Lemma 2.100 in \cite{Bahouri-Chemin-Danchin11}, we have
    \begin{eqnarray}
      &&\|\delta b(t)\|_{\dot{B}^{\frac{N}{2}-1}_{2,1}}\nn\\
       &\leq&  \int^t_0\left(
      \|-{\dv \Pe^\bot \delta u}\|_{\dot{B}^{\frac{N}{2}-1}_{2,1}}+\|\delta u\cdot\nabla b^1\|_{\dot{B}^{\frac{N}{2}-1}_{2,1}}
      +\|\delta b\mathrm{div}u^2\|_{\dot{B}^{\frac{N}{2}-1}_{2,1}}+\|b^1\mathrm{div}\Pe^\bot\delta u\|_{\dot{B}^{\frac{N}{2}-1}_{2,1}}\right.\nonumber\\
      &&\left.+\fr12\|\dv u^2\|_{L^\infty}\| \delta b\|_{\dot{B}^{\frac{N}{2}-1}_{2,1}}+\sum_{q\in \mathbb{Z}}2^{q(\frac{N}{2}-1)}\|[\dot{\Delta}_q,u^2]\cdot\nabla\delta b\|_{L^2}
      \right)ds\nonumber\\
      &\leq &C \int^t_0
      \left(\|\Pe^\bot  \delta u\|_{\dot{B}^{\frac{N}{2}}_{2,1}}+\|\Pe\delta u\|_{\dot{B}^\frac{N}{p}_{p,1}}\|  b^1\|_{\dot{B}^{\frac{N}{2}}_{2,1}}
      +\|\Pe^\bot\delta u\|_{\dot{B}^\frac{N}{2}_{2,1}}\|  b^1\|_{\dot{B}^{\frac{N}{2}}_{2,1}}
      +\| \delta b\|_{\dot{B}^{\frac{N}{2}-1}_{2,1}}\|u^2\|_{\dot{B}^{\frac{N}{p}+1}_{p,1}}
      \right)
      ds.      \label{5.15-Ng3}
    \end{eqnarray}
Applying Proposition \ref{prop3.4} to $\eqref{bcu-uniqu}_2$, we find that
 \begin{eqnarray}
     && \|\Pe^\bot\delta u\|_{\widetilde{L}^2_t(\dot{B}^{\frac{N}{2}-1}_{2,1})\cap \widetilde{L}^1_t(\dot{B}^{\frac{N}{2}}_{2,1})}\nonumber\\
      &\leq& C\int^t_0\left(\|{\nb \delta b}\|_{\dot{B}^{\frac{N}{2}-2}_{2,1}}
      +\|\Pe^\bot(u^2\cdot \nabla \delta u)\|_{\dot{B}^{\frac{N}{2}-2}_{2,1}}+\|\Pe^\bot(\delta u\cdot \nabla u^1)\|_{\dot{B}^{\frac{N}{2}-2}_{2,1}}+\|K( b^2){\nb b^2}-K( b^1){\nb b^1}\|_{\dot{B}^{\frac{N}{2}-2}_{2,1}}
      \right.\nonumber\\
      &&\left.+\|\Pe^\bot[I( b^2)\mathcal{A}\delta u]\|_{\dot{B}^{\frac{N}{2}-2}_{2,1}}
      +\|\Pe^\bot[(I( b^2)-I(b^1))\mathcal{A} u^1]\|_{\dot{B}^{\frac{N}{2}-2}_{2,1}}
      \right)ds.\label{5.3-Ng3}
    \end{eqnarray}
Since
	\begin{eqnarray*}
	&&\|\Lambda^{-1}\mathrm{div}( u^2\cdot \nabla \Pe \delta u)\|_{\dot{B}^{\frac{N}{2}-2}_{2,1}}\\
&\leq&\|\dot{T}_{ \nabla u^2}\nabla\Pe \delta u\|_{\dot{B}^{\frac{N}{2}-3}_{2,1}}+\| \dot{T}_{ \nabla\Pe \delta u} \nabla u^2\|_{\dot{B}^{\frac{N}{2}-3}_{2,1}}+\|\dot{R}(\Pe^\bot u^2,\nabla\Pe \delta u)\|_{\dot{B}^{\frac{N}{2}-2}_{2,1}}+
\|\mathrm{div}\dot{R}(\Pe u^2,\Pe \delta u)\|_{\dot{B}^{\frac{N}{2}-2}_{2,1}}\\
&\leq&C\|\nabla u^2\|_{\dot{B}^{\frac{N}{p}-1}_{p,1}}\|\nabla\Pe \delta u\|_{\dot{B}^{\frac{N}{p}-2}_{p,1}}
+C\|\Pe^\bot u^2\|_{\dot{B}^{\frac{N}{2}}_{2,1}}\|\nabla\Pe \delta u\|_{\dot{B}^{\frac{N}{p}-2}_{p,1}}
+C\|\Pe u^2\|_{\dot{B}^{\frac{N}{p}-1}_{p,1}}\|\Pe \delta u\|_{\dot{B}^{\frac{N}{p}}_{p,1}}\\
&\leq&C
\left( \|\Pe^\bot u^2\|_{\dot{B}^{\frac{N}{2}}_{2,1}}
+\|\Pe u^2\|_{\dot{B}^{\frac{N}{p}}_{p,1}}\right)\|\Pe \delta u\|_{\dot{B}^{\frac{N}{p}-1}_{p,1}}
+ C\|\Pe u^2\|_{\dot{B}^{\frac{N}{p}-1}_{p,1}}\|\Pe \delta u\|_{\dot{B}^{\frac{N}{p}}_{p,1}},
	\end{eqnarray*}
and
\begin{eqnarray*}
	\| u^2\cdot \nabla \Pe^\bot \delta u\|_{\dot{B}^{\frac{N}{2}-2}_{2,1}}&\leq&\|\dot{T}_{ u^2}\cdot\nabla\Pe^\bot \delta u)\|_{\dot{B}^{\frac{N}{2}-2}_{2,1}}+\| \dot{T}_{ \nabla\Pe^\bot \delta u} u^2\|_{\dot{B}^{\frac{N}{2}-2}_{2,1}}+\|\dot{R}(u^2,\nabla\Pe^\bot \delta u)\|_{\dot{B}^{\frac{N}{2}-2}_{2,1}}\\
&\leq&C\| u^2\|_{\dot{B}^{\frac{N}{p}}_{p,1}}\|\nabla\Pe^\bot \delta u\|_{\dot{B}^{\frac{N}{2}-2}_{2,1}}
 +C\| u^2\|_{\dot{B}^{\frac{N}{p}}_{p,1}}\|\Pe^\bot \delta u\|_{\dot{B}^{\frac{N}{2}-1}_{2,1}}\\
&\leq&C\left( \|\Pe^\bot u^2\|_{\dot{B}^{\frac{N}{2}}_{2,1}}
+\|\Pe u^2\|_{\dot{B}^{\frac{N}{p}}_{p,1}}\right)\|\Pe^\bot \delta u\|_{\dot{B}^{\frac{N}{2}-1}_{2,1}},
	\end{eqnarray*}
we get
\begin{eqnarray}
	&&\| \Pe^\bot(u^2\cdot \nabla   \delta u)\|_{\dot{B}^{\frac{N}{2}-2}_{2,1}} \nonumber\\
&\leq&C\left( \|\Pe^\bot u^2\|_{\dot{B}^{\frac{N}{2}}_{2,1}}
+\|\Pe u^2\|_{\dot{B}^{\frac{N}{p}}_{p,1}}\right)
\left(\|\Pe \delta u\|_{\dot{B}^{\frac{N}{p}-1}_{p,1}}+\|\Pe^\bot \delta u\|_{\dot{B}^{\frac{N}{2}-1}_{2,1}}\right)
+ C\|\Pe u^2\|_{\dot{B}^{\frac{N}{p}-1}_{p,1}}\|\Pe \delta u\|_{\dot{B}^{\frac{N}{p}}_{p,1}}.
	\end{eqnarray}
Similarly, we have
\begin{eqnarray*}
	\| \Pe\delta u\cdot \nabla u^1\|_{\dot{B}^{\frac{N}{2}-2}_{2,1}}&\leq&\|\dot{T}_{ \Pe \delta u }\cdot\nabla    u^1\|_{\dot{B}^{\frac{N}{2}-2}_{2,1}}+\| \dot{T}_{ \nabla u^1} \Pe \delta u\|_{\dot{B}^{\frac{N}{2}-2}_{2,1}}+\|\mathrm{div}\dot{R}(\Pe \delta u ,  u^1)\|_{\dot{B}^{\frac{N}{2}-2}_{2,1}} \\
&\leq&C\|  u^1\|_{\dot{B}^{\frac{N}{p}}_{p,1}}\|\Pe \delta u\|_{\dot{B}^{\frac{N}{p}-1}_{p,1}},
 	\end{eqnarray*}
\begin{eqnarray*}
	\| \Pe^\bot\delta u\cdot \nabla u^1\|_{\dot{B}^{\frac{N}{2}-2}_{2,1}}&\leq&\|\dot{T}_{ \Pe^\bot \delta u }\cdot\nabla    u^1\|_{\dot{B}^{\frac{N}{2}-2}_{2,1}}+\| \dot{T}_{ \nabla u^1} \Pe^\bot \delta u\|_{\dot{B}^{\frac{N}{2}-2}_{2,1}}+\| \dot{R}(\Pe^\bot \delta u ,  \nabla u^1)\|_{\dot{B}^{\frac{N}{2}-2}_{2,1}} \\
&\leq&C\| \nabla u^1\|_{\dot{B}^{\frac{N}{p}-1}_{p,1}}\|\Pe^\bot \delta u\|_{\dot{B}^{\frac{N}{2}-1}_{2,1}}.
  	\end{eqnarray*}
Thus,
	\begin{eqnarray}
\|\delta u\cdot \nabla u^1\|_{\dot{B}^{\frac{N}{2}-2}_{2,1}}
\leq C\|  u^1\|_{\dot{B}^{\frac{N}{p}}_{p,1}}\left(\|\Pe \delta u\|_{\dot{B}^{\frac{N}{p}-1}_{p,1}}+
\|\Pe^\bot \delta u\|_{\dot{B}^{\frac{N}{2}-1}_{2,1}}\right).
	\end{eqnarray}
Using Corollary \ref{coro-product} and Theorem 2.61 in \cite{Bahouri-Chemin-Danchin11}, we obtain
 \begin{eqnarray}
    &&  \|K( b^2){\nb b^2}-K( b^1){\nb b^1}\|_{\dot{B}^{\frac{N}{2}-2}_{2,1}}=\|\nabla[\widetilde{K}(b^2)-\widetilde{K}(b^1)]\|_{\dot{B}^{\frac{N}{2}-2}_{2,1}}\nonumber\\
    &\leq &C\| \widetilde{K}(b^2)-\widetilde{K}(b^1) \|_{\dot{B}^{\frac{N}{2}-1}_{2,1}}=C\left\|\int^1_0  K (b^1+\tau(b^2-b^1))d\tau \delta b\right\|_{\dot{B}^{\frac{N}{2}-1}_{2,1}}
    \nonumber\\
    &\leq&C\left\|\int^1_0  K (b^1+\tau(b^2-b^1))d\tau\right\|_{\dot{B}^{\frac{N}{2}}_{2,1}}\|\delta b \|_{\dot{B}^{\frac{N}{2}-1}_{2,1}} \nonumber\\
    &\leq&  C \| (b^1,b^2 ) \|_{\dot{B}^{\frac{N}{2}}_{2,1}}\|\delta b \|_{\dot{B}^{\frac{N}{2}-1}_{2,1}},
    \end{eqnarray}
 where $\widetilde{K}(b)=\int^b_0 K(s)ds$. Noting that $\Pe^\bot\Pe=0$, in view of Theorem 2.99 in \cite{Bahouri-Chemin-Danchin11} and Corollary \ref{coro-product}, we  are led to
	\begin{eqnarray}
	&&\|\Pe^\bot[I( b^2)\mathcal{A}\delta u]\|_{\dot{B}^{\frac{N}{2}-2}_{2,1}}\nonumber\\
	&\leq&\|[\Pe^\bot, I( b^2)]\mathcal{A}\Pe\delta u\|_{\dot{B}^{\frac{N}{2}-2}_{2,1}}
+\|I( b^2)\mathcal{A}\Pe^\bot\delta u\|_{\dot{B}^{\frac{N}{2}-2}_{2,1}}\nonumber\\
	&\leq&C\|\nabla I(b^2)\|_{\dot{B}^{\frac{N}{p*}-1}_{p*,1}}\|\Pe \delta u\|_{\dot{B}^{\frac{N}{p}}_{p,1}}+C\| I(b^2)\|_{\dot{B}^{\frac{N}{2}}_{2,1}}\|\Pe \delta u\|_{\dot{B}^{\frac{N}{p}}_{p,1}}
+C\| I(b^2)\|_{\dot{B}^{\frac{N}{2}}_{2,1}}\|\Pe^\bot\delta u\|_{\dot{B}^{\frac{N}{2}}_{2,1}}\nonumber\\
	&\leq& C\| b^2\|_{\dot{B}^{\frac{N}{2}}_{2,1}}\left(\|\Pe^\bot\delta u\|_{\dot{B}^{\frac{N}{2}}_{2,1}}
+\|\Pe \delta u\|_{\dot{B}^{\frac{N}{p}}_{p,1}}\right).
	\end{eqnarray}
Moreover, using Corollary \ref{coro-product} with $u= I(b^2)-I(b^1), v=\mathcal{A} u^1, \rho=p_1=q_2=2, p_2=q_1=p$, and $s_1=\sigma_2=\fr{N}{2}-1, s_2=\sigma_1=\fr{N}{p}-1$,
we find that
     \begin{eqnarray}
&&\|(I( b^2)-I(b^1))\mathcal{A} u^1\|_{\dot{B}^{\frac{N}{2}-2}_{2,1}}\nonumber\\
	&\leq&C\| I(b^2)-I(b^1)\|_{\dot{B}^{\frac{N}{2}-1}_{2,1}}
\| \mathcal{A}   u^1\|_{\dot{B}^{\frac{N}{p}-1}_{p,1}}\nonumber\\
    &\leq& C\left\|\int^1_0  I' (b^1+\tau(b^2-b^1))d\tau \delta b\right\|_{\dot{B}^{\frac{N}{2}-1}_{2,1}}
\|   u^1\|_{\dot{B}^{\frac{N}{p}+1}_{p,1}}
    \nonumber\\
    &\leq&C\left(1+\left\|\int^1_0  I' (b^1+\tau(b^2-b^1))d\tau-1\right\|_{\dot{B}^{\frac{N}{2}}_{2,1}}\right)\|\delta b \|_{\dot{B}^{\frac{N}{2}-1}_{2,1}}
\|   u^1\|_{\dot{B}^{\frac{N}{p}+1}_{p,1}} \nonumber\\
    &\leq&  C \left(1+\| (b^1,b^2 ) \|_{\dot{B}^{\frac{N}{2}}_{2,1}}\right)\|\delta b \|_{\dot{B}^{\frac{N}{2}-1}_{2,1}}
\|   u^1\|_{\dot{B}^{\frac{N}{p}+1}_{p,1}}.\label{5.8-Ng3}
    \end{eqnarray}
The estimates (\ref{5.3-Ng3})--(\ref{5.8-Ng3}) imply that
	\begin{eqnarray}
     && \|\Pe^\bot\delta u\|_{\widetilde{L}^2_t(\dot{B}^{\frac{N}{2}-1}_{2,1})\cap {L}^1_t(\dot{B}^{\frac{N}{2}}_{2,1})}\nonumber\\
      &\leq& C\int^t_0\left\{\|{  \delta b}\|_{\dot{B}^{\frac{N}{2}-1}_{2,1}}\left(1+\| (b^1,b^2 ) \|_{\dot{B}^{\frac{N}{2}}_{2,1}}\right)\left(1+
\|   u^1\|_{\dot{B}^{\frac{N}{p}+1}_{p,1}}\right)\right.\nonumber\\
&&\left.
      +\left( \|(\Pe^\bot u^1,\Pe^\bot u^2)\|_{\dot{B}^{\frac{N}{2}}_{2,1}}
+\|(\Pe u^1,\Pe u^2)\|_{\dot{B}^{\frac{N}{p}}_{p,1}}\right)
\left(\|\Pe \delta u\|_{\dot{B}^{\frac{N}{p}-1}_{p,1}}+\|\Pe^\bot \delta u\|_{\dot{B}^{\frac{N}{2}-1}_{2,1}}\right)
\right.\nonumber\\
&&\left.+ \left(\|\Pe u^2\|_{\dot{B}^{\frac{N}{p}-1}_{p,1}}+\| b^2\|_{\dot{B}^{\frac{N}{2}}_{2,1}}\right)\left(\|\Pe^\bot\delta u\|_{\dot{B}^{\frac{N}{2}}_{2,1}}
+\|\Pe \delta u\|_{\dot{B}^{\frac{N}{p}}_{p,1}}\right)
      \right\}ds.\label{5.9-Ng3}
    \end{eqnarray}
Next, using similar arguments as in the proof of \eqref{5.9-Ng3}, Corollary \ref{coro-product}  and the embedding $\dot{B}^{\fr{N}{2}-2}_{2,1}\hookrightarrow\dot{B}^{\fr{N}{p}-2}_{p,1}$ for $p>2$, we get
    \begin{eqnarray}
    &&  \|\Pe\delta u\|_{\widetilde{L}^2_t(\dot{B}^{\frac{N}{p}-1}_{p,1})\cap {L}^1_t(\dot{B}^\frac{N}{p}_{p,1})}\nonumber\\
      &\leq&C \int^t_0\left(\|{\nb \delta b}\|_{\dot{B}^{\frac{N}{p}-2}_{p,1}}
      +\|u^2\cdot \nabla \delta u\|_{\dot{B}^{\frac{N}{p}-2}_{p,1}}+\|\delta u\cdot \nabla u^1\|_{\dot{B}^{\frac{N}{p}-2}_{p,1}}
      \right.\nonumber\\
      &&\left.+\|I( b^2)\mathcal{A}\delta u\|_{\dot{B}^{\frac{N}{p}-2}_{p,1}}
      +\|[I( b^2)-I(b^1)]\mathcal{A} u^1\|_{\dot{B}^{\frac{N}{p}-2}_{p,1}}
      \right)ds\nonumber\\
        &\leq& C\int^t_0\left\{\|{  \delta b}\|_{\dot{B}^{\frac{N}{2}-1}_{2,1}}\left(1+\| (b^1,b^2 ) \|_{\dot{B}^{\frac{N}{2}}_{2,1}}\right)\left(1+
\|   u^1\|_{\dot{B}^{\frac{N}{p}+1}_{p,1}}\right)\right.\nonumber\\
&&\left.
      +\left( \|(\Pe^\bot u^1,\Pe^\bot u^2)\|_{\dot{B}^{\frac{N}{2}}_{2,1}}
+\|(\Pe u^1,\Pe u^2)\|_{\dot{B}^{\frac{N}{p}}_{p,1}}\right)
\left(\|\Pe \delta u\|_{\dot{B}^{\frac{N}{p}-1}_{p,1}}+\|\Pe^\bot \delta u\|_{\dot{B}^{\frac{N}{2}-1}_{2,1}}\right)
\right.\nonumber\\
&&\left.+ \left(\|\Pe u^2\|_{\dot{B}^{\frac{N}{p}-1}_{p,1}}+\| b^2\|_{\dot{B}^{\frac{N}{2}}_{2,1}}\right)\left(\|\Pe^\bot\delta u\|_{\dot{B}^{\frac{N}{2}}_{2,1}}
+\|\Pe \delta u\|_{\dot{B}^{\frac{N}{p}}_{p,1}}\right)
      \right\}ds.\label{5.10-Ng3}
    \end{eqnarray}
By virtue the  interpolation inequality and H\"{o}lder's inequality , we obtain
      \begin{equation}\label{5.20}
         \|\Pe^\bot u^i_L\|_{\widetilde{L}^2_T(\dot{B}^{\frac{N}{2}}_{2,1})}\leq C  \|\Pe^\bot u^i_L\|_{\widetilde{L}^\infty_T(\dot{B}^{\frac{N}{2}-1+\al}_{2,1})}^{\frac{1+\al}{2}}
        \|\Pe^\bot u^i_L\|_{{L}^1_T(\dot{B}^{\frac{N}{2}+1+\al}_{2,1})}^{\frac{1-\al}{2}}T^{\frac{\al}{2}}\leq CX(T)T^{\frac{\al}{2}}\leq CX^0T^{\frac{\al}{2}},
      \end{equation}
and
	\begin{equation}\label{5.21}
         \|\Pe^\bot u^i_L\|_{{L}^1_T(\dot{B}^{\frac{N}{2}+1}_{2,1})}\leq C  \|\Pe^\bot u^i_L\|_{\widetilde{L}^\infty_T(\dot{B}^{\frac{N}{2}-1+\al}_{2,1})}^{\frac{\al}{2}}
        \|\Pe^\bot u^i_L\|_{{L}^1_T(\dot{B}^{\frac{N}{2}+1+\al}_{2,1})}^{1-\frac{\al}{2}}T^{\frac{\al}{2}}\leq CX(T)T^{\frac{\al}{2}}\leq CX^0T^{\frac{\al}{2}},
      \end{equation}
for $i=1, 2$. Combining the above estimates, (\ref{uniform1}), (\ref{5.15-Ng3}) and (\ref{5.9-Ng3})--(\ref{5.10-Ng3}), choosing $T=1$, we have for all $t\in[0,1]$
  	$$
\|\delta b\|_{L^\infty_t(\dot{B}^{\frac{N}{2}-1}_{2,1})}\leq C(1+X^0)\left(\|\Pe^\bot \delta u\|_{L^1_t(\dot{B}^\frac{N}{2}_{2,1})}+\|\Pe \delta u\|_{L^1_t(\dot{B}^\frac{N}{p}_{p,1})}\right)
+CX^0\|\delta b\|_{L^\infty_t(\dot{B}^{\frac{N}{2}-1}_{2,1})},
	$$
and
\begin{eqnarray*}
     && \|\Pe^\bot\delta u\|_{\widetilde{L}^2_t(\dot{B}^{\frac{N}{2}-1}_{2,1})\cap {L}^1_t(\dot{B}^{\frac{N}{2}}_{2,1})}+ \|\Pe\delta u\|_{\widetilde{L}^2_t(\dot{B}^{\frac{N}{p}-1}_{p,1})\cap {L}^1_t(\dot{B}^\frac{N}{p}_{p,1})}\nonumber\\
      &\leq& C(1+X^0)\int^t_0\|{  \delta b}\|_{\dot{B}^{\frac{N}{2}-1}_{2,1}}ds+C(1+X^0)X_0\|{  \delta b}\|_{L^\infty_t(\dot{B}^{\frac{N}{2}-1}_{2,1})}\nonumber\\
&&+CX^0\left(\|\Pe^\bot\delta u\|_{\widetilde{L}^2_t(\dot{B}^{\frac{N}{2}-1}_{2,1})\cap {L}^1_t(\dot{B}^{\frac{N}{2}}_{2,1})}+ \|\Pe\delta u\|_{\widetilde{L}^2_t(\dot{B}^{\frac{N}{p}-1}_{p,1})\cap {L}^1_t(\dot{B}^\frac{N}{p}_{p,1})}\right).
    \end{eqnarray*}
When $CX^0\leq\frac{1}{2}$, the above inequalities reduce to
\beq\label{dlb}
\|\delta b\|_{L^\infty_t(\dot{B}^{\frac{N}{2}-1}_{2,1})}\leq 2C(1+X^0)\left(\|\Pe^\bot \delta u\|_{L^1_t(\dot{B}^\frac{N}{2}_{2,1})}+\|\Pe \delta u\|_{L^1_t(\dot{B}^\frac{N}{p}_{p,1})}\right),
\eeq
and
\beq\label{dlu}
&& \|\Pe^\bot\delta u\|_{\widetilde{L}^2_t(\dot{B}^{\frac{N}{2}-1}_{2,1})\cap {L}^1_t(\dot{B}^{\frac{N}{2}}_{2,1})}+ \|\Pe\delta u\|_{\widetilde{L}^2_t(\dot{B}^{\frac{N}{p}-1}_{p,1})\cap {L}^1_t(\dot{B}^\frac{N}{p}_{p,1})}\nonumber\\
      &\leq& 2C(1+X^0)(t+X_0)\|{  \delta b}\|_{L^\infty_t(\dot{B}^{\frac{N}{2}-1}_{2,1})}.
\eeq
Substituting \eqref{dlb} into \eqref{dlu}, we are led to
\beq\label{dlu1}
&& \|\Pe^\bot\delta u\|_{\widetilde{L}^2_t(\dot{B}^{\frac{N}{2}-1}_{2,1})\cap {L}^1_t(\dot{B}^{\frac{N}{2}}_{2,1})}+ \|\Pe\delta u\|_{\widetilde{L}^2_t(\dot{B}^{\frac{N}{p}-1}_{p,1})\cap {L}^1_t(\dot{B}^\frac{N}{p}_{p,1})}\nonumber\\
      \nn&\leq& 4C^2(1+X^0)^2(t+X_0)\left(\|\Pe^\bot \delta u\|_{L^1_t(\dot{B}^\frac{N}{2}_{2,1})}+\|\Pe \delta u\|_{L^1_t(\dot{B}^\frac{N}{p}_{p,1})}\right)\\
      &\leq& (8C^2+2)(t+X_0)\left(\|\Pe^\bot \delta u\|_{L^1_t(\dot{B}^\frac{N}{2}_{2,1})}+\|\Pe \delta u\|_{L^1_t(\dot{B}^\frac{N}{p}_{p,1})}\right).
\eeq
Accordingly, we conclude that if $X^0\le\fr{1}{4(8C^2+2)}$, then $\delta b=\delta u=0$ for all $t\in[0,T_0]$ with $T_0:=\fr{1}{4(8C^2+2)}$. Repeat this argument on $[T_0, 2T_0], [2T_0, 3T_0], \cdots$, we can easily prove that $(b^1,u^1)=(b^2,u^2)$ for all $t\geq0$.
The proof of Theorem \ref{thm-p>2} when $N\geq3$ is completed.

\subsection{The uniqueness  when $N=2$.}
Finally, we  prove the uniqueness part of Theorem \ref{thm-p>2} when $N=2$.
To this end, we give the following lemma with additional assumption on the initial data.
\begin{lem}\label{lem-add}
  Under the assumptions in Theorem \ref{thm-p>2} and $N=2$, in addition, if $\Pe u_0\in \dot{B}^{0}_{2,1}$, then for all $T>0$, we have
    \begin{equation}
      \Pe u\in \widetilde{L}^\infty_T(\dot{B}^0_{2,1})\cap L^1_T(\dot{B}^2_{2,1}),\label{5.1-p}
    \end{equation}
with
         \be\label{add}
          \| \Pe u\|_{\widetilde{L}^\infty_T(\dot{B}^0_{2,1})\cap L^1_T(\dot{B}^2_{2,1})}\leq C  \| \Pe u_0\|_{ \dot{B}^0_{2,1} }+
          C\left(X^0\right)^2(1+T^\frac{\al}{2}).
        \ee
\end{lem}
\begin{proof}
    According to Proposition \ref{prop3.4}, we get
        \begin{equation}\label{5.3-p}
          \| \Pe u\|_{\widetilde{L}^\infty_T(\dot{B}^0_{2,1})\cap L^1_T(\dot{B}^2_{2,1})}\leq C  \| \Pe u_0\|_{ \dot{B}^0_{2,1} }+
        C  \|\Pe\left(u\cdot\nb u+I( b)\mathcal{A}u\right)\|_{L^1_T(\dot{B}^0_{2,1})}.
        \end{equation}
    From (\ref{4.63}), we need to bound $\|\Pe\left(u\cdot\nb u\right)\|_{L^1_T(\dot{B}^0_{2,1})}$ by the following three terms,
        \begin{equation}\label{5.4-p}
          \|\Pe\left(u\cdot\nb u\right)\|_{L^1_T(\dot{B}^0_{2,1})}
          \leq \|\Pe\left(\Pe u\cdot\nb \Pe u\right)\|_{L^1_T(\dot{B}^0_{2,1})}
          +\|\Pe\left(\Pe u\cdot\nb \Pe^\bot u\right)\|_{L^1_T(\dot{B}^0_{2,1})}
        +  \|\Pe\left(\Pe^\bot u\cdot\nb\Pe u\right)\|_{L^1_T(\dot{B}^0_{2,1})}.
        \end{equation}
Using the estimate (\ref{uniform1}) and Propositions \ref{prop-classical}--\ref{p-TR}, the above three terms can be estimated as follows,
    \begin{eqnarray}
      &&\|\Pe\left(\Pe u\cdot\nb \Pe u\right)\|_{L^1_T(\dot{B}^0_{2,1})}=\|\Pe\mathrm{div}\left(\Pe u\otimes\Pe u\right)\|_{L^1_T(\dot{B}^0_{2,1})}\nonumber\\
        &\leq&C\| \Pe u\otimes\Pe u \|_{L^1_T(\dot{B}^1_{2,1})} \leq C\| \Pe u\otimes\Pe u \|_{L^1_T(\dot{B}^\frac{4}{p}_{\frac{p}{2},1})}\nonumber\\
            &\leq& C\|\Pe u\|_{\widetilde{L}^\infty_T(\dot{B}^{\frac{2}{p}-1}_{p,1})}\|\Pe u\|_{L^1_T(\dot{B}^{\frac{2}{p}+1}_{p,1})}
            \leq CX^2(T)\leq C\left(X^0\right)^2,
    \end{eqnarray}
        \begin{eqnarray}
      &&\|\Pe\left(\Pe u\cdot\nb \Pe^\bot u\right)\|_{L^1_T(\dot{B}^0_{2,1})}  \nonumber\\
       &\leq& C\| \dot{T}_{\Pe u}\nb\Pe^\bot u \|_{L^1_T(\dot{B}^0_{2,1})}+C\|\dot{T}_{\nb\Pe^\bot u}\Pe u \|_{L^1_T(\dot{B}^0_{2,1})}
       +C\|\dot{R}({\Pe u},\nb\Pe^\bot u) \|_{L^1_T(\dot{B}^{\frac{2}{p}}_{\frac{2p}{p+2},1})}\nonumber\\
           \nn &\leq& C\|\Pe u\|_{\widetilde{L}^2_T(\dot{B}^{0}_{\infty,1})}\|\nb\Pe^\bot u\|_{\widetilde{L}^2_T(\dot{B}^{0}_{2,1})}
            + C\|\nb\Pe^\bot u\|_{\widetilde{L}^2_T(\dot{B}^{-\frac{2}{p}}_{\frac{2p}{p-2},1})}\|\Pe  u\|_{\widetilde{L}^2_T(\dot{B}^{\frac{2}{p}}_{p,1})}\\
            &&+C\|\nb\Pe^\bot u\|_{\widetilde{L}^2_T(\dot{B}^{ 0}_{2,1})}\|\Pe  u\|_{\widetilde{L}^2_T(\dot{B}^{\frac{2}{p}}_{p,1})}\nonumber\\
            &\leq& CX (T)\left(\| \Pe^\bot u_L\|_{\widetilde{L}^2_T(\dot{B}^{1 }_{2,1})}+\| \Pe^\bot u_H\|_{\widetilde{L}^2_T(\dot{B}^{1 }_{2,1})}\right)\nn\\
            &\leq& CX^0\left(X^0+\| \Pe^\bot u_L\|_{\widetilde{L}^2_T(\dot{B}^{ 1 }_{2,1})}\right),
    \end{eqnarray}
    and
        \begin{eqnarray}
      &&\|\Pe\left(\Pe^\bot u\cdot\nb \Pe  u\right)\|_{L^1_T(\dot{B}^0_{2,1})}  \nonumber\\
       \nn&\leq& C\| \dot{T}_{\Pe^\bot u}\nb\Pe u \|_{L^1_T(\dot{B}^0_{2,1})}
       +C\|\dot{T}_{\nb\Pe u}\Pe^\bot u \|_{L^1_T(\dot{B}^0_{2,1})}
       +C\|\dot{R}({\Pe^\bot u},\nb\Pe u) \|_{L^1_T(\dot{B}^{\frac{2}{p}}_{\frac{2p}{p+2},1})}\nonumber\\
            &\leq& C\|\Pe^\bot u\|_{\widetilde{L}^{p}_T(\dot{B}^{0}_{\fr{2p}{p-2},1})}\|\nb\Pe  u\|_{\widetilde{L}^\fr{p}{p-1}_T(\dot{B}^{0}_{p,1})}
            + C\|\nb\Pe  u\|_{\widetilde{L}^2_T(\dot{B}^{-1}_{\infty,1})}\|\Pe^\bot  u\|_{\widetilde{L}^2_T(\dot{B}^{1}_{2,1})}\nn\\
            &&+C\|\Pe^\bot u\|_{\widetilde{L}^2_T(\dot{B}^{ 1}_{2,1})}\|\nb\Pe  u\|_{\widetilde{L}^2_T(\dot{B}^{\frac{2}{p}-1}_{p,1})}\nonumber\\
            \nn&\leq& CX (T)\left(\| \Pe^\bot u\|_{\widetilde{L}^2_T(\dot{B}^{ 1 }_{2,1})}+\|\Pe^\bot u\|_{\widetilde{L}^{p}_T(\dot{B}^{\frac{2}{p}}_{2,1})}
            \right)\\
            &\leq& CX^0\left(X^0+\| \Pe^\bot u_L\|_{\widetilde{L}^2_T(\dot{B}^{ 1 }_{2,1})}+\|\Pe^\bot u_L\|_{\widetilde{L}^{p}_T(\dot{B}^{\frac{2}{p}}_{2,1})}
            \right).
    \end{eqnarray}
Using the  interpolation inequality, we obtain
      \begin{equation}
         \|\Pe^\bot u_L\|_{\widetilde{L}^2_T(\dot{B}^{1}_{2,1})}\leq C  \|\Pe^\bot u_L\|_{\widetilde{L}^\infty_T(\dot{B}^{\al}_{2,1})}^{\frac{1+\al}{2}}
        \|\Pe^\bot u_L\|_{{L}^1_T(\dot{B}^{2+\al}_{2,1})}^{\frac{1-\al}{2}}T^{\frac{\al}{2}}\leq CX(T)T^{\frac{\al}{2}}\leq CX^0T^{\frac{\al}{2}},
      \end{equation}
       and
          \begin{equation}
         \|\Pe^\bot u_L\|_{\widetilde{L}^{p}_T(\dot{B}^{\frac{2}{p}}_{2,1})}\leq C  \|\Pe^\bot u_L\|_{\widetilde{L}^\infty_T(\dot{B}^{\al}_{2,1})}^{\frac{2p+\al p-2}{2p}}
        \|\Pe^\bot u_L\|_{{L}^1_T(\dot{B}^{2+\al}_{2,1})}^{\frac{2-\al p}{2p}}T^{\frac{\al}{2}}\leq CX(T)T^{\frac{\al}{2}}\leq CX^0T^{\frac{\al}{2}}.\label{5.9}
      \end{equation}
The above estimates (\ref{5.4-p})--(\ref{5.9}) imply that
     \begin{equation}\label{5.10}
          \|\Pe\left(u\cdot\nb u\right)\|_{L^1_T(\dot{B}^0_{2,1})}
          \leq C\left(X^0\right)^2 (1+T^\frac{\al}{2}).
        \end{equation}
       Next, using \eqref{4.27'} and Corollary \ref{coro-product}, one easily  deduces that
            \begin{eqnarray}
           \nn  && \|\Pe\left( I( b)\mathcal{A}u\right)\|_{L^1_T(\dot{B}^0_{2,1})}\\
           \nn&\le&C\| I( b)\mathcal{A}\Pe^\bot u\|_{L^1_T(\dot{B}^0_{2,1})}+C\| I( b)\mathcal{A}\Pe u\|_{L^1_T(\dot{B}^0_{2,1})}\\
            &\leq& CX^2(T)+C\|  I( b)\|_{\widetilde{L}^\infty_T(\dot{B}^1_{2,1})}
              \|\mathcal{A}\Pe u \|_{L^1_T(\dot{B}^0_{2,1})}\nonumber\\
                &\leq&C\left(X^0\right)^2+C\|b\|_{\widetilde{L}^\infty_T(\dot{B}^1_{2,1})}
              \|\Pe u \|_{L^1_T(\dot{B}^2_{2,1})}\nn\\
              &\leq&  CX^0\left(X^0+\| \Pe u \|_{L^1_T(\dot{B}^2_{2,1})}\right). \label{5.11}
            \end{eqnarray}
From (\ref{5.3-p}), (\ref{5.10}) and (\ref{5.11}), we have
     \begin{equation}
          \| \Pe u\|_{\widetilde{L}^\infty_T(\dot{B}^0_{2,1})\cap L^1_T(\dot{B}^2_{2,1})}\leq C  \| \Pe u_0\|_{ \dot{B}^0_{2,1} }+
          C\left(X^0\right)^2 (1+T^\frac{\al}{2})+CX^0\|\Pe u \|_{L^1_T(\dot{B}^2_{2,1})}.
        \end{equation}
Consequently, \eqref{add} holds if $CX^0\leq \frac{1}{2}$. This completes the proof of Lemma \ref{lem-add}.
\end{proof}

Now we turn to prove the uniqueness of  solutions for $N=2$ with the additional assumption that $\Pe u_0\in \dot{B}^0_{2,1}$. In fact, for any fixed $T>0$, from \eqref{5.20}-- \eqref{5.21},  Lemma \ref{lem-add} and \eqref{uniform1}, we infer that
\beq\label{u12}
\|u^i\|_{L^1_T(\dot{B}^{2}_{2,1})\cap \widetilde{L}^2_T(\dot{B}^{1}_{2,1})}\nn&\le&\|\Pe^\bot u^i_L\|_{L^1_T(\dot{B}^{2}_{2,1})\cap \widetilde{L}^2_T(\dot{B}^{1}_{2,1})}+\|\Pe^\bot u^i_H\|_{L^1_T(\dot{B}^{2}_{2,1})\cap \widetilde{L}^2_T(\dot{B}^{1}_{2,1})}+\|\Pe u^i\|_{L^1_T(\dot{B}^{2}_{2,1})\cap \widetilde{L}^2_T(\dot{B}^{1}_{2,1})}\\
 &\le&\nn C \left(X(T)T^\fr{\al}{2}+X(T)+\| \Pe u_0\|_{ \dot{B}^0_{2,1} }+
          \left(X^0\right)^2(1+T^\frac{\al}{2})\right)\\
 &\le&C\left(X^0+\| \Pe u_0\|_{ \dot{B}^0_{2,1} }+
          X^0\left(1+X^0\right)(1+T^\frac{\al}{2})\right),
\eeq
for $i=1, 2$. On this basis, using Propositions \ref{prop3.3}--\ref{prop3.4}, and \ref{prop-pro},  the estimate \eqref{uniform1}, we are led to
    \begin{eqnarray}
     && \|\delta b(t)\|_{\dot{B}^0_{2,\infty}}
       \leq  e^{CV_2(t)}\int^t_0
      \|-{\dv   \delta u} -\delta u\cdot\nabla b^1-\delta b\mathrm{div}u^2-b^1\mathrm{div}\delta u\|_{\dot{B}^0_{2,\infty}}
      ds\nonumber\\
      &\leq &Ce^{CV_2(t)}\int^t_0
      \left(\|  \delta u\|_{\dot{B}^1_{2,\infty}}+\|\delta u\|_{\dot{B}^1_{2,1}}\|  b^1\|_{\dot{B}^{1}_{2,1}}
      +\| \delta b\|_{\dot{B}^{0}_{2,\infty}}\| \mathrm{div}u^2\|_{\dot{B}^1_{2,1}}
      \right)
      ds\nn\\
      &\leq &Ce^{CV_2(t)}\int^t_0
      \left(\left(1+X^0\right)\|  \delta u\|_{\dot{B}^1_{2,1}}      +\| \delta b\|_{\dot{B}^{0}_{2,\infty}}\| u^2\|_{\dot{B}^2_{2,1}}
      \right)
      ds          \label{5.15}
    \end{eqnarray}
where $V_2(t)=\int^t_0\|\nabla u^2(\tau)\|_{\dot{B}^{1}_{2,1}}d\tau$,
and
    \begin{eqnarray}
      \nn&&\|\delta u\|_{\widetilde{L}^2_t(\dot{B}^0_{2,\infty})\cap \widetilde{L}^1_t(\dot{B}^1_{2,\infty})}\\
      &\leq& \int^t_0\left(\|{\nb \delta b}\|_{\dot{B}^{-1}_{2,\infty}}
      +\|u^2\cdot \nabla \delta u\|_{\dot{B}^{-1}_{2,\infty}}+\|\delta u\cdot \nabla u^1\|_{\dot{B}^{-1}_{2,\infty}}+\|K( b^2){\nb b^2}-K( b^1){\nb b^1}\|_{\dot{B}^{-1}_{2,\infty}}
      \right.\nonumber\\
      &&\left.+\|I( b^2)\mathcal{A}\delta u\|_{\dot{B}^{-1}_{2,\infty}}
      +\|[I( b^2)-I(b^1)]\mathcal{A} u^1\|_{\dot{B}^{-1}_{2,\infty}}
      \right)ds\nonumber\\
        \nn&\leq&C\|(u^1,u^2)\|_{\widetilde{L}^2_t(\dot{B}^1_{2,1})}\|\delta u\|_{\widetilde{L}^2(\dot{B}^0_{2,\infty})}+C\|b^2\|_{\widetilde{L}^\infty_t(\dot{B}^1_{2,1})} \|\delta u\|_{\widetilde{L}^1(\dot{B}^1_{2,\infty})}\\
        \nn&&+C\left(1+\| (b^1,b^2 ) \|_{\widetilde{L}^\infty_t(\dot{B}^{1}_{2,1})}\right)\int^t_0
       \left(1+\|u^1\|_{\dot{B}^2_{2,1}}\right)\|\delta b \|_{\dot{B}^0_{2,\infty}}ds\\
       \nn&\leq&C\|(u^1,u^2)\|_{\widetilde{L}^2_t(\dot{B}^1_{2,1})}\|\delta u\|_{\widetilde{L}^2(\dot{B}^0_{2,\infty})}+CX^0 \|\delta u\|_{\widetilde{L}^1(\dot{B}^1_{2,\infty})}\\
        &&+C\left(1+X^0\right)\int^t_0
       \left(1+\|u^1\|_{\dot{B}^2_{2,1}}\right)\|\delta b \|_{\dot{B}^0_{2,\infty}}ds,
       \label{5.16}
    \end{eqnarray}
where we have used the estimates
    \begin{eqnarray}
    &&  \|K( b^2){\nb b^2}-K( b^1){\nb b^1}\|_{\dot{B}^{-1}_{2,\infty}}\nonumber\\
    &\leq &C\left\|\int^1_0  K (b^1+\tau(b^2-b^1))d\tau \delta b\right\|_{\dot{B}^{0}_{2,\infty}}
    \nonumber\\
    &\leq&C\left\|\int^1_0  K (b^1+\tau(b^2-b^1))d\tau\right\|_{\dot{B}^{1}_{2,1}}\|\delta b \|_{\dot{B}^0_{2,\infty}} \nonumber\\
    &\leq&C \| (b^1,b^2 ) \|_{\dot{B}^{1}_{2,1}}\|\delta b \|_{\dot{B}^0_{2,\infty}},
    \end{eqnarray}
and
     \begin{eqnarray}
    &&\| I(b^2)-I(b^1) \|_{\dot{B}^{0}_{2,\infty}}=C\left\|\int^1_0  I' (b^1+\tau(b^2-b^1))d\tau \delta b\right\|_{\dot{B}^{0}_{2,\infty}}
    \nonumber\\
    &\leq&C\left(1+\left\|\int^1_0  I' (b^1+\tau(b^2-b^1))d\tau-1\right\|_{\dot{B}^{1}_{2,1}}\right)\|\delta b \|_{\dot{B}^0_{2,\infty}} \nonumber\\
    &\leq&C\left(1+ \| (b^1,b^2 ) \|_{\dot{B}^{1}_{2,1}}\right)\|\delta b \|_{\dot{B}^0_{2,\infty}}.
    \end{eqnarray}
Choosing $X^0$ and $t\le\bar{T}$ so  small that  the first two terms on the right hand side of \eqref{5.16} can be absorbed by the left hand side, then \eqref{5.16} reduces to
\be\label{5.16'}
\|\delta u\|_{\widetilde{L}^2_t(\dot{B}^0_{2,\infty})\cap \widetilde{L}^1_t(\dot{B}^1_{2,\infty})}
      \leq C(X^0, \bar{T})\int^t_0
       \left(1+\|u^1\|_{\dot{B}^2_{2,1}}\right)\|\delta b \|_{\dot{B}^0_{2,\infty}}ds,
\ee
where  $C(X^0,\bar{T})$ denotes the various constants depending on $X^0$ and $\bar{T}$. Thanks to \eqref{u12}, applying Gronwall's lemma to \eqref{5.15}, we find that for all $t\in[0,\bar{T}]$,
 \be\label{dlb1}
       \|\delta b\|_{L^\infty_t(\dot{B}^0_{2,\infty})} \\
      \leq C(X^0,\bar{T})
      \|  \delta u\|_{L^1_t(\dot{B}^1_{2,1})},
   \ee
From Proposition 2.8 in \cite{Danchin2003},  we have
   \be\label{log-inter}
    \|\delta u\|_{\widetilde{L}^1_s(\dot{B}^1_{2,1})}\leq C\|\delta u\|_{\widetilde{L}^1_s(\dot{B}^1_{2,\infty})}
    \ln\left(e+\frac{\|\delta u\|_{\widetilde{L}^1_s(\dot{B}^\al_{2,\infty})}+\|\delta u\|_{\widetilde{L}^1_s(\dot{B}^{2-\al}_{2,\infty})}}{\|\delta u\|_{\widetilde{L}^1_s(\dot{B}^1_{2,\infty})}}
    \right).
    \ee
Substituting \eqref{dlb1}--\eqref{log-inter} into \eqref{5.16'}, we obtain
    $$
      \|\delta u\|_{ \widetilde{L}^1_t(\dot{B}^1_{N,\infty})}
      \leq
      C \int^t_0  \left(1+\|u^1\|_{\dot{B}^2_{2,1}}\right)\|\delta u\|_{ \widetilde{L}^1_s(\dot{B}^1_{N,\infty})}\ln\left(e+V_3(s){\|\delta u\|^{-1}_{\widetilde{L}^1_s(\dot{B}^1_{N,\infty})}}
    \right)
        ds,
    $$
where
$$
V_3(t):=\|\delta u\|_{\widetilde{L}^1_t(\dot{B}^\al_{2,\infty})}+\|\delta u\|_{\widetilde{L}^1_t(\dot{B}^{2-\al}_{2,\infty})}.
$$
For all $t\in[0,\bar{T}]$, by H\"{o}lder's inequality and interpolations, there hold
\beqno
\|u^i\|_{L^1_t(\dot{B}^\al_{2,1})}&\le&\|\Pe^\bot u^i_L\|_{L^1_t(\dot{B}^\al_{2,1})}+\|\Pe^\bot u^i_H+\Pe u^i\|_{L^1_t(\dot{B}^\al_{2,1})}\\
&\le&C\bar{T}\|\Pe^\bot u^i_L\|_{L^\infty_t(\dot{B}^\al_{2,1})}+C\bar{T}^{1-\fr{\al}{2}}\|\Pe^\bot u^i_H+\Pe u^i\|^{{1-\fr{\al}{2}}}_{L^\infty_t(\dot{B}^0_{2,1})}\|\Pe^\bot u^i_H+\Pe u^i\|^\fr{\al}{2}_{L^1_t(\dot{B}^2_{2,1})}\\
&\le&C(X^0, \bar{T}),
\eeqno
and
\beqno
\|u^i\|_{L^1_t(\dot{B}^{2-\al}_{2,1})}&\le&\|\Pe^\bot u^i_L\|_{L^1_t(\dot{B}^{2-\al}_{2,1})}+\|\Pe^\bot u^i_H+\Pe u^i\|_{L^1_t(\dot{B}^{2-\al}_{2,1})}\\
&\le&C\|\Pe^\bot u^i_L\|_{L^\infty_t(\dot{B}^\al_{2,1})}^\al\|\Pe^\bot u^i_L\|_{L^1_t(\dot{B}^\al_{2,1})}^{1-\al}+C\bar{T}^{\fr{\al}{2}}\|\Pe^\bot u^i_H+\Pe u^i\|^{{\fr{\al}{2}}}_{L^\infty_t(\dot{B}^0_{2,1})}\|\Pe^\bot u^i_H+\Pe u^i\|^\fr{1-\al}{2}_{L^1_t(\dot{B}^2_{2,1})}\\
&\le&C(X^0, \bar{T}).
\eeqno
These two inequalities imply that
\beno
V_3(t)\le C(X^0, \bar{T}), \quad \mathrm{for\ \  all} \quad t\in[0,\bar{T}].
\eeno
Since
    $$
        \int^1_0\frac{ds}{s\ln(e+V_3(\bar{T})s^{-1})}=+\infty,
    $$
Osgood's lemma implies that
    $\delta b=\delta u=0$ on $[0,\bar{T}]$. Standard arguments then yield that $(b^1,u^1)=(b^2,u^2)$ for all $t\geq0$. The proof of Theorem \ref{thm-p>2} is completed.
{\hfill
$\square$\medskip}

\section{Proof of Theorem \ref{thm-p=2}}
\noindent To simplify the presentation, for $T>0$, let us denote
\begin{gather*}
Z_L(T):=\|(b_L, \Pe^\bot u_L)\|_{\widetilde{L}^\infty_T(\dot{B}^{\fr{N}{2}-1}_{2,1})\cap L^1_T(\dot{B}^{\fr{N}{2}+1}_{2,1})},\quad Z_L^0:=\|(b_{0L}, \Pe^\bot u_{0L})\|_{\dot{B}^{\fr{N}{2}-1}_{2,1}},\\
H(T):=\|\Pe u\|_{\widetilde{L}^\infty_T(\dot{B}^{\fr{N}{2}-1}_{2,1})\cap{L}^1_T(\dot{B}^{\fr{N}{2}+1}_{2,1})}, \quad H^0:=\|\Pe u\|_{\dot{B}^{\fr{N}{2}-1}_{2,1}},\\
Z(T):=\|(b, u)\|_{\mathcal{E}^{\fr{N}{2}}(T)}=Z_L(T)+X_H(T)+H(T), \quad Z^0:=Z_L^0+X_H^0+H^0.
\end{gather*}
Now we are in a position to prove Theorem \ref{thm-p=2}. On the one hand, from \eqref{1.37} and the embedding
\be\label{emb}
\|\Pe u_0\|_{\dot{B}^{\fr{N}{p}-1}_{p,1}}\le C\|\Pe u_0\|_{\dot{B}^{\fr{N}{2}-1}_{2,1}},
\ee
taking $c_1$ be any constant not larger than $\fr{c_0}{C}$, we infer that \eqref{initial3} implies \eqref{initial1}. Consequently, in view of Theorem \ref{thm-p>2}, there is a solution $(\rho, u)$ to the Navier-Stokes equations \eqref{CNS}. Moreover, using \eqref{1.37} and \eqref{emb} again, \eqref{uniform1} reduces to
\be\label{X1}
X(T)\le CC_0c_1, \quad \mathrm{for\ \ all}\quad T>0.
\ee
On the other hand, for the same initial data $(\rho_0, u_0)$, owing to Theorem \ref{local}, there exists a unique local solution $(\rho^*, u^*)$ in $\mathcal{E}^{\fr{N}{2}}(T^\ast)$, where $T^\ast$ is the maximal existence time   of $(\rho^*, u^*)$. By using the uniqueness of the solution, we conclude that
\beno
(\rho, u)\equiv(\rho^*, u^*), \quad \mathrm{for\ \ all}\quad t\in[0,T^*).
\eeno
Next, we go to bound $Z(T)$ for $T<T^*$. Since $X_{H}(T)$ has been estimated in Lemma \ref{lem4.11}, it suffices to dominate  $H(T)$ and $Z_L(T)$. To this end, using Proposition \ref{prop3.4}, \eqref{a5.2} and \eqref{a5.6} in the Appendix, we easily have
\be\label{H1}
H(T)\le H^0+CX(T)Z(T), \quad \mathrm{for\ \ all}\quad T<T^*.
\ee
To bound $Z_L(T)$, we follow the proof of Lemma \ref{lem4.10} line by line. Indeed,  replacing $\fr{N}{2}-1+\al$ by $\fr{N}{2}-1$ in \eqref{e1-d4L}, and using Lemmas \ref{lem-a4.1}--\ref{lem-a4.4} in the Appendix, it is not difficult to verify that
\be\label{ZL1}
Z_L(T)\le Z_L^0+CX(T)Z(T), \quad \mathrm{for\ \ all}\quad T<T^*.
\ee
Now from \eqref{XQH}, \eqref{H1}, \eqref{ZL1} and the fact that $X(T)\le CZ(T)$, we are led to
\be\label{Z1}
Z(T)\le Z^0+CX(T)Z(T), \quad \mathrm{for\ \ all}\quad T<T^*.
\ee
Combining \eqref{X1} with \eqref{Z1}, and choosing $c_1$ so small that
\be
C^2C_0c_1\le\fr12,
\ee
we conclude that
\be
Z(T)\le 2Z^0, \quad \mathrm{for\ \ all}\quad T<T^*.
\ee
This implies that the local solution $(\rho, u)$ can be extended to a global one. The proof of Theorem \ref{thm-p=2} is completed. {\hfill
$\square$\medskip}

\section{Appendix}
\subsection{Proof of Corollary \ref{coro-product}.}

 $ $

From Propositions \ref{prop-classical}--\ref{p-TR}, using the conditions $s_1-\frac{N}{p_1}\leq \min\{0,N(\frac{1}{p_2}-\frac{1}{\rho})\}$, $s=s_1+s_2+N(\frac{1}{\rho}-\frac{1}{p_1}-\frac{1}{p_2})$  and $\frac{1}{\rho}\leq\frac{1}{p_1}+\frac{1}{p_2}$, we have
	\begin{eqnarray}
	  \|\dot{T}_uv\|_{\dot{B}^s_{\rho,1}}&\leq& C\|u\|_{\dot{B}^{s-s_2}_{p_3,1}}\|v\|_{\dot{B}^{s_2}_{p_2,1}}\nonumber\\
	&\leq& C\|u\|_{\dot{B}^{s_1}_{p_1,1}}\|v\|_{\dot{B}^{s_2}_{p_2,1}},\ \textrm{ (when }\ p_2\geq \rho),\label{7.1}
	\end{eqnarray}
where $\frac{1}{\rho}=\frac{1}{p_2}+\frac{1}{p_3}$, and
	\begin{eqnarray}
	  \|\dot{T}_uv\|_{\dot{B}^s_{\rho,1}}&\leq& C\|u\|_{\dot{B}^{s-s_2+N(\frac{1}{p_2}-\frac{1}{\rho})}_{\infty,1}}\|v\|_{\dot{B}^{s_2-N(\frac{1}{p_2}-\frac{1}{\rho})}_{\rho,1}}\nonumber\\
	&\leq& C\|u\|_{\dot{B}^{s_1}_{p_1,1}}\|v\|_{\dot{B}^{s_2}_{p_2,1}},\ \textrm{ (when }\ p_2< \rho).
	\end{eqnarray}
Similarly, noting that $\sigma_1-\frac{N}{q_1}\leq \min\{0,N(\frac{1}{q_2}-\frac{1}{\rho})\}$, $s=\sigma_1+\sigma_2+N(\frac{1}{\rho}-\frac{1}{q_1}-\frac{1}{q_2})$  and $\frac{1}{\rho}\leq\frac{1}{q_1}+\frac{1}{q_2}$, we have
	\begin{eqnarray}
	  \|\dot{T}_vu\|_{\dot{B}^s_{\rho,1}}&\leq& C\|v\|_{\dot{B}^{s-\sigma_2}_{q_3,1}}\|u\|_{\dot{B}^{\sigma_2}_{q_2,1}}\nonumber\\
	&\leq& C\|v\|_{\dot{B}^{\sigma_1}_{q_1,1}}\|v\|_{\dot{B}^{\sigma_2}_{q_2,1}},\ \textrm{ (when }\ q_2\geq p),
	\end{eqnarray}
where $\frac{1}{\rho}=\frac{1}{q_2}+\frac{1}{q_3}$, and
	\begin{eqnarray}
	  \|\dot{T}_vu\|_{\dot{B}^s_{\rho,1}}&\leq& C\|v\|_{\dot{B}^{s-\sigma_2+N(\frac{1}{q_2}-\frac{1}{\rho})}_{\infty,1}}\|u\|_{\dot{B}^{\sigma_2-N(\frac{1}{q_2}-\frac{1}{\rho})}_{\rho,1}}\nonumber\\
	&\leq& C\|v\|_{\dot{B}^{\sigma_1}_{q_1,1}}\|u\|_{\dot{B}^{\sigma_2}_{q_2,1}},\ \textrm{ (when }\ q_2< \rho).
	\end{eqnarray}
Next, from Propositions \ref{prop-classical}--\ref{p-TR} , using the conditions   $s=s_1+s_2+N(\frac{1}{\rho}-\frac{1}{p_1}-\frac{1}{p_2})$,  $s_1+s_2>N\max \{0,\fr{1}{p_1}+\frac{1}{p_2}-1\}$,  and $\frac{1}{\rho}\leq\frac{1}{p_1}+\frac{1}{p_2}$, we have
	\begin{eqnarray}
	  \|\dot{R}(u,v)\|_{\dot{B}^s_{\rho,1}}&\leq& C \|\dot{R}(u,v)\|_{\dot{B}^{s+N(\frac{1}{p_1}+\frac{1}{p_2}-\frac{1}{\rho})}_{\frac{p_1p_2}{p_1+p_2},1}}\nonumber\\
	&\leq& C\|u\|_{\dot{B}^{s_1}_{p_1,1}}\|v\|_{\dot{B}^{s_2}_{p_2,1}},\ \textrm{ (when }\ \frac{1}{p_1}+\frac{1}{p_2}\leq1),
	\end{eqnarray}
and
    	\begin{eqnarray}
	  \|\dot{R}(u,v)\|_{\dot{B}^s_{\rho,1}}&\leq& C \|\dot{R}(u,v)\|_{\dot{B}^{s+N(1-\frac{1}{\rho})}_{1,1}}\nonumber\\
&\leq&   C\|u\|_{\dot{B}^{s-s_2+N(1-\frac{1}{\rho})}_{p_4,1}}\|v\|_{\dot{B}^{s_2}_{p_2,1}}\nonumber\\
	&\leq& C\|u\|_{\dot{B}^{s_1}_{p_1,1}}\|v\|_{\dot{B}^{s_2}_{p_2,1}},\ \textrm{ (when }\ \frac{1}{p_1}+\frac{1}{p_2}>1),\label{7.6}
	\end{eqnarray}
where $1=\frac{1}{p_2}+\frac{1}{p_4}$.  Combining \eqref{Bony-decom} with  (\ref{7.1})--(\ref{7.6}), we have (\ref{product1}). Then, we can easily obtain (\ref{product1-s}) and finish the proof of Corollary  \ref{coro-product}. {\hfill
$\square$\medskip}
\subsection{Action of smooth functions}
Here we give a variant of Theorem 2.61 in \cite{Bahouri-Chemin-Danchin11}, which will be used to deal with the nonlinear term stemming from the pressure $P=P(\rho)$.
\begin{lem}\label{lem-Kb}
Let $r\in[1,\infty], s\ge0$ and $T>0$. Assume that $u\in L^\infty_T(\dot{B}^{\fr{N}{2}}_{2,1})$ with $u_L\in\widetilde{L}^r_T(\dot{B}^{\fr{N}{2}+s}_{2,1})$, $u_H\in\widetilde{L}^r_T(\dot{B}^{\fr{N}{2}}_{2,1})$,  and $K$ is a smooth function on $\R$ which vanishes at 0. Then there hold
\be\label{App-1}
\|K(u)_L\|_{\widetilde{L}^r_T(\dot{B}^{\fr{N}{2}+s}_{2,1})}\\
\le C(K', \|u\|_{L^\infty_T(L^\infty)})\left(\|u_L\|_{\widetilde{L}^r_T(\dot{B}^{\fr{N}{2}+s}_{2,1})}+\|u_H\|_{\widetilde{L}^r_T(\dot{B}^{\fr{N}{2}}_{2,1})}\right),
\ee
and
\be\label{App-2}
\|K(u)_H\|_{\widetilde{L}^r_T(\dot{B}^{\fr{N}{2}}_{2,1})}\\
\le C(K', \|u\|_{L^\infty_T(L^\infty)})\left(\|u_L\|_{\widetilde{L}^r_T(\dot{B}^{\fr{N}{2}+s}_{2,1})}+\|u_H\|_{\widetilde{L}^r_T(\dot{B}^{\fr{N}{2}}_{2,1})}\right).
\ee
\end{lem}
\begin{proof}
In order to obtain \eqref{App-1}, we just need to modify the proof of Theorem 2.61 in \cite{Bahouri-Chemin-Danchin11}.  For the convenience of readers,  we give some details here. First of  all, using Meyer's first linearization method, we rewrite $K(u)$
as
\be\label{Kb}
K(u)=\sum_{q'\in\Z}m_{q'}\dot{\Dl}_{q'}u,
\ee
where
\beno
m_{q'}=\int_0^1K'(\dot{S}_{q'}u+\tau\dot{\Dl}_{q'}u)d\tau.
\eeno
The series in \eqref{Kb} converges to $K(u)$ in $L^\infty+L^2$, and $K(u)\in\dot{S}'_h$. In view of \eqref{Kb}, we have
\beq
\nn\|K(u)_L\|_{\widetilde{L}^{r}_T(\dot{B}^{\fr{N}{2}+s}_{2,1})}&\le& C\sum_{q<1}2^{q(\fr{N}{2}+s)}\|\dot{\Dl}_qK(u)\|_{L^{r}_T(L^2)}\\
\nn&\le& C\sum_{q<1}\sum_{q'>q}2^{q(\fr{N}{2}+s)}\|\dot{\Dl}_q(m_{q'}\dot{\Dl}_{q'}u)\|_{L^{r}_T(L^2)}\\
\nn&&+C\sum_{q<1}\sum_{q'\le q}2^{q(\fr{N}{2}+s)}\|\dot{\Dl}_q(m_{q'}\dot{\Dl}_{q'}u)\|_{L^{r}_T(L^2)}\\
&=:&I_1+I_2.
\eeq
Using the  H\"{o}lder's inequality and the convolution inequality, we have
\beq
I_1\nn&\le&C\sum_{q<1}\sum_{q'>q}2^{q(\fr{N}{2}+s)}\|m_{q'}\|_{L^\infty_T(L^\infty)}\|\dot{\Dl}_{q'}u\|_{L^{r}_T(L^2)}\\
\nn&\le&C(K',\|u\|_{L^\infty_T(L^\infty)})\sum_{q<1}\sum_{q'>q}2^{q(\fr{N}{2}+s)}\|\dot{\Dl}_{q'}u\|_{L^{r}_T(L^2)}\\
\nn&\le&C(K',\|u\|_{L^\infty_T(L^\infty)})\left(\sum_{q<1}\sum_{q'>q}2^{q(\fr{N}{2}+s)}\|\dot{\Dl}_{q'}u_L\|_{L^{r}_T(L^2)}+\sum_{q<1}\sum_{q'>q}2^{q\fr{N}{2}}\|\dot{\Dl}_{q'}u_H\|_{L^{r}_T(L^2)}\right)\\
&\le&C(K',\|u\|_{L^\infty_T(L^\infty)})\left(\|u_L\|_{\widetilde{L}^{r}_T(\dot{B}^{\fr{N}{2}+s}_{2,1})}+\|u_H\|_{\widetilde{L}^{r}_T(\dot{B}^{\fr{N}{2}}_{2,1})}\right),
\eeq
and
\beq
I_2\nn&\le&C\sum_{q<1}\sum_{q'\le q}\left(2^{q(\fr{N}{2}+s)}\|\dot{\Dl}_q(m_{q'}\dot{\Dl}_{q'}u_L)\|_{L^{r}_T(L^2)}+2^{q\fr{N}{2}}\|\dot{\Dl}_q(m_{q'}\dot{\Dl}_{q'}u_H)\|_{L^{r}_T(L^2)}\right)\\
\nn&\le&C\sum_{q<1}\sum_{q'\le q}\left(2^{q(\fr{N}{2}+s-[\fr{N}{2}+s]-1)}\sum_{|\beta|=[\fr{N}{2}+s]+1}\|\pr^\beta\dot{\Dl}_q(m_{q'}\dot{\Dl}_{q'}u_L)\|_{L^{r}_T(L^2)}\right.\\
\nn&&\left.+2^{q(\fr{N}{2}-[\fr{N}{2}]-1)}\sum_{|\beta|=[\fr{N}{2}]+1}\|\pr^\beta\dot{\Dl}_q(m_{q'}\dot{\Dl}_{q'}u_H)\|_{L^{r}_T(L^2)}\right)\\
\nn&\le&C(K',\|u\|_{L^\infty_T(L^\infty)})\sum_{q<1}\sum_{q'\le q}\left(2^{(q-q')(\fr{N}{2}+s-[\fr{N}{2}+s]-1)}\left(2^{q'(\fr{N}{2}+s)}\|\dot{\Dl}_{q'}u_L\|_{L^{r}_T(L^2)}\right)\right.\\
\nn&&\left.+2^{(q-q')(\fr{N}{2}-[\fr{N}{2}]-1)}\left(2^{q'\fr{N}{2}}\|\dot{\Dl}_{q'}u_H\|_{L^{r}_T(L^2)}\right)\right)\\
&\le&C(K',\|u\|_{L^\infty_T(L^\infty)})\left(\|u_L\|_{\widetilde{L}^{r}_T(\dot{B}^{\fr{N}{2}+s}_{2,1})}+\|u_H\|_{\widetilde{L}^{r}_T(\dot{B}^{\fr{N}{2}}_{2,1})}\right).
\eeq
The proof of \eqref{App-2} can be given in a similar way. This completes the proof of Lemma \ref{lem-Kb}.
\end{proof}

\subsection{Nonlinear estimates}
Here, we give the detail proofs of Lemmas \ref{lem4.1} in Section  \ref{S4}.

\textbf{Proof of Lemma \ref{lem4.1}.}

Clearly,
\beno
b\dv u=b\dv \mathbb{P}^\bot u=b_L\dv\mathbb{P}^\bot u_L+b_L\dv\mathbb{P}^\bot u_H+b_H\dv\mathbb{P}^\bot u_L+b_H\dv\mathbb{P}^\bot u_H.
\eeno
From Corollary \ref{coro-product} with $u=b_L$, $v=\dv\mathbb{P}^\bot u_L$, $\rho=p_2=q_2=2$, $p_1=q_1=p$, $s_1=\frac{N}{p}+2\al-1$, $s_2=\frac{N}{2}-\al$, $\sigma_1=\frac{N}{p}+2\al-2$, $\sigma_2=\frac{N}{2}-\al+1$,  one deduces that
     \begin{eqnarray}\label{e4.3}
    &&  \|b_L\dv\mathbb{P}^\bot u_L\|_{L^1_T(\dot{B}^{\frac{N}{2}-1+\al}_{2,1})}\nonumber\\
    &\leq& C\| b_L\|_{\widetilde{L}^{\frac{1}{\al}}_T(\dot{B}^{\frac{N}{p}+2\al-1}_{p,1})}
      \|\dv \Pe^\bot u_L\|_{\widetilde{L}^{\frac{1}{1-\al}}_T(\dot{B}^{\frac{N}{2}-\al}_{2,1})}
      +\|\dv \Pe^\bot u_L\|_{\widetilde{L}^{\frac{1}{\al}}_T(\dot{B}^{\frac{N}{p}+2\al-2}_{p,1})} \|b_L\|_{\widetilde{L}^{\frac{1}{1-\al}}_T(\dot{B}^{\frac{N}{2}-\al+1}_{2,1})}.
     \end{eqnarray}
Using  Corollary \ref{coro-product} again with $u=b_L$, $v=\dv\mathbb{P}^\bot u_H$, $\rho=p_2=q_2=p_1=q_1=2$,  $s_1=\sigma_2=\frac{N}{2}+\al-1$, $s_2=\sigma_1=\frac{N}{2}$,  we have
     \be
    \|b_L\dv\mathbb{P}^\bot u_H\|_{L^1_T(\dot{B}^{\frac{N}{2}-1+\al}_{2,1})}
    \leq C\| b_L\|_{\widetilde{L}^{\infty}_T(\dot{B}^{\frac{N}{2}+ \al-1}_{2,1})}
      \|\dv \Pe^\bot u_H\|_{L^1_T(\dot{B}^{\frac{N}{2} }_{2,1})}.
     \ee
Similarly,  taking $u=b_H$, $v=\dv\mathbb{P}^\bot u_L$, $\rho=p_2=q_2=p_1=q_1=2$,  $s_1=\sigma_2=\frac{N}{2}$, $s_2=\sigma_1=\frac{N}{2}+ \al-1$ in Corollary \ref{coro-product} yields
     \be
     \|b_H\dv\mathbb{P}^\bot u_L\|_{L^1_T(\dot{B}^{\frac{N}{2}-1+\al}_{2,1})}
    \leq C\| b_H\|_{\widetilde{L}^{2}_T(\dot{B}^{\frac{N}{2}}_{2,1})}
      \|\dv \Pe^\bot u_L\|_{\widetilde{L}^{2}_T(\dot{B}^{\frac{N}{2}-1+\al }_{2,1})}.
     \ee
By virtue of the low frequency embedding (\ref{lf-embeding1})
and Corollary \ref{coro-product} with $u=b_H$, $v=\dv\mathbb{P}^\bot u_H$, $\rho=p_2=q_2=p_1=q_1=2$,  $s_1=\sigma_2=\frac{N}{2} $, $s_2=\sigma_1=\frac{N}{2}-1$,  we obtain
     \be
      \|P_{<1}(b_H\dv\mathbb{P}^\bot u_H)\|_{L^1_T(\dot{B}^{\frac{N}{2}-1+\al}_{2,1})} \leq C\|b_H\dv\mathbb{P}^\bot u_H\|_{L^1_T(\dot{B}^{\frac{N}{2}-1 }_{2,1})}
    \leq C\| b_H\|_{\widetilde{L}^{2}_T(\dot{B}^{\frac{N}{2}}_{2,1})}
      \|\dv \Pe^\bot u_H\|_{\widetilde{L}^{2}_T(\dot{B}^{\frac{N}{2}-1  }_{2,1})}.
     \ee
Then, it follows from the above estimates, and the fact
  \begin{equation}\label{5.3}
    \widetilde{L}^\infty_T(\dot{B}^{\frac{N}{2}-1+\al}_{2,1})\cap L^1_T(\dot{B}^{\frac{N}{2}+1+\al}_{2,1})
    \subset \widetilde{L}^{\frac{1}{1-\al}}_T(\dot{B}^{\frac{N}{2}-\al+1}_{2,1}),
   \end{equation}
    that the estimate (\ref{5.1}) holds.

    Next,  $\|P_{<1}(\dot{T}'_{\nb b}u)\|_{L^1_T(\dot{B}^{\frac{N}{2}-1+\al}_{2,1})}$ and $\|P_{<1}(\dot{T}_{u}\nb b)\|_{L^1_T(\dot{B}^{\frac{N}{2}-1+\al}_{2,1})}$ will be bounded as follows. On the one hand, using Proposition \ref{p-TR}, we are led to
\beq\label{4.10}
\nn&&\|P_{<1}(\dot{T}'_{\nb b}\Pe^\bot u)\|_{L^1_T(\dot{B}^{\fr{N}{2}-1+\al}_{2,1})}\\
\nn&\le&\|P_{<1}(\dot{T}'_{\nb b_L}\Pe^\bot u_L)\|_{L^1_T(\dot{B}^{\fr{N}{2}-1+\al}_{2,1})}+\|P_{<1}(\dot{T}'_{\nb b_L}\Pe^\bot u_H)\|_{L^1_T(\dot{B}^{\fr{N}{2}-1+\al}_{2,1})}\\
\nn&&+\|P_{<1}(\dot{T}'_{\nb b_H}\Pe^\bot u_L)\|_{L^1_T(\dot{B}^{\fr{N}{2}-1+\al}_{2,1})}+\|P_{<1}(\dot{T}'_{\nb b_H}\Pe^\bot u_H)\|_{L^1_T(\dot{B}^{\fr{N}{2}-1}_{2,1})}\\
\nn&\le&C\|\nb b_L\|_{L^{\fr{1}{\al}}_T(\dot{B}^{2\al-2}_{\infty,1})}\|\Pe^\bot u_L\|_{L^{\fr{1}{1-\al}}_T(\dot{B}^{\fr{N}{2}-\al+1}_{2,1})}+C\|\nb b_L\|_{L^{2}_T(\dot{B}^{\fr{N}{2}-1+\al}_{2,1})}\|\Pe^\bot u_H\|_{L^{2}_T(\dot{B}^{\fr{N}{2}}_{2,1})}\\
\nn&&+C\|\nb b_H\|_{L^{2}_T(\dot{B}^{-1}_{\infty,1})}\|\Pe^\bot u_L\|_{L^{2}_T(\dot{B}^{\fr{N}{2}+\al}_{2,1})}+C\|\nb b_H\|_{L^{2}_T(\dot{B}^{\fr{N}{2}-1}_{2,1})}\|\Pe^\bot u_H\|_{L^{2}_T(\dot{B}^{\fr{N}{2}}_{2,1})}\\
&\le&CX^2(T).
\eeq
On the other hand, since \eqref{p1} ensures that $\fr{N}{p*}-1\le0$, i.e. $p\le\fr{2N}{N-2}$, thanks to $\dv \Pe u=0$, we obtain
\beq\label{4.11}
\nn&&\|P_{<1}(\dot{T}'_{\nb b}\Pe u)\|_{L^1_T(\dot{B}^{\fr{N}{2}-1+\al}_{2,1})}\\
\nn&\le&\|P_{<1}(\dot{T}'_{\nb b_L}\Pe u)\|_{L^1_T(\dot{B}^{\fr{N}{2}-1+\al}_{2,1})}+C\|P_{<1}(\pr_k\dot{T}'_{ b_H}(\Pe u)^k)\|_{L^1_T(\dot{B}^{\fr{N}{2}-1}_{2,1})}\\
\nn&\le&C\|\nb b_L\|_{L^{\infty}_T(\dot{B}^{\fr{N}{p^*}-2+\al}_{p^*,1})}\|\Pe u\|_{L^{1}_T(\dot{B}^{\fr{N}{p}+1}_{p,1})}+C\| b_H\|_{L^{\infty}_T(\dot{B}^{\fr{N}{p^*}-1}_{p^*,1})}\|\Pe u\|_{L^{1}_T(\dot{B}^{\fr{N}{p}+1}_{p,1})}\\
\nn&\le&C\|b_L\|_{L^{\infty}_T(\dot{B}^{\fr{N}{2}-1+\al}_{2,1})}\|\Pe u\|_{L^{1}_T(\dot{B}^{\fr{N}{p}+1}_{p,1})}+C\| b_H\|_{L^{\infty}_T(\dot{B}^{\fr{N}{2}-1}_{2,1})}\|\Pe u\|_{L^{1}_T(\dot{B}^{\fr{N}{p}+1}_{p,1})}\\
\nn&\le&C\left(\|b_L\|_{L^{\infty}_T(\dot{B}^{\fr{N}{2}-1+\al}_{2,1})}+\| b_H\|_{L^{\infty}_T(\dot{B}^{\fr{N}{2}}_{2,1})}\right)\|\Pe u\|_{L^{1}_T(\dot{B}^{\fr{N}{p}+1}_{p,1})}\\
&\le&CX^2(T),
\eeq
where we have used (\ref{lf-embeding1})-(\ref{hf-embedding1}).
Similarly, owing to \eqref{5.3} and the following interpolation,
\begin{equation}\label{5.4}
    \widetilde{L}^\infty_T(\dot{B}^{\frac{N}{2}-1 }_{2,1})\cap L^1_T(\dot{B}^{\frac{N}{2}+1 }_{2,1})
   \subset \widetilde{L}^{\frac{1}{\al}}_T(\dot{B}^{\frac{N}{2}+2\al-1}_{2,1}) \subset \widetilde{L}^{\frac{1}{\al}}_T(\dot{B}^{\fr{N}{p}+2\al-1}_{p,1}),
     \end{equation}
 we have
\beq\label{4.13}
\nn&&\|P_{<1}(\dot{T}_{\Pe^\bot u}\cdot\nb b)\|_{L^1_T(\dot{B}^{\fr{N}{2}-1+\al}_{2,1})}\\
\nn&\le&\|P_{<1}(\dot{T}_{\Pe^\bot u}\cdot\nb b_L)\|_{L^1_T(\dot{B}^{\fr{N}{2}-1+\al}_{2,1})}+C\|P_{<1}(\dot{T}_{\Pe^\bot u}\cdot\nb b_H)\|_{L^1_T(\dot{B}^{\fr{N}{2}-2+2\al}_{2,1})}\\
\nn&\le&C\|\Pe^\bot u\|_{L^{\fr{1}{\al}}_T(\dot{B}^{2\al-1}_{\infty,1})}\|\nb b_L\|_{L^{\fr{1}{1-\al}}_T(\dot{B}^{\fr{N}{2}-\al}_{2,1})}+C\|\Pe^\bot u\|_{L^{\fr{1}{\al}}_T(\dot{B}^{2\al-1}_{\infty,1})}\|\nb b_H\|_{L^{\fr{1}{1-\al}}_T(\dot{B}^{\fr{N}{2}-1}_{2,1})}\\
\nn&\le&C\|\Pe^\bot u\|_{L^{\fr{1}{\al}}_T(\dot{B}^{\fr{N}{p}+2\al-1}_{p,1})}\left(\| b_L\|_{L^{\fr{1}{1-\al}}_T(\dot{B}^{\fr{N}{2}+1-\al}_{2,1})}+\| b_H\|_{L^{\fr{1}{1-\al}}_T(\dot{B}^{\fr{N}{2}}_{2,1})}\right)\\
&\le&CX^2(T),
\eeq
and
\beq\label{4.14}
\nn&&\|P_{<1}(\dot{T}_{\Pe u}\cdot\nb b)\|_{L^1_T(\dot{B}^{\fr{N}{2}-1+\al}_{2,1})}\\
\nn&\le&\|P_{<1}(\dot{T}_{\Pe u}\cdot\nb b_L)\|_{L^1_T(\dot{B}^{\fr{N}{2}-1+\al}_{2,1})}+C\|P_{<1}(\dot{T}_{\Pe u}\cdot\nb b_H)\|_{L^1_T(\dot{B}^{\fr{N}{2}-1}_{2,1})}\\
\nn&\le&C\|\Pe u\|_{L^{\infty}_T(\dot{B}^{-1}_{\infty,1})}\|\nb b_L\|_{L^{1}_T(\dot{B}^{\fr{N}{2}+\al}_{2,1})}+C\|\Pe u\|_{L^{2}_T(L^\infty)}\|\nb b_H\|_{L^{2}_T(\dot{B}^{\fr{N}{2}-1}_{2,1})}\\
\nn&\le&C\|\Pe u\|_{L^{\infty}_T(\dot{B}^{\fr{N}{p}-1}_{p,1})}\| b_L\|_{L^{1}_T(\dot{B}^{\fr{N}{2}+1+\al}_{2,1})}+C\|\Pe u\|_{L^{2}_T(\dot{B}^{\fr{N}{p}}_{p,1})}\| b_H\|_{L^{2}_T(\dot{B}^{\fr{N}{2}}_{2,1})}\\
&\le&CX^2(T).
\eeq
Combining the above estimates, we complete the proof of Lemma \ref{lem4.1}.
{\hfill
$\square$\medskip}

\textbf{Proof of Lemma \ref{lem4.2}.}

First of all, using the fact
    \begin{equation}\label{5.6}
    \widetilde{L}^\infty_T(\dot{B}^{\frac{N}{2}-1+\al}_{2,1})\cap L^1_T(\dot{B}^{\frac{N}{2}+1+\al}_{2,1})
    \subset \widetilde{L}^{\frac{2}{1-\al}}_T(\dot{B}^{\frac{N}{2}}_{2,1})\subset L^{\frac{2}{1-\al}}_T(L^\infty),
  \end{equation}
 and the decomposition $b=b_L+b_H$, we have
\be\label{b1}
\| b \|_{\widetilde{L}^{\fr{2}{1-\al}}_T(\dot{B}^{\frac{N}{2}}_{2,1})}\le \| b_L \|_{\widetilde{L}^{\fr{2}{1-\al}}_T(\dot{B}^{\frac{N}{2}}_{2,1})}+\| b_H \|_{\widetilde{L}^{\fr{2}{1-\al}}_T(\dot{B}^{\frac{N}{2}}_{2,1})}\le CX(T).
\ee
Moreover, with the aid of the following low frequency embedding
\beno
\| b_L \|_{\widetilde{L}^{\infty}_T(\dot{B}^{\frac{N}{2}}_{2,1})}\le C \| b_L \|_{\widetilde{L}^{\infty}_T(\dot{B}^{\frac{N}{2}-1+\al}_{2,1})},
\eeno
we find that
\be\label{b2}
\| b \|_{\widetilde{L}^{\infty}_T(\dot{B}^{\frac{N}{2}}_{2,1})}\le C\left( \| b_L \|_{\widetilde{L}^{\infty}_T(\dot{B}^{\frac{N}{2}-1+\al}_{2,1})}+ \| b_H \|_{\widetilde{L}^{\infty}_T(\dot{B}^{\frac{N}{2}}_{2,1})}\right)\le CX(T).
\ee
Now using Bony's decomposition, the high frequency embedding (\ref{hf-embedding1}),
Lemma \ref{Bernstein} and Proposition \ref{p-TR}, we are led to
\beq\label{4.19}
 \nn&&\|P_{\ge1} (b \dv u)\|_{L^1_T(\dot{B}^{\frac{N}{2}}_{2,1})}\\
 \nn&\le&\|P_{\ge1} (b \dv u_L)\|_{L^1_T(\dot{B}^{\frac{N}{2}}_{2,1})}+\|P_{\ge1} (b \dv u_H)\|_{L^1_T(\dot{B}^{\frac{N}{2}}_{2,1})}\\
 \nn&\le&C\|P_{\ge1} (\dot{T}_b \dv u_L)\|_{L^1_T(\dot{B}^{\frac{N}{2}+\al}_{2,1})}+C\|\dot{T}'_{\dv u_L} b \|_{L^1_T(\dot{B}^{\frac{N}{2}}_{2,1})}+C\|b \dv u_H\|_{L^1_T(\dot{B}^{\frac{N}{2}}_{2,1})}\\
 \nn&\le&C\|b \|_{L^\infty_T(\dot{B}^{\frac{N}{2}}_{2,1})}\| \dv u_L\|_{L^1_T(\dot{B}^{\frac{N}{2}+\al}_{2,1})}+C\| b \|_{L^{\fr{2}{1-\al}}_T(\dot{B}^{\frac{N}{2}}_{2,1})}\|{\dv u_L}  \|_{L^{\fr{2}{1+\al}}_T(\dot{B}^{\frac{N}{2}}_{2,1})}\\
 \nn&&+C\|b \|_{L^\infty_T(\dot{B}^{\frac{N}{2}}_{2,1})}\| \dv u_H\|_{L^1_T(\dot{B}^{\frac{N}{2}}_{2,1})}\\
 \nn&\le&C\|b \|_{L^\infty_T(\dot{B}^{\frac{N}{2}}_{2,1})}\| \Pe^\bot u_L\|_{L^1_T(\dot{B}^{\frac{N}{2}+1+\al}_{2,1})}+C\| b \|_{L^{\fr{2}{1-\al}}_T(\dot{B}^{\frac{N}{2}}_{2,1})}\|{\Pe^\bot u_L}  \|_{L^{\fr{2}{1+\al}}_T(\dot{B}^{\frac{N}{2}+2\al}_{2,1})}\\
 \nn&&+C\|b \|_{L^\infty_T(\dot{B}^{\frac{N}{2}}_{2,1})}\| \Pe^\bot u_H\|_{L^1_T(\dot{B}^{\frac{N}{2}+1}_{2,1})}\\
 &\le&CX^2(T),
 \eeq
 where we have used \eqref{b1}--\eqref{b2}, the interpolation
\be\label{5.6-1}
\widetilde{L}^\infty_T(\dot{B}^{\fr{N}{2}-1+\al}_{2,1})\cap{L}^1_T(\dot{B}^{\fr{N}{2}+1+\al}_{2,1})\subset \widetilde{L}^{\fr{2}{1+\al}}_T(\dot{B}^{\fr{N}{2}+2\al}_{2,1}),
\ee
and the following low frequency embedding
\be\label{lfe1}
\|\Pe^\bot u_L\|_{\widetilde{L}^{\fr{2}{1+\al}}_T(\dot{B}^{\fr{N}{2}+1}_{2,1})}\le C\|\Pe^\bot u_L\|_{\widetilde{L}^{\fr{2}{1+\al}}_T(\dot{B}^{\fr{N}{2}+2\al}_{2,1})}.
\ee
Next, noting that $\fr{N}{p*}\le1$, using Proposition \ref{p-TR}, we find that
\beq\label{4.22}
\nn&&\|P_{\ge1}(\dot{T}'_{\nb b}u)\|_{L^1_T(\dot{B}^{\frac{N}{2}}_{2,1})}\\
\nn&\le&\|P_{\ge1}(\dot{T}'_{\nb b}\Pe^\bot u_L)\|_{L^1_T(\dot{B}^{\frac{N}{2}}_{2,1})}+\|P_{\ge1}(\dot{T}'_{\nb b}(\Pe^\bot u_H+\Pe u))\|_{L^1_T(\dot{B}^{\frac{N}{2}}_{2,1})}\\
\nn&\le&C\|P_{\ge1}(\dot{T}'_{\nb b}\Pe^\bot u_L)\|_{L^1_T(\dot{B}^{\frac{N}{2}+\al}_{2,1})}+C\|\nb b\|_{L^\infty_T(\dot{B}^{\frac{N}{p^*}-1}_{p^*,1})}\|\Pe^\bot u_H+\Pe u\|_{L^1_T(\dot{B}^{\frac{N}{p}+1}_{p,1})}\\
\nn&\le&C\|b\|_{L^\infty_T(\dot{B}^{\frac{N}{2}}_{2,1})}\|\Pe^\bot u_L\|_{L^1_T(\dot{B}^{\frac{N}{2}+1+\al}_{2,1})}+C\| b\|_{L^\infty_T(\dot{B}^{\frac{N}{2}}_{2,1})}\left(\|\Pe^\bot u_H\|_{L^1_T(\dot{B}^{\frac{N}{2}+1}_{2,1})}+\|\Pe u\|_{L^1_T(\dot{B}^{\frac{N}{p}+1}_{p,1})}\right)\\
&\le&CX^2(T).
\eeq
This completes the proof of Lemma \ref{lem4.2}.
{\hfill
$\square$\medskip}

\textbf{Proof of Lemma \ref{lem4.3}.}

From the low frequency embedding \eqref{lf-embeding1}, Corollary \ref{coro-product} and Theorem 2.61 in \cite{Bahouri-Chemin-Danchin11}, we infer that
\beq\label{4.27'}
&&\nn\|P_{<1}\left(I(b)\mathcal{A}\Pe^\bot u\right)\|_{L^1_T(\dot{B}^{\fr{N}{2}-1+\al}_{2,1})}+\|P_{\ge1}\left(I(b)\mathcal{A}\Pe^\bot u\right)\|_{L^1_T(\dot{B}^{\fr{N}{2}-1}_{2,1})}\\
\nn&\le&C\|I(b)\mathcal{A}\Pe^\bot u_L\|_{L^1_T(\dot{B}^{\fr{N}{2}-1}_{2,1})}+C\|I(b)\mathcal{A}\Pe^\bot u_H\|_{L^1_T(\dot{B}^{\fr{N}{2}-1}_{2,1})}\\
\nn&\le&C\|I(b)\|_{L^{\fr{2}{1-\al}}_T(\dot{B}^{\fr{N}{2}}_{2,1})}\|\mathcal{A}\Pe^\bot u_L\|_{L^{\fr{2}{1+\al}}_T(\dot{B}^{\fr{N}{2}-1}_{2,1})}+C\|I(b)\|_{L^\infty_T(\dot{B}^{\fr{N}{2}}_{2,1})}\|\mathcal{A}\Pe^\bot u_H\|_{L^1_T(\dot{B}^{\fr{N}{2}-1}_{2,1})}\\
\nn&\le&C\|b\|_{L^{\fr{2}{1-\al}}_T(\dot{B}^{\fr{N}{2}}_{2,1})}\|\Pe^\bot u_L\|_{L^{\fr{2}{1+\al}}_T(\dot{B}^{\fr{N}{2}+2\al}_{2,1})}+C\|b\|_{L^\infty_T(\dot{B}^{\fr{N}{2}}_{2,1})}\|\Pe^\bot u_H\|_{L^1_T(\dot{B}^{\fr{N}{2}+1}_{2,1})}\\
&\le&CX^2(T),
\eeq
where we have used \eqref{b1}--\eqref{b2}, and \eqref{5.6-1}--\eqref{lfe1}. Next, using $\dv\Pe u=0$, we decompose $\Lm^{-1}\dv(I(b)\mathcal{A}\Pe u)$ as follows:
\beno
\Lm^{-1}\dv(I(b)\mathcal{A}\Pe u)=\Lm^{-1}\left(\dot{T}_{\nb I(b)}\mathcal{A}\Pe u\right)+\Lm^{-1}\dv\left(\dot{T}_{\mathcal{A}\Pe u} I(b)\right)+\Lm^{-1}\dv\left(\dot{R}({ I(b)},\mathcal{A}\Pe u)\right).
\eeno
Then according to Lemma \ref{Bernstein},  Proposition \ref{p-TR},  \eqref{b2}, and Theorem 2.61 in \cite{Bahouri-Chemin-Danchin11} again, we have
\beq\label{TbPu1}
\nn&&\|\Lm^{-1}\left(\dot{T}_{\nb I(b)}\mathcal{A}\Pe u\right)\|_{L^1_T(\dot{B}^{\fr{N}{2}-1}_{2,1})}\\
\nn&\le&C\|\dot{T}_{\nb I(b)}\mathcal{A}\Pe u\|_{L^1_T(\dot{B}^{\fr{N}{2}-2}_{2,1})}\\
\nn&\le&C\|{\nb I(b)}\|_{L^\infty_T(\dot{B}^{\fr{N}{p^*}-1}_{p^*,1})}\|\mathcal{A}\Pe u\|_{L^1_T(\dot{B}^{\fr{N}{p}-1}_{p,1})}\\
\nn&\le&C\|b\|_{L^\infty_T(\dot{B}^{\fr{N}{2}}_{2,1})}\|\Pe u\|_{L^1_T(\dot{B}^{\fr{N}{p}+1}_{p,1})}\\
&\le&CX^2(T),
\eeq
and
\beq\label{4.27}
\nn&&\|\Lm^{-1}\dv\left(\dot{T}_{\mathcal{A}\Pe u} I(b)\right)\|_{L^1_T(\dot{B}^{\fr{N}{2}-1}_{2,1})}\\
\nn&\le&C\|\mathcal{A}\Pe u\|_{L^1_T(\dot{B}^{-1}_{\infty,1})}\|{ I(b)}\|_{L^\infty_T(\dot{B}^{\fr{N}{2}}_{2,1})}\\
\nn&\le&C\|b\|_{L^\infty_T(\dot{B}^{\fr{N}{2}}_{2,1})}\|\Pe u\|_{L^1_T(\dot{B}^{\fr{N}{p}+1}_{p,1})}\\
&\le&CX^2(T).
\eeq
 Finally, we go to bound the remainder term $\Lm^{-1}\dv\left(\dot{R}({ I(b)},\mathcal{A}\Pe u)\right)$. Noting that $\fr{N}{2}-1=0$ for $N=2$, we can not use Proposition \ref{p-TR} directly if $N=2$. Fortunately, Proposition \ref{prop-classical} enables us to bound $\|\Lm^{-1}\dv\left(\dot{R}({ I(b)},\mathcal{A}\Pe u)\right)\|_{\dot{B}^{\fr{N}{2}-1}_{2,1}}$ by
 $$\|\Lm^{-1}\dv\left(\dot{R}({ I(b)},\mathcal{A}\Pe u)\right)\|_{\dot{B}^{\fr{N}{2}+\fr{N}{p}-1}_{\fr{2p}{p+2},1}}$$
 first. Then Proposition \ref{p-TR} is applicable, and we have
\beq\label{4.30}
\nn&&\|\Lm^{-1}\dv\left(\dot{R}({ I(b)},\mathcal{A}\Pe u)\right)\|_{L^1_T(\dot{B}^{\fr{N}{2}-1}_{2,1})}\\
\nn&\le&C\|\dot{R}({ I(b)},\mathcal{A}\Pe u)\|_{L^1_T(\dot{B}^{\fr{N}{2}+\fr{N}{p}-1}_{\fr{2p}{p+2},1})}\\
\nn&\le&C\|{ I(b)}\|_{L^\infty_T(\dot{B}^{\fr{N}{2}}_{2,1})}\|\mathcal{A}\Pe u\|_{L^1_T(\dot{B}^{\fr{N}{p}-1}_{p,1})}\\
\nn&\le&C\|b\|_{L^\infty_T(\dot{B}^{\fr{N}{2}}_{2,1})}\|\Pe u\|_{L^1_T(\dot{B}^{\fr{N}{p}+1}_{p,1})}\\
&\le&CX^2(T).
\eeq
It follows from \eqref{TbPu1}--\eqref{4.30} and the low frequency embedding \eqref{lf-embeding1} that \eqref{4.26} holds. This completes the proof of Lemma \ref{lem4.3}.
{\hfill
$\square$\medskip}

\textbf{Proof of Lemma \ref{lem4.4}.}

Using the low frequency embedding \eqref{lf-embeding1}, high frequency embedding \eqref{hf-embedding1}, and the decomposition $b=b_L+b_H$, we have
\beq
\nn&&\|P_{<1}\left(K(b)\nb b\right)\|_{L^1_T(\dot{B}^{\fr{N}{2}-1+\al}_{2,1})}+\|P_{\ge1}\left(K(b)\nb b\right)\|_{L^1_T(\dot{B}^{\fr{N}{2}-1}_{2,1})}\\
\nn&\le&C\|K(b)\nb b_L\|_{L^1_T(\dot{B}^{\fr{N}{2}-1+\al}_{2,1})}+C\|K(b)\nb b_H\|_{L^1_T(\dot{B}^{\fr{N}{2}-1}_{2,1})}.
\eeq
By virtue of Corollary \ref{coro-product},  \eqref{b2}, and Theorem 2.16 in \cite{Bahouri-Chemin-Danchin11}, we obtain
\beq\label{e4.32}
\|K(b)\nb b_H\|_{L^1_T(\dot{B}^{\fr{N}{2}-1}_{2,1})}\nn&\le& C\|K(b)\|_{L^\infty_T(\dot{B}^{\fr{N}{2}}_{2,1})}\|\nb b_H\|_{L^1_T(\dot{B}^{\fr{N}{2}-1}_{2,1})}\\
\nn&\le&C\|b\|_{L^\infty_T(\dot{B}^{\fr{N}{2}}_{2,1})}\| b_H\|_{L^1_T(\dot{B}^{\fr{N}{2}}_{2,1})}\\
&\le&CX^2(T).
\eeq
The remaining term will be divided into three parts.
\beq\label{e4.33}
\nn&&\|K(b)\nb b_L\|_{L^1_T(\dot{B}^{\fr{N}{2}-1+\al}_{2,1})}\\
\nn&\le& C\|\dot{T}_{K(b)}\nb b_L\|_{L^1_T(\dot{B}^{\fr{N}{2}-1+\al}_{2,1})}+C\|\dot{T}'_{\nb b_L}K(b)_H\|_{L^1_T(\dot{B}^{\fr{N}{2}-1+\al}_{2,1})}+C\|\dot{T}'_{\nb b_L}K(b)_L\|_{L^1_T(\dot{B}^{\fr{N}{2}-1+\al}_{2,1})}\\
&=:&J_1+J_2+J_3.
\eeq
To bound $J_1$,
using  the  interpolation inequality,
we infer that
\be\label{bL}
\|b_L\|_{\widetilde{L}^{r_1}_T(\dot{B}^{\fr{N}{2}+\al}_{p_1,1})}\le C\|b_L\|^\theta_{\widetilde{L}^{\fr{1}{\al}}_T(\dot{B}^{\fr{N}{p}+2\al-1}_{p,1})}\|b_L\|^{1-\theta}_{\widetilde{L}^{\fr{1}{1-\al}}_T(\dot{B}^{\fr{N}{2}+1-\al}_{2,1})}\le C X(T),
\ee
where
$$
\theta=\fr{(2-4\al)p}{(4+N-6\al)p-2N},\ \fr{1}{r_1}=\theta\al+(1-\theta)(1-\al), \quad \fr{1}{p_1}=\fr{\theta}{p}+\fr{1-\theta}{2},
$$
and
\beq\label{b4}
\|b\|_{\widetilde{L}^{r_2}_T(\dot{B}^{\fr{N}{p}}_{p_2,1})}\nn&\le& \|b_L\|_{\widetilde{L}^{r_2}_T(\dot{B}^{\fr{N}{p}}_{p_2,1})}+\|b_H\|_{\widetilde{L}^{r_2}_T(\dot{B}^{\fr{N}{p}}_{p_2,1})}\\
\nn&\le& C\left(\|b_L\|_{\widetilde{L}^{r_2}_T(\dot{B}^{\fr{N}{p}}_{p_2,1})}+\|b_H\|_{\widetilde{L}^{r_2}_T(\dot{B}^{\frac{N}{p}+\frac{N}{2}-\frac{N}{p_2}}_{2,1})}\right)\\
\nn&\le& C\left(\|b_L\|^{1-\theta}_{\widetilde{L}^{\fr{1}{\al}}_T(\dot{B}^{\fr{N}{p}+2\al-1}_{p,1})}\|b_L\|^{\theta}_{\widetilde{L}^{\fr{1}{1-\al}}_T(\dot{B}^{\fr{N}{2}+1-\al}_{2,1})}+\|b_H\|_{\widetilde{L}^{r_2}_T(\dot{B}^{\fr{N}{2}}_{2,1})} \right)\\
&\le&CX(T),
\eeq
where
\beno
\fr{1}{r_2}=(1-\theta)\al+\theta(1-\al), \quad \fr{1}{p_2}=\fr{1-\theta}{p}+\fr{\theta}{2}.
\eeno
Here, $r_1$, $r_2$, $p_1$ and $p_2$ satisfy
\be\label{p2}
\fr{1}{r_1}+\fr{1}{r_2}=1, \quad \fr{1}{p_2}=\fr{1}{p^*_1}+\fr{1}{p},\ 2\leq p_1,p_2\leq p.
\ee
Thus, using Proposition \ref{p-TR}, Theorem 2.61 in \cite{Bahouri-Chemin-Danchin11} and \eqref{bL}--\eqref{p2}, we arrive at
\beq\label{e4.35}
J_1\nn&\le&C\|K(b)\|_{L^{r_2}_T(\dot{B}^{0}_{p_1^*,1})}\|\nb b_L\|_{L^{r_1}_T(\dot{B}^{\fr{N}{2}-1+\al}_{p_1,1})}\\
\nn&\le&C\|K(b)\|_{L^{r_2}_T(\dot{B}^{\fr{N}{p}}_{p_2,1})}\| b_L\|_{L^{r_1}_T(\dot{B}^{\fr{N}{2}+\al}_{p_1,1})}\\
\nn&\le&C\|b\|_{L^{r_2}_T(\dot{B}^{\fr{N}{p}}_{p_2,1})}\| b_L\|_{L^{r_1}_T(\dot{B}^{\fr{N}{2}+\al}_{p_1,1})}\\
&\le&CX^2(T).
\eeq
As for $J_2$, according to \eqref{b1}, \eqref{5.6-1}, the following high frequency embedding
\beno
\|K(b)_H\|_{L^{\fr{2}{1-\al}}_T(\dot{B}^{\fr{N}{2}-\al}_{2,1})}\le C\|K(b)_H\|_{L^{\fr{2}{1-\al}}_T(\dot{B}^{\fr{N}{2}}_{2,1})},
\eeno
and Theorem 2.61 in \cite{Bahouri-Chemin-Danchin11} again,  we obtain
\beq
J_2\nn&\le&C\|\dot{T}'_{\nb b_L}K(b)_H\|_{L^1_T(\dot{B}^{\fr{N}{2}-1+\al}_{2,1})}\\
\nn&\le&C\|{\nb b_L}\|_{L^{\fr{2}{1+\al}}_T(\dot{B}^{2\al-1}_{\infty,1})}\|K(b)_H\|_{L^{\fr{2}{1-\al}}_T(\dot{B}^{\fr{N}{2}-\al}_{2,1})}\\
\nn&\le&C\|{ b_L}\|_{L^{\fr{2}{1+\al}}_T(\dot{B}^{\fr{N}{2}+2\al}_{2,1})}\|K(b)_H\|_{L^{\fr{2}{1-\al}}_T(\dot{B}^{\fr{N}{2}}_{2,1})}\\
\nn&\le&C\|{ b_L}\|_{L^{\fr{2}{1+\al}}_T(\dot{B}^{\fr{N}{2}+2\al}_{2,1})}\|b\|_{L^{\fr{2}{1-\al}}_T(\dot{B}^{\fr{N}{2}}_{2,1})}\\
&\le&CX^2(T).
\eeq
Finally, in view of  Lemma \ref{lem-Kb} in the Appendix, we have
\be\label{KbL}
\|K(b)_L\|_{\widetilde{L}^{\fr{1}{1-\al}}_T(\dot{B}^{\fr{N}{2}+1-\al}_{2,1})}\\
\le C\left(\|b_L\|_{\widetilde{L}^{\fr{1}{1-\al}}_T(\dot{B}^{\fr{N}{2}+1-\al}_{2,1})}+\|b_H\|_{\widetilde{L}^{\fr{1}{1-\al}}_T(\dot{B}^{\fr{N}{2}}_{2,1})}\right).
\ee
Therefore, using Proposition \ref{p-TR} again, one deduces that
\beq\label{e4.38}
J_3\nn&\le&C\|\nb b_L\|_{\widetilde{L}^{\fr{1}{\al}}_T(\dot{B}^{2\al-2}_{\infty,1})}\|K(b)_L\|_{\widetilde{L}^{\fr{1}{1-\al}}_T(\dot{B}^{\fr{N}{2}+1-\al}_{2,1})}\\
\nn&\le&C\| b_L\|_{\widetilde{L}^{\fr{1}{\al}}_T(\dot{B}^{\fr{N}{p}+2\al-1}_{p,1})}\left(\|b_L\|_{\widetilde{L}^{\fr{1}{1-\al}}_T(\dot{B}^{\fr{N}{2}+1-\al}_{2,1})}+\|b_H\|_{\widetilde{L}^{\fr{1}{1-\al}}_T(\dot{B}^{\fr{N}{2}}_{2,1})}\right)\\
&\le&CX^2(T).
\eeq
The proof of Lemma \ref{lem4.4} is completed.
{\hfill
$\square$\medskip}

\subsection{Some nonlinear estimates in $L^1_T(\dot{B}^{\fr{N}{2}-1}_{2,1})$}

In next three lemmas, we shall bound  $\|(f,g)\|_{L^1_T(\dot{B}^{\fr{N}{2}-1}_{2,1})}$ in terms of $X(T)Z(T)$, where $(f,g)$ are the nonlinear terms on the right hand side of \eqref{viscoelastic-local}.
 \begin{lem}\label{lem-a4.1}
Let $p$ and $\al$ satisfy \eqref{p1} and  \eqref{al1}, respectively. Assume $(b,u)\in \mathcal{E}^{\frac{N}{2}}(T)$  , then we have
    \begin{equation}\label{a5.1}
    \|b \dv u\|_{L^1_T(\dot{B}^{\frac{N}{2}-1}_{2,1})}+\|\dot{T}'_{\nb b}u\|_{L^1_T(\dot{B}^{\frac{N}{2}-1}_{2,1})}
    \leq C X(T)Z(T),
      \end{equation}
    and
    \begin{equation}\label{a5.2}
    \| ( u\cdot\nabla) u\|_{L^1_T(\dot{B}^{\frac{N}{2}-1}_{2,1})}+\| \dot{T}_{ u}\cdot\nabla d\|_{L^1_T(\dot{B}^{\frac{N}{2}-1}_{2,1})}
    \leq C X(T)Z(T).
    \end{equation}
        \end{lem}
\begin{proof}
To begin with, we give the  estimate of $\|b\|_{\widetilde{L}^{\fr{1}{\al}}_T(\dot{B}^{\fr{N}{p}+2\al-1}_{p,1})}$. In fact, using the high frequency embedding
\beno
\|b_H\|_{\widetilde{L}^{\fr{1}{\al}}_T(\dot{B}^{\fr{N}{p}+2\al-1}_{p,1})}\le C\|b_H\|_{\widetilde{L}^{\fr{1}{\al}}_T(\dot{B}^{\fr{N}{p}}_{p,1})}\le C\|b_H\|_{\widetilde{L}^{\fr{1}{\al}}_T(\dot{B}^{\fr{N}{2}}_{2,1})},
\eeno
and the decomposition $b=b_L+b_H$, one deduces that
\be\label{b3}
\|b\|_{\widetilde{L}^{\fr{1}{\al}}_T(\dot{B}^{\fr{N}{p}+2\al-1}_{p,1})}\le C\left(\|b_L\|_{\widetilde{L}^{\fr{1}{\al}}_T(\dot{B}^{\fr{N}{p}+2\al-1}_{p,1})}+\|b_H\|_{\widetilde{L}^{\fr{1}{\al}}_T(\dot{B}^{\fr{N}{2}}_{2,1})}\right)\le CX(T).
\ee
Now if $N\ge3$, using Proposition \ref{p-TR},   \eqref{5.4},  \eqref{5.5},  and \eqref{b3}
     we have
    \begin{eqnarray}\label{a5.3}
      &&\|b\dv u\|_{L^1_T(\dot{B}^{\frac{N}{2}-1}_{2,1})}\nonumber\\
      &\leq&C \left( \|\dot{T}_{b}\dv u\|_{L^1_T(\dot{B}^{\frac{N}{2}-1 }_{2,1})}
      +\|\dot{T}'_{\dv u}b_L\|_{L^1_T(\dot{B}^{\frac{N}{2} -1}_{2,1})}
      +\|\dot{T}'_{\dv u}b_H\|_{L^1_T(\dot{B}^{\frac{N}{2}-1 }_{2,1})}
      \right)\nonumber\\
  \nn&\leq&C \left( \| b\|_{L^{\frac{1}{\al}}_T(\dot{B}^{2\al-1}_{\infty,1})}
      \|\dv  u\|_{L^{\frac{1}{1-\al}}_T(\dot{B}^{\frac{N}{2}-2\al}_{2,1})}
      +\|\dv  u\|_{L^{\frac{1}{\al}}_T(\dot{B}^{2\al-2}_{\infty,1})} \|b_L\|_{L^{\frac{1}{1-\al}}_T(\dot{B}^{\frac{N}{2}-2\al+1}_{2,1})}\right.\\
            &&\nn\left.+\|\dv u\|_{L^{2}_T(\dot{B}^{\frac{N}{2}-1}_{2,1})} \|b_H\|_{L^{2}_T(\dot{B}^{\frac{N}{2}}_{2,1})}
    \right)\nonumber\\
    \nn&\leq&C \left( \| b\|_{L^{\frac{1}{\al}}_T(\dot{B}^{2\al-1}_{\infty,1})}
      \|  u\|_{L^{\frac{1}{1-\al}}_T(\dot{B}^{\frac{N}{2}-2\al+1}_{2,1})}
      +\|  u\|_{L^{\frac{1}{\al}}_T(\dot{B}^{\fr{N}{p}+2\al-1}_{p,1})} \|b_L\|_{L^{\frac{1}{1-\al}}_T(\dot{B}^{\frac{N}{2}-2\al+1}_{2,1})}\right.\\
            &&\nn\left.+\| u\|_{L^{2}_T(\dot{B}^{\frac{N}{2}}_{2,1})} \|b_H\|_{L^{2}_T(\dot{B}^{\frac{N}{2}}_{2,1})}
    \right)\nonumber\\
      &\leq&C X(T)Z(T),
    \end{eqnarray}
and
\beq\label{a5.4}
\|\dot{T}'_{\nb b} u\|_{L^1_T(\dot{B}^{\fr{N}{2}-1}_{2,1})}
\nn&\le&C\|\nb b\|_{L^{\fr{1}{\al}}_T(\dot{B}^{2\al-2}_{\infty,1})}\|u\|_{L^{\fr{1}{1-\al}}_T(\dot{B}^{\fr{N}{2}-2\al+1}_{2,1})}\\
&\le&CX(T)Z(T).
\eeq
If $N=2$, we just need to reestimate $\|\dot{R}(\dv u, b_L)\|_{L^1_T(\dot{B}^{\frac{N}{2}-1}_{2,1})}$ and $\|\dot{R}(\nb b,  u)\|_{L^1_T(\dot{B}^{\frac{N}{2}-1}_{2,1})}$. Indeed, they can be treated in the same say as follows:
\beq\label{a5.5}
\nn&&\|\dot{R}(\dv u, b_L)\|_{L^1_T(\dot{B}^{\frac{N}{2}-1}_{2,1})}+\|\dot{R}(\nb b,  u)\|_{L^1_T(\dot{B}^{\frac{N}{2}-1}_{2,1})}\\
\nn&\le&C\|\dot{R}(\dv u, b_L)\|_{L^1_T(\dot{B}^{\frac{2N}{p}-1}_{\fr{p}{2},1})}+\|\dot{R}(\nb b,  u)\|_{L^1_T(\dot{B}^{\frac{2N}{p}-1}_{\fr{p}{2},1})}\\
\nn&\le&C\|\dv u\|_{L^{\fr{1}{\al}}_T(\dot{B}^{\fr{N}{p}+2\al-2}_{p,1})}\|b_L\|_{L^{\fr{1}{1-\al}}_T(\dot{B}^{\fr{N}{p}-2\al+1}_{p,1})}+C\|\nb b\|_{L^{\fr{1}{\al}}_T(\dot{B}^{\fr{N}{p}+2\al-2}_{p,1})}\|u\|_{L^{\fr{1}{1-\al}}_T(\dot{B}^{\fr{N}{p}-2\al+1}_{p,1})}\\
\nn&\le&C\| u\|_{L^{\fr{1}{\al}}_T(\dot{B}^{\fr{N}{p}+2\al-1}_{p,1})}\|b_L\|_{L^{\fr{1}{1-\al}}_T(\dot{B}^{\fr{N}{2}-2\al+1}_{2,1})}+C\|b\|_{L^{\fr{1}{\al}}_T(\dot{B}^{\fr{N}{p}+2\al-1}_{p,1})}\|u\|_{L^{\fr{1}{1-\al}}_T(\dot{B}^{\fr{N}{2}-2\al+1}_{2,1})}\\
&\le&CX(T)Z(T).
\eeq
It follows from \eqref{a5.3}--\eqref{a5.5} that \eqref{a5.1} holds. \eqref{a5.2} can be obtained similarly since $b$ and $u$ lie in the same space $\widetilde{L}^{\fr{1}{\al}}_T(\dot{B}^{\fr{N}{p}+2\al-1}_{p,1})$. The proof of Lemma \ref{lem-a4.1} is completed.
\end{proof}

\begin{lem}\label{lem-a4.3}
Under the conditions of Lemma \ref{lem-a4.1}, we have
\be\label{a5.6}
\|I(b)\mathcal{A} u\|_{L^1_T(\dot{B}^{\fr{N}{2}-1}_{2,1})}\le CX(T)Z(T).
\ee
\end{lem}
\begin{proof}
Using Corollary \ref{coro-product}, Lemma \ref{Bernstein},  and \eqref{b2}, we are led to
\beqno
\|{ I(b)}\mathcal{A} u\|_{L^1_T(\dot{B}^{\fr{N}{2}-1}_{2,1})}
\nn&\le&C\|{ I(b)}\|_{L^\infty_T(\dot{B}^{\fr{N}{2}}_{2,1})}\|\mathcal{A} u\|_{L^1_T(\dot{B}^{\fr{N}{2}-1}_{2,1})}\\
\nn&\le&C\|b\|_{L^\infty_T(\dot{B}^{\fr{N}{2}}_{2,1})}\| u\|_{L^1_T(\dot{B}^{\fr{N}{2}+1}_{2,1})}\\
&\le&CX(T)Z(T).
\eeqno
This completes the proof of Lemma \ref{lem-a4.3}.
\end{proof}
Finally,  we estimate $\|K(b)\nb b\|_{L^1_T(\dot{B}^{\fr{N}{2}-1}_{2,1})}$ in a  similar  manner as  Lemma \ref{lem4.4}.
\begin{lem}\label{lem-a4.4}
Under the conditions of Lemma \ref{lem-a4.1},  we have
\be\label{a5.7}
\|K(b)\nb b\|_{L^1_T(\dot{B}^{\fr{N}{2}-1}_{2,1})}\le CX(T)Z(T).
\ee
\end{lem}
\begin{proof}
First of all, we can use  \eqref{e4.32} to bound $\|K(b)\nb b_H\|_{L^1_T(\dot{B}^{\fr{N}{2}-1}_{2,1})}$. In order to bound $\|\dot{T}_{K(b)}\nb b_L\|_{L^1_T(\dot{B}^{\fr{N}{2}-1}_{2,1})}$, using the interpolation inequality,
we get
\beno
\|b_L\|_{\widetilde{L}^{\bar{r}_1}_T(\dot{B}^{\fr{N}{2}}_{\bar{p}_1,1})}\le C\|b_L\|^{\bar{\theta}}_{\widetilde{L}^{\fr{1}{\al}}_T(\dot{B}^{\fr{N}{p}+2\al-1}_{p,1})}\|b_L\|^{1-\bar{\theta}}_{\widetilde{L}^{\fr{1}{1-\al}}_T(\dot{B}^{\fr{N}{2}+1-2\al}_{2,1})}\le C X^{\bar{\theta}}(T)Z^{1-\bar{\theta}}(T),
\eeno
where
 \beno
\bar{\theta}=\fr{(2-4\al)p}{(4+N-8\al)p-2N},
\ \fr{1}{\bar{r}_1}=\bar{\theta}\al+(1-\bar{\theta})(1-\al), \quad \fr{1}{\bar{p}_1}=\fr{\bar{\theta}}{p}+\fr{1-\bar{\theta}}{2},
\eeno
and
\beqno
\|b\|_{\widetilde{L}^{\bar{r}_2}_T(\dot{B}^{\fr{N}{p}}_{\bar{p}_2,1})}\nn&\le& \|b_L\|_{\widetilde{L}^{\bar{r}_2}_T(\dot{B}^{\fr{N}{p}}_{\bar{p}_2,1})}+\|b_H\|_{\widetilde{L}^{\bar{r}_2}_T(\dot{B}^{\fr{N}{p}}_{\bar{p}_2,1})}\\
\nn&\le& C\left(\|b_L\|_{\widetilde{L}^{\bar{r}_2}_T(\dot{B}^{\fr{N}{p}}_{\bar{p}_2,1})}+\|b_H\|_{\widetilde{L}^{\bar{r}_2}_T(\dot{B}^{\frac{N}{p}-\frac{N}{\bar{p}_2}+\frac{N}{2}}_{2,1})}\right)\\
\nn&\le& C\left(\|b_L\|^{1-\bar{\theta}}_{\widetilde{L}^{\fr{1}{\al}}_T(\dot{B}^{\fr{N}{p}+2\al-1}_{p,1})}\|b_L\|^{\bar{\theta}}_{\widetilde{L}^{\fr{1}{1-\al}}_T(\dot{B}^{\fr{N}{2}+1-2\al}_{2,1})}
+\|b_H\|_{\widetilde{L}^{\bar{r}_2}_T(\dot{B}^{\frac{N}{2}}_{2,1})}\right)\\
&\le&CX^{1-\bar{\theta}}(T)Z^{\bar{\theta}}(T),
\eeqno
where
\beno
\fr{1}{\bar{r}_2}=(1-\bar{\theta})\al+\bar{\theta}(1-\al), \quad \fr{1}{\bar{p}_2}=\fr{1-\bar{\theta}}{p}+\fr{\bar{\theta}}{2}.
\eeno
Here
\beno
\fr{1}{\bar{r}_1}+\fr{1}{\bar{r}_2}=1, \quad \fr{1}{\bar{p}_2}=\fr{1}{\bar{p}^*_1}+\fr{1}{p},\quad 2\leq \bar{p}_1,\bar{p}_2\leq p.
\eeno
Then similar to \eqref{e4.35}, we find  that
\beq
\|\dot{T}_{K(b)}\nb b_L\|_{L^1_T(\dot{B}^{\fr{N}{2}-1}_{2,1})}\nn&\le&C\|K(b)\|_{L^{\bar{r}_2}_T(\dot{B}^{0}_{\bar{p}_1^*,1})}\|\nb b_L\|_{L^{\bar{r}_1}_T(\dot{B}^{\fr{N}{2}-1}_{\bar{p}_1,1})}\\
\nn&\le&C\|K(b)\|_{L^{\bar{r}_2}_T(\dot{B}^{\fr{N}{p}}_{\bar{p}_2,1})}\| b_L\|_{L^{\bar{r}_1}_T(\dot{B}^{\fr{N}{2}}_{\bar{p}_1,1})}\\
\nn&\le&C\|b\|_{L^{\bar{r}_2}_T(\dot{B}^{\fr{N}{p}}_{\bar{p}_2,1})}\| b_L\|_{L^{\bar{r}_1}_T(\dot{B}^{\fr{N}{2}}_{\bar{p}_1,1})}\\
&\le&CX(T)Z(T).
\eeq
Next, similar to \eqref{e4.38}, if $N\ge3$, we have
\beq
\|\dot{T}'_{\nb b_L}K(b)_L\|_{L^1_T(\dot{B}^{\fr{N}{2}-1}_{2,1})}\nn&\le&C\|\nb b_L\|_{L^{\fr{1}{\al}}_T(\dot{B}^{2\al-2}_{\infty,1})}\|K(b)_L\|_{L^{\fr{1}{1-\al}}_T(\dot{B}^{\fr{N}{2}+1-2\al}_{2,1})}\\
\nn&\le&C\| b_L\|_{L^{\fr{1}{\al}}_T(\dot{B}^{\fr{N}{p}+2\al-1}_{p,1})}\left(\|b_L\|_{L^{\fr{1}{1-\al}}_T(\dot{B}^{\fr{N}{2}+1-2\al}_{2,1})}+\|b_H\|_{L^{\fr{1}{1-\al}}_T(\dot{B}^{\fr{N}{2}}_{2,1})}\right)\\
&\le&CX(T)Z(T).
\eeq
If $N=2$, the remainder $\dot{R}({\nb b_L}, K(b)_L)$ should be estimated as follows:
\beq
\|\dot{R}({\nb b_L}, K(b)_L)\|_{L^1_T(\dot{B}^{\fr{N}{2}-1}_{2,1})}\nn&\le&C\|\dot{R}({\nb b_L}, K(b)_L)\|_{L^1_T(\dot{B}^{\fr{2N}{p}-1}_{\fr{p}{2},1})}\\
\nn&\le&C\|\nb b_L\|_{L^{\fr{1}{\al}}_T(\dot{B}^{\fr{N}{p}+2\al-2}_{p,1})}\|K(b)_L\|_{L^{\fr{1}{1-\al}}_T(\dot{B}^{\fr{N}{p}+1-2\al}_{p,1})}\\
\nn&\le&C\| b_L\|_{L^{\fr{1}{\al}}_T(\dot{B}^{\fr{N}{p}+2\al-1}_{p,1})}\left(\|b_L\|_{L^{\fr{1}{1-\al}}_T(\dot{B}^{\fr{N}{2}+1-2\al}_{2,1})}+\|b_H\|_{L^{\fr{1}{1-\al}}_T(\dot{B}^{\fr{N}{2}}_{2,1})}\right)\\
&\le&CX(T)Z(T).
\eeq
Finally, thanks to  Lemma \ref{lem-Kb} in the Appendix, we have
\be\label{KbH}
\|K(b)_H\|_{L^{1}_T(\dot{B}^{\fr{N}{2}}_{2,1})}
\le C\left(\|b_L\|_{L^{1}_T(\dot{B}^{\fr{N}{2}+1}_{2,1})}+\|b_H\|_{L^{1}_T(\dot{B}^{\fr{N}{2}}_{2,1})}\right).
\ee
Then using Proposition \ref{p-TR} and \eqref{b2}, we arrive at
\beq
\|\dot{T}'_{\nb b_L}K(b)_H\|_{L^1_T(\dot{B}^{\fr{N}{2}-1}_{2,1})}
\nn&\le&C\|{\nb b_L}\|_{L^{\infty}_T(\dot{B}^{\fr{N}{2}-1}_{2,1})}\|K(b)_H\|_{L^{1}_T(\dot{B}^{\fr{N}{2}}_{2,1})}\\
\nn&\le&C\|{ b_L}\|_{L^{\infty}_T(\dot{B}^{\fr{N}{2}}_{2,1})}\left(\|b_L\|_{L^{1}_T(\dot{B}^{\fr{N}{2}+1}_{2,1})}+\|b_H\|_{L^{1}_T(\dot{B}^{\fr{N}{2}}_{2,1})}\right)\\
&\le&CX(T)Z(T).
\eeq
Combining \eqref{e4.32} with the above  estimates yields \eqref{a5.7}. We complete the proof of Lemma \ref{lem-a4.4}.
\end{proof}

\bigbreak
\noindent{\bf Acknowledgments}
\bigbreak
Research supported by China Postdoctoral Science Foundation funded project 2014M552065, and National Natural Science Foundation of China 11401237, 11671353, 11271017 and 11331005, and  the Fundamental Research Funds for the Central Universities CCNU15A05041, CCNU16A02011. Part of this work was carried out when the third author was visiting the Department of Mathematics of Zhejiang University.

\end{document}